\documentclass{amsart}

\usepackage[mathscr]{eucal}
\usepackage{amssymb}
\usepackage[usenames,dvipsnames]{xcolor} 
\usepackage[normalem]{ulem}
\usepackage{amsthm}
\usepackage{mathtools}
\usepackage{bbold}
\usepackage{mathdots}
\usepackage{float}
\usepackage{enumitem}

\usepackage{adjustbox}
\usepackage{amsmath}
\usepackage{wasysym}
\usepackage{tikz}
\usetikzlibrary{calc}
\usepackage{tikz-cd}
\usetikzlibrary{arrows}
\usetikzlibrary{shapes}
\usetikzlibrary{patterns.meta}
\usetikzlibrary{snakes}

\usepackage[unicode]{hyperref} 

\hypersetup{colorlinks=true,citecolor=Brown,urlcolor=Brown,linkcolor=Brown}

\usepackage{microtype}
\usepackage[nameinlink,capitalise,noabbrev]{cleveref}

\numberwithin{equation}{section}
\swapnumbers 

\newtheorem{Thm}[equation]{Theorem}
\newtheorem*{Thm*}{Theorem}
\newtheorem{Prop}[equation]{Proposition}
\newtheorem{Lem}[equation]{Lemma}
\newtheorem{Cor}[equation]{Corollary}
\theoremstyle{remark}
\newtheorem{Def}[equation]{Definition}
\newtheorem{Ter}[equation]{Terminology}
\newtheorem{Not}[equation]{Notation}
\newtheorem{Exa}[equation]{Example}
\newtheorem{Exas}[equation]{Examples}
\newtheorem{Cons}[equation]{Construction}
\newtheorem{Conv}[equation]{Convention}
\newtheorem{Hyp}[equation]{Hypothesis}

\newtheorem{Rem}[equation]{Remark}

\newcommand{\nc}{\newcommand}
\nc{\dmo}{\DeclareMathOperator}

\renewcommand{\emptyset}{\varnothing}
\DeclarePairedDelimiter\ceil{\lceil}{\rceil}
\DeclarePairedDelimiter\floor{\lfloor}{\rfloor}
\nc{\MonSeq}{\mathsf{MSeq}}
\nc{\AsymSeq}{\mathsf{ASeq}}
\nc{\PowerSeq}{\mathsf{PSeq}}
\nc{\sigmakf}{\sigma^k\hspace{-.25ex}f}
\nc{\sigmakg}{\sigma^k\hspace{-.25ex}g}
\nc{\sigmakh}{\sigma^k\hspace{-.25ex}h}
\nc{\sigmainff}{\sigma^{\inf(f)}\hspace{-.25ex}f}
\dmo{\nil}{nil}
\nc{\PRad}{\mathsf{PRad}}
\nc{\mono}[1]{\widehat{#1}}
\nc{\radideal}[1]{\sqrt{\langle #1 \rangle}}
\nc{\ideal}[1]{\langle #1 \rangle}
\dmo{\thickiddmo}{thickid}
\dmo{\thickdmo}{thick}
\nc{\thicksub}[1]{\thickdmo(#1)}
\nc{\loew}{{\ell\ell}}
\nc{\tame}{\tau}
\nc{\Specplus}{\Spec_+\hspace{-0.1em}}
\nc{\SpecRK}{\Spec(R_{\cat K})}
\nc{\altmathbb}[1]{\mathbbold{#1}}
\nc{\gP}{\altmathbb{g}_{\cat P}}
\nc{\gp}{g_{\frakp}}
\nc{\Pone}{\mathbb{P}^1_k}

\nc{\SpcT}{\Spc(\cat T^c)}
\nc{\SpcS}{\Spc(\cat S^c)}
\nc{\frakm}{\mathfrak m}
\nc{\overbar}[1]{\mkern 1.5mu\overline{\mkern-1.5mu#1\mkern-1.5mu}\mkern 1.5mu}
\nc{\bbullet}{{\scriptscriptstyle\hspace{-1pt}\bullet}}
\nc{\bullett}{{\scriptscriptstyle\bullet}\hspace{-1pt}}
\nc{\LF}{L\hspace{-0.2ex}F}
\nc{\SpG}{\Sp^G}
\nc{\Prst}{{\cat P}\mathrm{r^{st}}}
\nc{\Mack}{\mathcal{M}ack}
\nc{\SC}{S\cat C}
\nc{\OX}{\cat O_{\hspace{-0.2ex}X}}
\nc{\OXX}{\OX\hspace{-0.1ex}(X)}
\dmo{\BLat}{\mathsf{BLat}}
\dmo{\BDLat}{\mathsf{BDLat}}
\dmo{\DLat}{\mathsf{DLat}}
\dmo{\Lat}{\mathsf{Lat}}
\dmo{\Pos}{\mathsf{Pos}}
\dmo{\Spectral}{\mathsf{Spec}}
\dmo{\kosz}{Kosz}
\nc{\Kosz}{\kosz}
\dmo{\Ann}{Ann}
\dmo{\Stab}{Stab}
\dmo{\Aff}{Aff}
\dmo{\Ext}{Ext}
\dmo{\Tor}{Tor}
\dmo{\gen}{gen}
\dmo{\Pic}{Pic}
\dmo{\DM}{DM}
\dmo{\DMT}{DMT}
\dmo{\DMAT}{DMAT}
\nc{\DMQ}{\DM_Q}
\dmo{\DerKal}{DMack}
\dmo{\Der}{D}
\nc{\Dperf}{\Der_\mathrm{perf}}
\dmo{\Dps}{\Der_\mathrm{ps}}
\nc{\Derps}{\Dps}
\nc{\Dpfl}{\Der_{+}^{\mathrm{fl}}}
\nc{\Dzfl}{\Der_{\ge 0}^{\mathrm{fl}}}
\nc{\Dpfg}{\Der_{+}^{\mathrm{fg}}}
\dmo{\Dbfl}{\Der_{b}^{\mathrm{fl}}}
\nc{\Esplit}{E_{\mathrm{split}}}
\nc{\Eeven}{E_{\mathrm{even}}}
\nc{\Eodd}{E_{\mathrm{odd}}}
\nc{\fsplit}{f_{\mathrm{split}}}
\dmo{\Ddash}{\Der^{\mathrm{--}}}
\dmo{\Ddashfl}{\Der^{\mathrm{--}}_{\mathrm{fl}}}
\nc{\Dmfl}{\Ddashfl}
\nc{\Derqc}{\Der_{\mathrm{qc}}}
\dmo{\DMot}{DMot}
\dmo{\rmH}{H}
\dmo{\piu}{\underline{\pi}}
\dmo{\Sphere}{\mathbb{S}}
\nc{\HA}{{\rmH \hspace{-0.2em}\bbA}}
\nc{\HZ}{{\rmH \hspace{-0.2em}\bbZ}}
\nc{\HZp}{{\rmH \hspace{-0.2em}\bbZ_{(p)}}}
\nc{\HZbar}{{\rmH \hspace{-0.2em}\underline{\bbZ}}}
\nc{\Fp}{{\bbF_{\hspace{-0.1em}p}}}
\nc{\HFp}{{\rmH \hspace{-0.15em}\bbF_{\hspace{-0.1em}p}}}
\nc{\DHZpG}{\Der(\HZp_G)}
\nc{\DHZG}{\Der(\HZ_G)}
\nc{\DHZH}{\Der(\HZ_H)}
\nc{\DHZK}{\Der(\HZ_K)}
\nc{\DHZGN}{\Der(\HZ_{G/N})}
\nc{\DHZGG}{\Der(\HZ_{G/G})}
\nc{\DHZCp}{\Der(\HZ_{C_p})}
\nc{\DHZGprime}{\Der(\HZ_{G'})}
\nc{\DHZ}{\Der(\HZ)}
\nc{\frakp}{\mathfrak{p}}
\nc{\frakq}{\mathfrak{q}}
\nc{\Z}{\mathbb{Z}}
\nc{\SSG}{\text{sSet}_*^G}
\nc{\sSet}{\text{sSet}}

\nc{\Loco}[1]{\Loc_{\otimes}\hspace{-0.3ex}\langle #1 \rangle}
\nc{\Coloco}[1]{\Coloc^{\mathrm{Hom}}\hspace{-0.3ex}\langle #1 \rangle}

\dmo{\Con}{Conj}
\dmo{\Sub}{Sub}
\dmo{\Id}{Id}
\dmo{\Loc}{Loc}
\dmo{\Coloc}{Coloc}
\dmo{\rmK}{\textrm{\rm K}}
\dmo{\Spc}{Spc}
\dmo{\cone}{cone}
\dmo{\End}{End}
\dmo{\Mor}{Mor}
\dmo{\Hom}{Hom}
\dmo{\id}{id}
\dmo{\incl}{incl}
\dmo{\Img}{Im}
\dmo{\im}{im}
\dmo{\Ker}{Ker}
\dmo{\ind}{ind}
\dmo{\CoInd}{coind}
\dmo{\res}{res}
\dmo{\infl}{infl}
\dmo{\triv}{triv}
\dmo{\Tel}{Tel} 
\dmo{\Mod}{Mod}%
\dmo{\opname}{op}
\dmo{\SH}{SH}
\dmo{\smallb}{b}
\dmo{\Spec}{Spec}
\dmo{\supp}{supp}
\dmo{\Supp}{Supp}
\dmo{\Cosupp}{Cosupp}
\nc{\SHc}{{\SH^c}}
\nc{\SHp}{{\SH_{(p)}}}
\nc{\SHcp}{{\SH^c_{(p)}}}
\nc{\SHG}{\SH(G)}
\nc{\SHGp}{\SH(G)_{(p)}}
\nc{\SHGc}{\SHG^c}
\nc{\SHGcp}{\SHG^c_{(p)}}
\nc{\quadtext}[1]{\quad\textrm{#1}\quad}
\nc{\qquadtext}[1]{\qquad\textrm{#1}\qquad}
\nc{\adj}{\dashv}
\nc{\adjto}{\rightleftarrows}
\nc{\bbL}{\mathbb{L}}
\nc{\bbA}{\mathbb{A}}
\nc{\bbN}{\mathbb{N}}
\nc{\bbQ}{\mathbb{Q}}
\nc{\bbZ}{\mathbb{Z}}
\nc{\bbF}{\mathbb{F}}
\nc{\bbR}{\mathbb{R}}
\nc{\cat}[1]{\mathscr{#1}}
\nc{\ie}{{\sl i.e.}, }
\nc{\into}{\mathop{\rightarrowtail}}
\nc{\inv}{^{-1}}
\nc{\isoto}{\mathop{\overset{\sim}\to}}
\nc{\isotoo}{\mathop{\overset{\sim}\too}}
\nc{\onto}{\mathop{\twoheadrightarrow}}
\nc{\too}{\mathop{\longrightarrow}\limits}
\nc{\mapstoo}{\longmapsto}
\nc{\adh}[1]{\overline{#1}}
\nc{\adhpt}[1]{\adh{\{#1\}}}
\nc{\aka}{{a.\,k.\,a.}\ }
\nc{\calF}{\mathcal{F}}
\nc{\eg}{{\sl e.\,g.}}
\nc{\Homcat}[1]{\Hom_{\cat #1}}
\nc{\hook}{\hookrightarrow}
\nc{\ihomname}{\mathsf{hom}}
\nc{\ihom}[1]{\mathsf{hom}(#1)}
\nc{\Mid}{\,\big|\,}
\nc{\MMod}{\,\text{-}\Mod}%
\nc{\op}{^{\opname}}
\nc{\oto}[1]{\overset{#1}\to}
\nc{\otoo}[1]{\overset{#1}{\,\too\,}}
\nc{\sminus}{\!\smallsetminus\!}
\nc{\poplus}[1]{^{\oplus #1}}%
\nc{\potimes}[1]{^{\otimes #1}}
\nc{\sbull}{{\scriptscriptstyle\bullet}}
\nc{\SET}[2]{\big\{\,#1\Mid#2\,\big\}}
\nc{\SpcK}{\Spc(\cat K)}
\nc{\then}{\Rightarrow}
\nc{\unit}{\mathbb{1}}
\nc{\unitT}{\unit_{\cat T}}
\nc{\unitS}{\unit_{\cat S}}
\nc{\xra}{\xrightarrow}
\nc{\phigeom}[1]{\widetilde{\Phi}^{#1}}
\nc{\phigeomb}[1]{\Phi^{#1}}
\dmo{\Oname}{O}
\dmo{\proper}{proper}
\dmo{\lenormal}{\unlhd}
\dmo{\lnormal}{\lhd}
\nc{\normal}{\trianglelefteq}
\nc{\Op}{\Oname^p}
\nc{\Oq}{\Oname^q}
\dmo{\Sp}{Sp}
\dmo{\Ho}{Ho}
\dmo{\Fin}{Fin}
\dmo{\add}{add}
\dmo{\Fun}{Fun}
\dmo{\CAlg}{CAlg}
\dmo{\CMon}{CMon}
\dmo{\CC}{\cat C}
\dmo{\DD}{\cat D}
\dmo{\OO}{\mathcal{O}}
\dmo{\Map}{Map}
\dmo{\Span}{Span}
\dmo{\N}{N}
\dmo{\Cat}{Cat}
\dmo{\colim}{colim}
\dmo{\Ch}{Ch}
\dmo{\A}{\mathbb{A}^{eff}}
\nc{\AGeff}{\mathbb{A}_G^{\mathrm{eff}}}
\nc{\BGeff}{\mathcal{B}_G^{\mathrm{eff}}}
\nc{\BG}{{\mathcal{B}_G}}
\nc{\NBGeff}{{\N}{\BGeff}}
\dmo{\Ab}{Ab}
\nc{\Set}{\mathsf{Set}}
\dmo{\ev}{ev}
\dmo{\Spcl}{Spcl}
\nc{\Funadd}{\Fun_{\add}}
\dmo{\proj}{proj}
\dmo{\cof}{cof}

\newcounter{enum-resume-hack}
\Crefname{Thm}{Theorem}{Theorems}
\Crefname{Prop}{Proposition}{Propositions}
\Crefname{thmx}{Theorem}{Theorems}

\begin{document}


\title[The Balmer spectrum of pseudo-coherent complexes]{The Balmer spectrum of pseudo-coherent complexes over a discrete valuation ring}
\author{Beren Sanders}
\author{Yufei Zhang}
\date{\today}

\address{Beren Sanders, Mathematics Department, UC Santa Cruz, 95064 CA, USA}
\email{beren@ucsc.edu}
\urladdr{http://people.ucsc.edu/$\sim$beren/}

\address{Yufei Zhang, Mathematics Department, UC Santa Cruz, 95064 CA, USA}
\email{yzhan544@ucsc.edu}

\maketitle

\begin{abstract}
	We study the derived category of pseudo-coherent complexes over a noetherian commutative ring, building on prior work by Matsui--Takahashi. Our main theorem is a computation of the Balmer spectrum of this category in the case of a discrete valuation ring. We prove that it coincides with the spectral space associated to a bounded distributive lattice of asymptotic equivalence classes of monotonic sequences of natural numbers. The proof of this theorem involves an extensive study of generation behaviour in the derived category of pseudo-coherent complexes. We find that different types of generation are related to different asymptotic boundedness conditions on the growth of torsion in homology. Consequently, we introduce certain distributive lattices of (equivalence classes of) monotonic sequences where the partial orders are defined by different notions of asymptotic boundedness. These lattices, and the spectral spaces corresponding to them via Stone duality, may be of independent interest. The complexity of these spectral spaces shows that, even in the simplest nontrivial case, the spectrum of pseudo-coherent complexes is vastly more complicated than the spectrum of perfect complexes. From a broader perspective, these results demonstrate that the spectrum of a rigid tensor-triangulated category can expand tremendously when we pass to a (non-rigid) tensor-triangulated category which contains it.
\end{abstract}

{
\hypersetup{linkcolor=black}
\tableofcontents
}

\section{Introduction}

A foundational result in tensor triangular geometry is the fact that a quasi-compact and quasi-separated scheme $X$ can be recovered as the Balmer spectrum of its derived category of perfect complexes:
	\begin{equation}\label{eq:recover}
		\Spc(\Dperf(X)) \cong X.
	\end{equation}
In particular, applied to an affine scheme $X=\Spec(R)$, the Balmer spectrum recovers the usual Zariski spectrum: $\Spc(\Dperf(R))\cong \Spec(R)$.

Perfect complexes are precisely the dualizable objects of the larger tensor-triangulated category $\Derqc(X)$ of complexes of $\cat O_X$-modules with quasi-coherent cohomology. However, there is another essentially small tensor-triangulated category which sits between them: the derived category of pseudo-coherent complexes $\Dps(X)$. We have inclusions of tensor-triangulated categories
	\[
		\Dperf(X) \hookrightarrow \Dps(X) \hookrightarrow \Derqc(X).
	\]
Applying the Balmer spectrum to this new category, we obtain a locally ringed space $\Spc(\Dps(X))$ and a surjective morphism
	\begin{equation}\label{eq:intro-comp-X}
		\Spc(\Dps(X)) \twoheadrightarrow \Spc(\Dperf(X))\cong X.
	\end{equation}
It is natural to attempt to understand the space $\Spc(\Dps(X))$ and its relationship with the original scheme $X$.

In this paper we will focus on $X=\Spec(R)$ for $R$ a commutative noetherian ring. In this case, $\Dps(R)\cong \Der_+^\mathrm{fg}(R)$ is simply the derived category of bounded below complexes of finitely generated $R$-modules, and the surjective morphism \eqref{eq:intro-comp-X} coincides with Balmer's comparison map
	\begin{equation}\label{eq:intro-comp}
		\rho:\Spc(\Dps(R)) \to \Spec(R).
	\end{equation}
This has been studied in the work of Matsui and Takahashi \cite{MatsuiTakahashi17}. They prove, among other things, that the map $\rho$ is a bijection if and only if $\Spec(R)$ is discrete (i.e.~$R$ is artinian). Thus, naive dreams that \eqref{eq:intro-comp-X} could have some familiar algebro-geometric interpretation are not long-lasting.

Our primary goal in this paper is to give an in-depth study of the simplest nontrivial case --- that of a discrete valuation ring. This was also considered by Matsui--Takahashi, but we will go further (and also correct an error in their discussion). Our main theorem is a complete computation of $\Spc(\Dps(R))$ in this particular case. In order to state our theorem, we first need to discuss the asymptotic behaviour of sequences of natural numbers.
\[ \ast \ast \ast \]

Given two monotonic sequences $f,g :\bbN \to \bbN$, we say that $f$ is asymptotically bounded by $g$, denoted $f \le g$, if there exists a constant $A$ such that 
	\[
		f(n) \le Ag(n) \text{ for $n \gg 0$.}
	\]
This defines a pre-order on the set of all monotonic sequences $\MonSeq$. This pre-order induces an equivalence relation defined by $f \sim g$ if $f \le g$ and $g\le f$. Moreover, the set of equivalence classes $\AsymSeq \coloneqq \MonSeq/{\sim}$ inherits a partial order from $\le$. If we formally adjoin a top element, we obtain a bounded distributive lattice $\AsymSeq_+$. It corresponds, via Stone duality, to a spectral space~$\Spec(\AsymSeq_+)$.

We will also need two weaker notions of asymptotic boundedness. Consider the scaling operators $\sigma$ and $\mu$ defined by $(\sigma f)(n) \coloneqq f(n+1)$ and $(\mu f)(n) \coloneqq f(2n)$. Then define:
	\begin{itemize}[leftmargin=2em]
		\item $f \le_\sigma g$ if $f \le \sigma^k g$ for some $k \ge 1$.
		\item $f \le_\mu g$ if $f \le \mu^k g$ for some $k \ge 1$.
	\end{itemize}
These also define pre-orders on the set of monotonic sequences and in the same fashion we obtain lattices of ``$\sigma$-equivalence classes'' and ``$\mu$-equivalence classes'' of monotonic sequences. These lattices are quotients of each other:
	\[
		\AsymSeq \twoheadrightarrow \AsymSeq/\sigma \twoheadrightarrow \AsymSeq/\mu.
	\]
In particular, $(\AsymSeq/\mu)_+ = \SET{[f]_\mu}{f:\bbN\to\bbN \text{ monotonic}} \sqcup \{\infty\}$ is the bounded distributive lattice of $\mu$-equivalence classes of monotonic sequences. Our main theorem (\cref{cor:main}) is:

\begin{Thm}\label{thm:intro-main-cor}
	Let $R$ be a discrete valuation ring. We have an isomorphism 
	\[
		\Spc(\Dps(R)) \cong \Spec((\AsymSeq/\mu)_+)^\vee
	\]
	where $(-)^\vee$ denotes the Hochster dual.
\end{Thm}

The connection between asymptotic sequences and pseudo-coherent complexes arises from the asymptotic growth of torsion in homology. For any $R$-module $M$ of finite length, define 
	\[
		\loew(M) \coloneqq \min\{ n\in \bbN \mid \frakm^n M=0\}
	\]
where $\frakm=(x)$ denotes the unique maximal ideal of $R$. It is the maximum $i$ such that~$R/x^i$ is a direct summand of $M$. It is also the Loewy length of $M$, which explains the notation. The largest proper thick ideal of $\Dps(R)$ is given by the pseudo-coherent complexes whose homology modules are all of finite length:
	\[
		\Dpfl(R) \subseteq \Der_+^{\mathrm{fg}}(R)=\Dps(R).
	\]
For a complex $E \in \Dpfl(R)$, we may define $\loew_E:\bbN\to\bbN$ by 
	\[
		\loew_E(n) \coloneqq \loew(H_n(E))
	\]
for each $n \in \bbN$.\footnote{In the main body of the paper we will use a variant definition in which $\loew_E(0)$ is the Loewy length of the first nonzero homology module of $E$, rather than of $H_0(E)$. This difference does not affect the statements of the theorems in the Introduction; see \cref{rem:loew-pref}.} For example, given a sequence $f:\bbN\to\bbN$, we have $\loew_{R/x^f} = f$ for the complex
	\[
		R/x^f \coloneqq \bigoplus_{n \in \bbN} R/x^{f(n)}[n]
	\]
with zero differentials. We establish the following as \cref{thm:rid-full}:

\begin{Thm}\label{thm:intro-radical}
	Let $f$ be a nonzero monotonic sequence. The radical thick ideal generated by $R/x^f$ is given by 
	\[
		\radideal{R/x^f} = \SET{E\in \Dpfl(R)}{\loew_E \le_\mu f}.
	\]
\end{Thm}

The proof of this theorem is quite involved and is the culmination of a series of tools developed in \cref{sec:thick,sec:explodable,sec:convolutions}. The result simplifies in certain cases:

\begin{Thm}
	Let $f$ be a nonzero monotonic sequence. The thick ideal $\ideal{R/x^f}$ is radical if and only if $f$ is $\mu$-stable, meaning that $\mu f\le f$. In this case, we have
	\[
		\ideal{R/x^f} = \SET{E\in \Dpfl(R)}{\loew_E \le f}.
	\]
\end{Thm}

For example, $f(n)\coloneqq n^2$ is $\mu$-stable while $f(n)\coloneqq 2^n$ is not. In any case, we establish in \cref{cor:rid-fullfull} that every principal radical ideal is generated by a complex of the form $R/x^f$ for some monotonic sequence $f$ and it follows from \cref{thm:intro-radical} that~$R/x^f$ and $R/x^g$ generate the same radical ideal if and only if~$f$ and $g$ are \mbox{$\mu$-equivalent,} that is, $[f]_\mu = [g]_\mu$ in $\AsymSeq/\mu$. Moreover, Kock and Pitsch \cite{KockPitsch17} showed that the Balmer spectrum $\Spc(\cat K)$ of a tensor-triangulated category $\cat K$ is the Hochster dual of the spectral space corresponding to the bounded distributive lattice~$\PRad(\cat K)$ of principal radical ideals of $\cat K$. This leads to:

\begin{Thm}\label{thm:intro-lattice}
	Let $R$ be a discrete valuation ring. We have an isomorphism of bounded distributive lattices
	\[
		(\AsymSeq/\mu)_+ \xrightarrow{\sim} \PRad(\Dps(R))
	\]
	given by $[f]_\mu \mapsto\radideal{R/x^f}$ and $\infty \mapsto \Dps(R)=\radideal{R}$.
\end{Thm}

This is \cref{thm:main} and provides \cref{thm:intro-main-cor} as an immediate corollary. According to the isomorphism of \cref{thm:intro-lattice}, the prime ideals of $\Dps(R)$ that are finitely generated (as thick ideals or, equivalently, as radical thick ideals) correspond to the $\mu$-equivalence classes $[f]_\mu$ which are prime elements of the lattice $(\AsymSeq/\mu)_+$. For example, the constant zero sequence $[\underline{0}]_\mu$ and the class of bounded sequences~$[\underline{1}]_\mu$ are prime elements. In summary:

\begin{Thm}\label{thm:intro-three-points}
	Let $R$ be a discrete valuation ring. The spectrum $\Spc(\Dps(R))$ is a local irreducible space. The closed point is the zero ideal $(0)$ and the generic point is~$\Dpfl(R)$. Moreover, the ideal generated by the module $R/x$ is prime and given by
	\[
		\ideal{R/x} = \SET{E\in\Dpfl(R)}{\loew(H_n(E)) \text{ is bounded}}.
	\]
	It is the smallest nonzero prime ideal.
\end{Thm}

The lattice $\AsymSeq/\mu$ of $\mu$-equivalence classes of monotonic sequences is extremely intricate. To gain more explicit information, we can look for a simpler sublattice~${L \subset \AsymSeq/\mu}$ and study the surjective morphism 
	\[
		\Spec((\AsymSeq/\mu)_+)^\vee \twoheadrightarrow \Spec(L_+)^\vee.
	\]
For example, sitting inside $\AsymSeq/\mu$ is a totally ordered sublattice $\PowerSeq \subset \AsymSeq/\mu$ consisting of the ($\mu$-equivalence classes of) power functions $f_\alpha(n) \coloneqq \floor{n^\alpha}$ for each~$\alpha \in \mathbb{R}_{\ge 0}$. The spectrum of this totally ordered lattice can be explicitly described (\cref{rem:pseq-spec}) and we obtain a factorization of the comparison map~\eqref{eq:intro-comp}:

\begin{equation}\label{eq:intro-power-fac}
	\begin{tikzpicture}[baseline=(current bounding box.center)]
	\node (A) at (0,2.5) {$\Spc(\Dps(R))$};
	\node (B) at (3,2.5) {$\Spec(\PowerSeq_+)^\vee$};
	\node (C) at (6,2.5) {$\Spec(R)$};
	\draw[->>] (A) -- (B);
	\draw[->>] (B) -- (C);
	\filldraw[fill=cyan!20,snake=bumps] (0,1.5cm) -- (0.5cm,1.0) -- (0.5cm,0.5cm) -- (0.0,0.0) -- (-0.5,0.5) -- (-0.5,1.0) -- (0,1.5);
	\node (b) at (0,-.025cm) {$\bullet$};
	\node (b) at (0,1.525cm) {$\bullet$};
	\node (b) at (0,2.0cm) {$\bullet$};
	\draw (0,2.0cm) -- (0,1.5cm);
	\draw [pattern color=red,thin,pattern={Lines[angle=170,distance=4pt]}] (3,0.75) ellipse (0.375cm and 0.75cm);
	\draw [semithick] (3,0.75) ellipse (0.375cm and 0.75cm);
	\node (b) at (3,0cm) {$\bullet$};
	\node (b) at (3,1.5cm) {$\bullet$};
	\node (b) at (3,2.0cm) {$\bullet$};
	\draw (3,2.0cm) -- (3,1.5cm);
	\draw[snake=brace] (3.8,1.5)--(3.8,0.0);
	\draw[|->] (4.2,0.75) -- (5.6,0.75);
	\node (b) at (6,0.75cm) {$\bullet$};
	\node (b) at (6,2.0cm) {$\bullet$};
	\draw (6,2.0cm) -- (6,0.75cm);
\end{tikzpicture}
\end{equation}

This shows that there is at least a continuum of prime ideals in $\Dps(R)$, but we have only explicitly described three of them, as in \cref{thm:intro-three-points} above. In fact, we will prove (\cref{thm:not-prime-mu}):

\begin{Thm}
	Let $f$ be an unbounded monotonic sequence. Its $\mu$-equivalence class~$[f]_\mu$ is not a prime element in $(\AsymSeq/\mu)_+$.
\end{Thm}

To prove that an \emph{arbitrary} unbounded sequence is not prime is a difficult task, which demands the construction of some bizzare monotonic sequences, detailed in~\cref{sec:technical}.

\begin{Cor}
	There are at least $2^{\aleph_0}$ prime ideals in $\Dps(R)$ but only two of these prime ideals are finitely generated.
\end{Cor}

This frustrates attempts to give simple explicit descriptions of the prime ideals in the cyan region of \eqref{eq:intro-power-fac}. It also establishes:

\begin{Cor}
	For any $d \ge 1$, the thick ideal
	\[
		\SET{E\in \Dpfl(R)}{ \loew(H_n(E)) \text{ is asymptotically bounded by the polynomial $n^d$}}
	\]
	is not prime.
\end{Cor}

We single this out because it contradicts \cite[Theorem~F]{MatsuiTakahashi17}. Only two of the thick ideals claimed to be prime in that theorem are actually prime. We explain the error in \cref{rem:MT-error}. Nevertheless, it was Matsui and Takahashi's very interesting paper which inspired our interest in these questions, and our work most certainly depends on, and builds upon, their work.
	\[ \ast\ast\ast \]
We also give a treatment of $\Spc(\Dps(R))$ for a general commutative noetherian ring, again building on the work of \cite{MatsuiTakahashi17}. A key construction is a splitting
	\[
		\tame:\Spec(R) \hookrightarrow \Spc(\Dps(R))
	\] 
of the comparison map \eqref{eq:intro-comp}. The primes $\tame(\frakp)$ in the image of this map are the \emph{tame primes} of $\Dps(R)$. In general, Matsui--Takahashi proved that each fiber~$\rho^{-1}(\{\frakp\})$ is irreducible with the tame prime $\tau(\frakp)$ serving as the generic point. In \cref{sec:tame,sec:base-change}, we prove various general results about tame primes from a tensor-triangular perspective and consider how $\Spc(\Dps(R))$ behaves when we vary the ring. Among other results, we prove that the functor $\Dps(R)\to\Dps(R_\frakp)$ is a Verdier quotient (\cref{prop:localization-verdier}). It follows that $\Dps(R_\frakp)$ is the local category of $\Dps(R)$ at the tame prime $\tame(\frakp)$:
	\[
		\Dps(R)/\tame(\frakp) \cong \Dps(R_\frakp).
	\]
We also explain how the difficulty in understanding $\Dps(R)$ compared with $\Dperf(R)$ is that $\Dps(R)$ does not interact well with algebraic localization and that an analogue of the Neeman--Thomason theorem for $\Dperf(R)$ fails strongly for $\Dps(R)$. Ultimately, we hope that our geometric exposition of these ideas will provide a strong foundation for future studies of the tensor triangular geometry of pseudo-coherent complexes.

From a broader perspective, $\Dps(R)$ is a prototypical example of a tensor-triangulated category which is not rigid, meaning that not all of its objects are dualizable. It is for this reason that we must distinguish between the thick ideal~$\ideal{E}$ and the radical thick ideal $\radideal{E}$ generated by an object $E \in \Dps(R)$. Our example is particularly interesting because we find that the difference between $\ideal{E}$ and $\radideal{E}$ amounts to the difference between the asymptotic bound $f \le_\sigma g$ and the weaker asymptotic bound $f \le_\mu g$. Our results also demonstrate that passing from the rigid tensor-triangulated category~$\Dperf(R)$ to the slightly larger non-rigid tensor-triangulated category $\Dps(R)$ enlarges the spectrum to an extreme degree.

\subsection*{Outline of the paper:}
We briefly review Stone duality in \cref{sec:lattices} and then introduce several bounded distributive lattices of asymptotic equivalence classes of monotonic sequences in \cref{sec:asymptotic}. To our knowledge, these lattices --- and their associated spectral spaces --- have not been considered in the literature, and may be of independent interest. In \cref{sec:pseudo} we turn to commutative algebra and recall the basics concerning the derived category of pseudo-coherent complexes over a commutative noetherian ring. We study its Balmer spectrum in \cref{sec:tame} and state several results obtained by Matsui--Takahashi \cite{MatsuiTakahashi17} from a slightly different perspective. In \cref{sec:base-change} we study the behaviour of $\Spc(\Dps(R))$ as we vary the ring~$R$. In particular, we consider the case of closed immersions (\cref{cor:img-base-change}) and localizations (\cref{prop:localization-verdier}). These results require a general tt-geometric surjectivity theorem (\cref{prop:img-of-Spc}) which may be of independent interest.

\Cref{sec:DVR,sec:thick,sec:explodable,sec:convolutions,sec:the-spectrum,sec:complexity,sec:technical} are devoted to the case of a discrete valuation ring. In \cref{sec:DVR} we consider the basic features of $\Dps(R)$ in this case, and introduce some constructions necessary to understand generation behaviour in this category, including the complex~$R/x^f$ associated to a sequence of natural numbers $f$  and the Loewy sequence~$\loew_E$ of a complex~$E$ (\cref{def:loewy-seq}). Then in \cref{sec:thick} we develop methods to control the complexes lying in the thick ideal generated by $R/x^f$. We find that they are those complexes whose Loewy sequence is asymptotically $\sigma$-bounded by~$f$;  see \cref{thm:loew-thick} and  \cref{cor:thickRxf}. The full story, however, is subtle and requires studying the notion of an ``explodable complex'' introduced in \cref{sec:explodable}; see \cref{thm:f-explodable}. By studying such complexes and relating them to the notions of $\sigma$-stability and \mbox{$\mu$-stability} from \cref{sec:asymptotic}, we are able to answer certain questions raised in \cite{MatsuiTakahashi17}. In particular, we show that two conditions considered in their work, which we dub (MT1) and (MT2), are actually equivalent for a complex $R/x^f$; in fact they are both equivalent to $\mu$-stability. See \cref{thm:MT1-MT2-mustable} and \cref{cor:summary-ideal-Rxf}. This provides a negative answer to \cite[Question~7.4]{MatsuiTakahashi17}; see \cref{exa:not-MT1}.

In \cref{sec:convolutions} we study convolutions as a way to get a closer understanding of \emph{radical} ideals in $\Dps(R)$. A key insight is that pseudo-coherent complexes are ``explodable up to tensor-powers''; see \cref{prop:almost-explodable}. This leads to \cref{thm:rid-full} which describes the radical ideal generated by $R/x^f$. Armed with the above, we prove our main result \cref{thm:main} in \cref{sec:the-spectrum}. In \cref{sec:complexity} we study the complexity of the lattice of $\mu$-equivalence classes of monotonic sequences and introduce the totally orered sublattice of ``power sequences'', also known as ``exponential polynomial functions''. We describe the spectral space associated to this totally ordered lattice and use this to gain information about $\Spc(\Dps(R))$. Finally, in \cref{sec:technical} we provide some technical constructions which prove that the lattice of $\mu$-equivalence classes of asymptotic sequences has only two prime elements (\cref{thm:not-prime-mu}). In particular, $\Dps(R)$ has only two finitely generated prime ideals (\cref{cor:new-not-prime-mu} and \cref{cor:two-fg}). This also proves that a claim made in \cite{MatsuiTakahashi17} concerning the prime ideals in $\Spc(\Dps(R))$ is false. We explain the error in \cref{rem:MT-error}.

\section{Lattices and Stone duality}\label{sec:lattices}

We begin with a review of Stone duality for bounded distributive lattices. We do not give an exhaustive treatment and rather direct the reader to \cite{Johnstone82}, \cite[Chapter~3]{DickmannSchwartzTressl19} or the summary in \cite[Section~3]{BCHNPS_descent} for further details.

\begin{Def}
	A partially ordered set $(A,\le)$ is a \emph{lattice} if every two element subset $\{a,b\} \subseteq A$ admits a join (or least upper bound) $a \vee b$ and a meet (or greatest lower bound) $a \wedge b$. A morphism of lattices $f:A \to B$ is a function satisfying $f(a \vee b) = f(a) \vee f(b)$ and $f(a \wedge b) = f(a) \wedge f(b)$ for all $a,b\in A$. A morphism of lattices is order-preserving (since $a \le b$ if and only if $a \vee b = b$) but the converse is not true; see \cite[Remark~3.3]{BCHNPS_descent}. In other words, the category of lattices $\Lat$ is a non-full subcategory of the category of partially ordered sets $\Pos$.
\end{Def}

\begin{Rem}
	If $f:A\to B$ is a bijective lattice morphism, then $f^{-1}:B \to A$ is also a lattice morphism. It follows that the forgetful functors $\Lat \to \Pos \to \Set$ reflect isomorphisms. In particular, an isomorphism of lattices is the same thing as a bijective morphism of lattices and this is the same thing as an order-isomorphism of partially ordered sets (i.e.~a bijection satisfying $a \le b$ if and only if $f(a) \le f(b)$).
\end{Rem}

\begin{Def}
	A lattice is \emph{bounded} if it contains a least element $0$ and a greatest element $1$. A morphism of bounded lattices $f:A \to B$ is morphism of lattices which also satisfies $f(0)=0$ and $f(1)=1$. The bounded lattices thus form a non-full subcategory $\BLat \subset \Lat$ of the category of all lattices.
\end{Def}

\begin{Def}
	A lattice $(A,\le)$ is said to be \emph{distributive} if 
	\[
		a \wedge (b \vee c) = (a\wedge b) \vee (a \wedge c)
	\]
	for all $a,b,c \in A$. There are a number of equivalent characterizations; see \cite[Lemma~10]{Graetzer78}. The bounded distributive lattices form a full subcategory $\BDLat \subset \BLat$ of the category of bounded lattices.
\end{Def}

\begin{Exa}\label{exa:N}
	The non-negative natural numbers $\mathbb{N}\coloneqq \mathbb{N}_{\ge 0}$ with its natural ordering is a distributive lattice in which $a \wedge b = \min(a,b)$ and $a \vee b = \max(a,b)$.
\end{Exa}

\begin{Rem}
	The functors $\BDLat \to \BLat \to \Lat$ evidently reflect isomorphisms. In particular, a morphism of bounded distributive lattices is an isomorphism if and only if it is bijective.
\end{Rem}

\begin{Def}
	An \emph{ideal} of a bounded distributive lattice $(A,\le)$ is a nonempty subset $I\subseteq A$ such that
	\begin{enumerate}
		\item if $a \in I$ and $b \in A$ satisfy $b \le a$, then $b \in I$;
		\item for all $a,b\in I$, we have $a \vee b \in I$.
	\end{enumerate}
	Moreover $I$ is \emph{prime} if in addition it is a proper subset of $A$ and satisfies
	\begin{enumerate}[resume]
		\item if $a\wedge b \in I$ then either $a \in I$ or $b \in I$.
	\end{enumerate}
\end{Def}

\begin{Exa}\label{exa:prime-element}
	Let $A$ be a bounded distributive lattice. For each element $a \in A$, we have the principal ideal $a^{\downarrow}\coloneqq \SET{b\in A}{b\le a}$. It is a prime ideal if and only if the element $a$ is \emph{prime}: $a\neq 1$ and $b\wedge c \le a$ implies $b \le a$ or $c \le a$.
\end{Exa}

\begin{Def}
	A topological space is a \emph{spectral space} if it is quasi-compact, the quasi-compact open subsets are closed under finite intersection and form a basis for the topology, and every irreducible closed subset admits a unique generic point; see \cite{DickmannSchwartzTressl19}. A \emph{spectral map} of spectral spaces is a continuous function $f:X\to Y$ with the property that the preimage of any quasi-compact open subset is again quasi-compact. We denote the category of spectral spaces and spectral maps by~$\Spectral$.
\end{Def}

\begin{Exa}
	Let $\Spec(A)$ denote the set of prime ideals of a bounded distributive lattice $A$. The sets $U(a) \coloneqq \SET{P \in \Spec(A)}{a \not\in P}$ for each $a \in A$ are the quasi-compact open sets for a spectral topology on $\Spec(A)$. Endowed with this topology, the spectral space $\Spec(A)$ is called the \emph{spectrum} of $A$. Moreover, for a morphism $f:A \to B$ of bounded distributive lattices, we have an induced map $\Spec(f):\Spec(B)\to\Spec(A)$ which sends a prime ideal $P$ to its preimage $f^{-1}(P)$ and which evidently satisfies $\Spec(f)^{-1}(U(a)) = U(f(a))$ for each $a \in A$. We thus obtain a contravariant functor $\Spec:\BDLat\op\to \Spectral$.
\end{Exa}

\begin{Rem}\label{rem:stone-duality}
	Stone duality establishes an anti-equivalence between the category of bounded distributive lattices and the category of spectral spaces. More precisely, there is an adjoint equivalence of categories
	\[
		\Omega:\Spectral \adjto \BDLat\op:\Spec
	\]
	where the functor $\Omega$ sends a spectral space $X$ to the bounded distributive lattice~$\Omega(X)$ consisting of its quasi-compact open subsets ordered by inclusion.
\end{Rem}

\begin{Rem}
	Let $\frakp$ and $\frakq$ be points in a spectral space $X$. We write $\frakp \rightsquigarrow \frakq$ and say that $\frakq$ is a \emph{specialization} of $\frakp$ (and that $\frakp$ is a \emph{generalization} of $\frakq$) if $\frakq \in \overbar{\{\frakp\}}$. Thus, $\overbar{\{\frakp\}}$ is the set of specializations of $\frakp$ and $\gen(\frakp) \coloneqq \SET{\frakq}{\frakp \in \overbar{\{\frakq\}}}$ is the set of generalizations of $\frakp$. For the spectrum $X=\Spec(A)$ of a bounded distributive lattice, $\frakp \rightsquigarrow \frakq$ is equivalent to an inclusion of prime ideals $\frakp \subseteq \frakq$. In particular, the closed points are the maximal ideals of the lattice $A$; see \cite[2.4]{Johnstone82}.
\end{Rem}

\begin{Rem}[Hochster duality]\label{rem:hochster-duality}
	By definition, the quasi-compact open sets form a basis for the topology of a spectral space $X$. The complements of the quasi-compact open sets are called the \emph{Thomason closed sets} and form a basis for another spectral topology on $X$ called the \emph{Hochster dual} topology.\footnote{The Hochster dual topology is called the ``inverse topology'' in \cite{DickmannSchwartzTressl19}. Our terminology refers to \cite{Hochster69} and \cite{Thomason97}.} We write $X^\vee$ for the set $X$ equipped with the Hochster dual topology. The open sets of $X^\vee$ --- that is, the arbitrary unions of Thomason closed sets --- are the so-called \emph{Thomason sets} of $X$. One can readily check that $X^{\vee\vee}=X$. This construction is far more transparent in the category of bounded distributive lattices. Under Stone duality, taking the Hochster dual amounts to replacing the order of a bounded distributive lattice $(A,\le)$ with its opposite order $(A,\le_{\mathrm{op}})$.
\end{Rem}

\begin{Exa}\label{exa:balmer-spectrum}
	Let $\cat K$ be an essentially small tensor-triangulated category. The collection of principal radical thick tensor ideals $\sqrt{\langle a \rangle}$ forms a bounded distributive lattice $\PRad(\cat K)$ under inclusion. The join and meet are given by the direct sum and tensor product of generating objects, respectively. As explained in \cite{KockPitsch17}, the Balmer spectrum of $\cat K$ is the Hochster dual of the associated spectral space: $\Spc(\cat K) \cong \Spec(\PRad(\cat K))^{\vee}$.
\end{Exa}

\begin{Rem}
	We have $\frakp \rightsquigarrow \frakq$ in $X$ if and only if $\frakq \rightsquigarrow \frakp$ in $X^{\vee}$. In other words, $X$ and $X^{\vee}$ have the opposite specialization orders. In particular, the closed points of~$X$ are the generic points of $X^{\vee}$ and vice versa.
\end{Rem}

\begin{Rem}
	The topology of a spectral space $X$ is completely determined by its specialization order together with its associated \emph{constructible topology}; see \cite[Section~1.5]{DickmannSchwartzTressl19} for details. We suffice ourselves with recalling that the so-called \emph{constructible sets} are the subsets which are both closed and open for the constructible topology, and the quasi-compact open sets are precisely the constructible sets which are generalization closed. (Equivalently, the Thomason closed sets are precisely the constructible sets which are specialization closed.)
\end{Rem}

\begin{Exa}\label{exa:totally-ordered}
	Let $(A,\le)$ be a bounded totally ordered set. Then $(A,\le)$ is a bounded distributive lattice with $a \vee b = \max\{a,b\}$ and $a \wedge b = \min\{a,b\}$. Its spectrum $\Spec(A)$ is described in detail in  \cite[Section~3.6]{DickmannSchwartzTressl19}. The prime ideals coincide with the proper nonempty down-sets. For each $a \in A$ we have an associated quasi-compact open set $U(a) = \SET{P \in \Spec(L)}{ P \cap a^{\uparrow} = \emptyset}$ and Thomason closed set $V(a) = \SET{P \in \Spec(L)}{ a^{\downarrow} \subseteq P}$. These are all the quasi-compact open sets and Thomason closed sets. In particular, the constructible sets are the finite unions of sets of the form $V(a) \cap U(b)$.
\end{Exa}

In our applications, we will naturally come across lattices which have a least element but not a greatest element.

\begin{Cons}\label{cons:top}
	Let $A$ be a lattice which has a least element 0. We define $A_+ \coloneqq A \sqcup \{\infty\}$ and extend the partial order in the unique way so that $\infty$ is the greatest element. Then $A_+$ is a bounded lattice. Moreover, if $A$ is distributive then $A_+$ will be a bounded distributive lattice. This construction provides a left adjoint to the forgetful functor $\BDLat \to \DLat_0$ where $\DLat_0$ denotes the category of distributive lattices which have least elements $0$ and lattice morphisms which preserve them.
\end{Cons}

\begin{Not}
	When notationally convenient, we will write $\Specplus(A)\coloneqq \Spec(A_+)$.
\end{Not}

\begin{Exa}
	The lattice $\bbN\coloneqq \bbN_{\ge 0}$ has no greatest element but we can consider $\Specplus(\bbN)$. According to \cref{exa:totally-ordered}, the prime ideals are 
	\[
		\Specplus(\bbN) = [0] \rightsquigarrow [0,1] \rightsquigarrow [0,2] \rightsquigarrow [0,3] \rightsquigarrow \cdots \rightsquigarrow \bbN
	\]
	with $\bbN$ the unique closed point and $[0]$ the unique generic point. Note that the prime ideal $\bbN$ is not principal, but all other prime ideals are principal and all elements of~$\bbN$ are prime. The homotopy theorist may recognize that it is homeomorphic to the Balmer spectrum of the $p$-local stable homotopy category: $\Specplus(\bbN) \cong \Spc(\SH_{(p)}^c)$.
\end{Exa}

We now turn to monomorphisms and epimorphisms. 

\begin{Rem}\label{rem:congruence-relation}
	An equivalence relation $\sim$ on a lattice $(A,\le)$ is a \emph{congruence relation} if $a_0 \sim b_0$ and $a_1 \sim b_1$ implies $a_0 \vee a_1 \sim b_0 \vee b_1$ and $a_0 \wedge a_1 \sim b_0\wedge b_1$. The set of equivalence classes $A/{\sim}$ then inherits the structure of a lattice and the quotient map $A \to A/{\sim}$ is a surjective lattice morphism. Moreover, every surjective lattice morphism arises in this way from a congruence relation. See \cite[pp.~20--22]{Graetzer78} for further details. We will call the lattice~$A/{\sim}$ a \emph{quotient lattice} of $A$. It is immediate that if $A$ is (bounded) distributive then $A/{\sim}$ is also (bounded) distributive.
\end{Rem}

\begin{Rem}\label{rem:BDLat-surjective}
	Since $\BDLat$ is an algebraic category (in the sense that it is monadic over $\Set$) surjective morphisms are the same thing as regular epimorphisms (i.e.~coequalizers) and these are the same as the extremal epimorphisms. Thus if $A\to B$ is a surjective morphism in $\BDLat$ then the corresponding spectral map $\Spec(B) \to \Spec(A)$ is a regular monomorphism (equivalently, an extremal monomorphism) in~$\Spectral$ and hence is a topological embedding; cf.~\cite[Theorem~5.4.3]{DickmannSchwartzTressl19}. In other words, a quotient $A\to A/{\sim}$ of a bounded distributive lattice induces a topological embedding $\Spec(A/{\sim})\hookrightarrow \Spec(A)$.
\end{Rem}

\begin{Rem}
	On the other hand, $A \to B$ is an epimorphism in $\BDLat$ if and only if $\Spec(B)\to\Spec(A)$ is a monomorphism of spectral spaces which is the same thing as being an injective spectral map. Finally, $A \to B$ is a monomorphism in $\BDLat$ if and only if $A\to B$ is injective if and only if $\Spec(B)\to\Spec(A)$ is surjective if and only if $\Spec(B)\to\Spec(A)$ is an epimorphism of spectral spaces.
\end{Rem}

\begin{Exa}\label{exa:BDLat-injective}
	If $A$ is a distributive lattice with a least element $0$ and $B \subset A$ is a sublattice containing $0$ then $B_+ \hookrightarrow A_+$ is a monomorphism of bounded distributive lattices and $\Specplus(A)\to\Specplus(B)$ is a surjective spectral map.
\end{Exa}

The following observation will be used several times in the next section.

\begin{Rem}\label{rem:pre-order-induce}
	Let $(X,\le)$ be a pre-ordered set, i.e.~a set equipped with a reflexive and transitive relation. There is an induced equivalence relation on $X$ defined by $a \sim b$ if $a \le b$ and $b \le a$. Moreover, the set of equivalence classes $X/{\sim}$ inherits a partial order given by $[a]\le [b]$ if $a \le b$. In this way, every pre-ordered set $(X,\le)$ gives rise to a partially ordered set $(X/{\sim},\le)$. Indeed, this construction provides a left adjoint to the fully faithful inclusion of partially ordered sets in the category of pre-ordered sets.
\end{Rem}

\section{Spectral spaces of asymptotic sequences}\label{sec:asymptotic}

We now introduce some distributive lattices arising from the asymptotic growth behaviour of sequences of natural numbers. 

\begin{Not}
	Let $\bbN \coloneqq \bbN_{\ge 0}$ denote the set of natural numbers.
\end{Not}

\begin{Def}
	A function $f:\bbN \to \bbN$ is \emph{monotonic} if $f(n) \le f(n+1)$ for all $n \in \bbN$.
\end{Def}

\begin{Ter}
	A \emph{monotonic sequence} is a monotonic function $\bbN \to \bbN$.
\end{Ter}

\begin{Exa}\label{exa:to-mono}
	An arbitrary function $f:\bbN \to \bbN$ gives rise to a monotonic sequence~$\mono{f}$ given by
	\[
		\mono{f}(n) \coloneqq \max_{0 \le i \le n}f(i).
	\]
	Note that $f$ is monotonic if and only if $f=\mono{f}$.
\end{Exa}

\begin{Rem}\label{rem:pointwise-lattice}
	The collection of all monotonic sequences $\MonSeq$ is a distributive lattice with the pointwise ordering inherited from the target lattice $\bbN$ (\cref{exa:N}). It has a least element, the constant function $\underline{0}$, but does not have a greatest element. The join and meet are given pointwise.
\end{Rem}

\begin{Def}\label{def:asymptotic-ordering}
	We define a relation on the set of monotonic sequences by writing $f \le g$ if there exists a constant $A$ and an $n_0 \in \bbN$ such that $f(n) \le Ag(n)$ for all $n \ge n_0$. Note that this is saying $f=O(g)$ in ``big-Oh'' notation \cite[\S 89]{Hardy52}. We say that $f$ is \emph{asymptotically bounded} by $g$.
\end{Def}

\begin{Rem}
	Being reflexive and transitive, $\le$ is a pre-order on the set of monotonic sequences. It thus induces an equivalence relation defined by $f \sim g$ if $f \le g$ and $g \le f$ (cf.~\cref{rem:pre-order-induce}). In this case, we say $f$ and $g$ are \emph{asymptotically equivalent}.
\end{Rem}

\begin{Def}
	Let $\AsymSeq\coloneqq \MonSeq/{\sim}$ denote the set of asymptotic equivalence classes of monotonic sequences. It inherits a well-defined partial order given by $[f] \le [g]$ if $f \le g$.
\end{Def}

\begin{Ter}\label{ter:asymptotic}
	The elements of $\AsymSeq$ are the \emph{asymptotic monotonic sequences}. Although an asymptotic monotonic sequence is really an asymptotic equivalence class of monotonic sequences, we will often abuse notation and simple write $f$ for the equivalence class $[f] \in \AsymSeq$. In such circumstances, the constructions and definitions given do not depend on the choice of representative.
\end{Ter}

\begin{Prop}
	The asymptotic monotonic sequences $(\AsymSeq,\le)$ form a distributive lattice which has a least element given by the constant sequence $\underline{0}$.
\end{Prop}

\begin{proof}
	One readily checks that the equivalence relation $\sim$ on $\MonSeq$ is a congruence relation (\cref{rem:congruence-relation}) for the pointwise lattice structure of \cref{rem:pointwise-lattice}. In other words, the map $\MonSeq \to \AsymSeq$ is a surjective lattice homomorphism. Hence, the distributivity of $\MonSeq$ implies the distributivity of $\AsymSeq$.
\end{proof}

\begin{Rem}
	We can thus consider the bounded distributive lattice $\AsymSeq_+$ obtained by adjoining a top element (\cref{cons:top}). There is an associated spectral space $\Specplus(\AsymSeq) \coloneqq \Spec(\AsymSeq_+)$ by Stone duality.
\end{Rem}

\begin{Not}
	We write $\underline{n}$ for the constant monotonic sequence with value $n$.
\end{Not}

\begin{Exa}
	A monotonic sequence $f$ satisfies $[f]=[\underline{0}]$ in $\AsymSeq$ if and only if $f=\underline{0}$ pointwise. As already mentioned, $[\underline{0}]$ is the bottom element of $\AsymSeq$.
\end{Exa}

\begin{Exa}\label{exa:bounded-seq}
	A monotonic sequence $f$ is a \emph{bounded sequence} if there is a constant~$A$ such that $f(n) \le A$ for all $n \in \bbN$. All nonzero bounded sequences are asymptotically equivalent and provide a single asymptotic monotonic sequence $[\underline{1}] \in \AsymSeq$. In fact, the class $[\underline{1}]$ of nonzero bounded sequences is bounded by all nonzero asymptotic monotonic sequences: $[\underline{1}] \le [f]$ for all $[f] \neq 0$.
\end{Exa}

\begin{Exa}\label{exa:polynomial}
	Let $f(n)=\sum_{i=0}^k a_i n^i$ and $g(n)=\sum_{i=0}^l b_i n^i$ be nonzero polynomials with nonnegative integer coefficients. Then $f \le g$ if and only if $\deg(f) \le \deg(g)$. In particular, $f$ and $g$ are asymptotically equivalent if and only if $\deg(f)=\deg(g)$. For each $d \ge 1$ we have the unbounded asymptotic monotonic sequence $[n\mapsto n^d] \in \AsymSeq$.
\end{Exa}

\begin{Exa}
	The monotonic sequences $f(n) = 2^{2^{2n+(-1)^n}}$ and $g(n)=2^{2^{2n}}$ are not comparable, meaning $f\not\leq g$ and $g \not\leq f$. This readily follows from the observation that
	\[
		f(n) = \begin{cases}
			\sqrt{g(n)} & \text{if $n$ is odd} \\
			(g(n))^2 & \text{if $n$ is even}.
		\end{cases}
	\]
	This example shows that the lattice $\AsymSeq$ is not totally ordered.
\end{Exa}

Our next goal is to introduce certain quotients of the lattice of asymptotic monotonic sequences.

\begin{Def}
	Let $f$ be a monotonic sequence. We define monotonic sequences $\sigma f$ and $\mu f$ by setting $(\sigma f)(n) \coloneqq f(n+1)$ and $(\mu f)(n)\coloneqq f(2n)$ for each $n \in \bbN$.
\end{Def}

\begin{Rem}\label{rem:tame-lattice-hom}
	It is immediate from the definitions that $\sigma$ and $\mu$ both define lattice endomorphisms of the lattice $\MonSeq$ of monotonic sequences. It is also immediate from the definitions that $[f] \mapsto [\sigma f]$ and $[f] \mapsto [\mu f]$ are well-defined lattice endomorphisms of the lattice $\AsymSeq$ of asymptotic monotonic sequences.
\end{Rem}

\begin{Def}\label{def:sigma-mu-preorders}
	We define two pre-orders $\le_{\sigma}$ and $\le_{\mu}$ on $\MonSeq$ as follows:
	\begin{enumerate}
		\item $f \le_{\sigma} g$ if $f \le \sigmakg$ for some positive integer $k$, and
		\item $f \le_{\mu} g$ if $f \le \mu^k g$ for some positive integer $k$.
	\end{enumerate}
	We also consider the induced equivalence relations (\cref{rem:pre-order-induce}) defined by:
	\begin{enumerate}
		\item $f$ is \emph{$\sigma$-equivalent} to $g$, denoted $f \sim_\sigma g$, if $f \le_{\sigma} g$ and $g \le_\sigma f$, and
		\item $f$ is \emph{$\mu$-equivalent} to $g$, denoted $f \sim_{\mu} g$, if $f \le_{\mu} g$ and $g \le_\mu f$.
	\end{enumerate}
\end{Def}

\begin{Rem}\label{rem:mu-sigma}
	Note that $f \le \sigma f \le \mu f$. Hence $f \le g \Rightarrow f \le_{\sigma} g\Rightarrow f \le_\mu g$ and $f \sim g \Rightarrow f \sim_\sigma g \Rightarrow f \sim_\mu g$. These observations imply that the relations $f \sim_{\sigma} g$ and $f \sim_{\mu} g$ only depend on the asymptotic equivalence classes of $f$ and $g$, and similarly for the relations $f \le_{\sigma} g$ and $f \le_{\mu} g$. In other words, we have well-defined pre-orders on $\AsymSeq$ given by $[f] \le_\sigma [g]$ if $f \le_\sigma g$ and $[f] \le_\mu [g]$ if $f \le_\mu g$. Moreover, the equivalence relations $\sim_\sigma$ and $\sim_\mu$ on $\MonSeq$ induce equivalence relations on $\AsymSeq$ which coincide with those induced by these pre-orders.
\end{Rem}

\begin{Not}\label{not:sigma-mu}
	We write 
	\[
		\AsymSeq/\sigma \coloneqq \AsymSeq/{\sim}_\sigma = \MonSeq/{\sim}_\sigma
		\quad\text{ and }\quad
		\AsymSeq/\mu \coloneqq \AsymSeq/{\sim}_\mu = \MonSeq/{\sim}_\mu
	\]
	for the set of $\sigma$-equivalence classes (respectively, $\mu$-equivalence classes) of asymptotic sequences.
\end{Not}

\begin{Rem}
	It readily follows from \cref{rem:tame-lattice-hom} that the equivalence relations $\sim_\sigma$ and $\sim_\mu$ are congruence relations on $\MonSeq$ with respect to the pointwise ordering (\cref{rem:pointwise-lattice}). This implies that each of the surjective maps in
	\[
		\MonSeq\twoheadrightarrow \AsymSeq \twoheadrightarrow \AsymSeq/\sigma \twoheadrightarrow \AsymSeq/\mu
	\]
	are lattice morphisms. In particular, each of these lattices is a distributive lattice with least element (\cref{rem:congruence-relation}).
\end{Rem}

\begin{Rem}
	From this point on, we will mostly be working in $\AsymSeq$ rather than in~$\MonSeq$ bearing in mind \cref{ter:asymptotic}. We emphasize that $f\le g$ refers to the asymptotic ordering of \cref{def:asymptotic-ordering}. If we wish to indicate that two monotonic sequences satisfy $f(n) \le g(n)$ for all $n \in \bbN$ we will say that ``$f \le g$ pointwise''. As explained above, we have well-defined pre-orders $\le_{\sigma}$ and $\le_\mu$ on $\AsymSeq$ and the quotient lattices $\AsymSeq/\sigma$ and $\AsymSeq/\mu$ coincide with the induced partially ordered sets as in \cref{rem:pre-order-induce}.
\end{Rem}

\begin{Exa}
	The monotonic sequences $f(n)= 2^n$ and $g(n)=4^n$ are $\mu$-equivalent but not $\sigma$-equivalent.
\end{Exa}

\begin{Exa}
	The monotonic sequences $f(n)=n!$ and $g(n)=(n+1)!$ are \mbox{$\sigma$-equivalent} but not asymptotically equivalent.
\end{Exa}

\begin{Exa}
	For two polynomials $f$ and $g$ as in \cref{exa:polynomial}, we have $f \le_\mu g$ if and only if $f \le_\sigma g$ if and only if $f \le g$ if and only if $\deg(f)\le \deg(g)$. This reflects a special property enjoyed by polynomials, which we now isolate and study.
\end{Exa}

\begin{Def}\label{def:sigma-mu-stable}
	An asymptotic monotonic sequence $f$ is \emph{$\sigma$-stable} if $\sigma f \le f$ (equivalently, if $f \sim \sigma f$) and is \emph{$\mu$-stable} if $\mu f \le f$ (equivalently, if $f \sim \mu f$). Every $\mu$-stable sequence is $\sigma$-stable.
\end{Def}

\noindent\begin{minipage}{\linewidth}
\begin{Exas}\label{exa:stable}\hspace{1em}
	\begin{enumerate}
		\item The zero sequence $[\underline{0}]$ and the bounded sequence $[\underline{1}]$ are $\mu$-stable.
		\item The polynomial $f(n)= n^d$ is $\mu$-stable for each $d \ge 0$.
		\item The exponential $f(n)= 2^n$ is $\sigma$-stable but not $\mu$-stable.
		\item The factorial $f(n)= n!$ is neither $\mu$-stable nor $\sigma$-stable.
	\end{enumerate}
\end{Exas}
\end{minipage}

\smallskip
\begin{Rem}
	Recall that $\sigma$ and $\mu$ induce endomorphisms of the lattice $\AsymSeq$ by $\sigma[f] = [\sigma f]$ and $\mu[f]=[\mu f]$. The classes of $\sigma$-stable and $\mu$-stable sequences are the fixed points of these endomorphisms. Thus, we write $\AsymSeq^\sigma \subset \AsymSeq$ and $\AsymSeq^\mu \subseteq \AsymSeq$ for the subsets of $\sigma$-stable and $\mu$-stable sequences, respectively. It is immediate that $\AsymSeq^\mu \subseteq \AsymSeq^\sigma \subseteq \AsymSeq$ are sublattices which contain the bottom element.
\end{Rem}

\begin{Lem}\label{lem:sigma-auto}
	The endomorphism $\sigma:\AsymSeq\to\AsymSeq$ is an automorphism.
\end{Lem}

\begin{proof}
	We have already observed in \cref{rem:mu-sigma} that $f \le g$ implies $\sigma f \le \sigma g$. The converse is also immediate from the definitions. Thus, $\sigma f \sim \sigma g$ implies $f \sim g$. In other words, $\sigma:\AsymSeq \to \AsymSeq$ is injective. On the other hand, every monotonic sequence $f$ is of the form $\sigma g$ where $g$ is the monotonic sequence defined by $g(0) \coloneqq 0$ and $g(n) \coloneqq f(n-1)$ for $n \ge 1$.
\end{proof}

\begin{Rem}\label{rem:mu-not-injective}
	In contrast, the endomorphism $\mu:\AsymSeq \to \AsymSeq$ is not injective. For example, if $f$ is the monotonic sequence defined by
	\[
		f(n) \coloneqq \begin{cases}
		n! & \text{if $n$ is even}\\
		(n+1)! & \text{if $n$ is odd}
		\end{cases}
	\]
	and $g$ is defined by $g(n)\coloneqq n!$ then $f \not\le g$ even though $\mu f=\mu g$ pointwise.
\end{Rem}

\begin{Rem}\label{rem:mu-to-sigma}
	On the other hand, one can readily check from the definitions that $\mu f \le \mu g$ implies that $f \le \sigma g$. The counterexample in \cref{rem:mu-not-injective} (in which $f \not \le g$) is related to the fact that $g(n)\coloneqq n!$ is not $\sigma$-stable. More generally, one can verify that  $\mu^k f \le \mu^k g$ implies $f \le \sigma^{2^k-1} g$ using the fact that $2^k\ceil{n/2^k} \le n+2^k-1$.
\end{Rem}

\begin{Lem}\label{lem:tame-stable}
	Let $f,g \in \AsymSeq$ be two asymptotic monotonic sequences and let $\tame \in \{\sigma,\mu\}$. If either $f$ or $g$ is $\tame$-stable then $f \le g$ if and only if $f \le_\tame g$.
\end{Lem}

\begin{proof}
	The $(\Rightarrow)$ direction always holds by \cref{rem:mu-sigma}. Conversely, suppose $f \le_\tame g$. By definition, this means $f \le \tame^k g$ for some $k \ge 1$. If $g$ is $\tame$-stable then $\tame g \le g$ so that $\tame^k g \le g$ by induction and hence $f \le g$. On the other hand, if $f$ is $\tame$-stable then $\tame^k f \le f$ so that $\tame^k f \le \tame^k g$. If $\tame =\sigma$ this implies $f \le g$ by \cref{lem:sigma-auto}. If $\tame = \mu$ then it implies $f \le \sigma^{2^k-1} g$ by \cref{rem:mu-to-sigma}. However, the $\mu$-stable sequence $f$ is in particular $\sigma$-stable so $\sigma^{2^k-1} f \le f$, hence $\sigma^{2^k-1} f \le \sigma^{2^k-1} g$ and we can invoke \cref{lem:sigma-auto} again.
\end{proof}

\begin{Rem}\label{rem:sublattices}
	It follows from \cref{lem:tame-stable} that the composites $\AsymSeq^{\sigma} \hookrightarrow \AsymSeq \twoheadrightarrow \AsymSeq/\sigma$ and $\AsymSeq^{\mu} \hookrightarrow \AsymSeq \twoheadrightarrow \AsymSeq/\mu$ are injective. We may thus regard the lattices $\AsymSeq^\sigma$ and $\AsymSeq^\mu$ as sublattices of $\AsymSeq/\sigma$ and $\AsymSeq/\mu$, respectively.
\end{Rem}

\begin{Rem}
	We thus have a diagram of spectral spaces
	\[\begin{tikzcd}
		\Specplus(\AsymSeq/\mu) \ar[r,hook] \ar[d,two heads] &  \Specplus(\AsymSeq/\sigma) \ar[r,hook]\ar[d,two heads] &\Specplus(\AsymSeq)\\
		\Specplus(\AsymSeq^\mu) &\Specplus(\AsymSeq^\sigma)\ar[l,two heads]&
	\end{tikzcd}\]
	where the hooked arrows are embeddings and the double-headed arrows are surjective; cf.~\cref{rem:BDLat-surjective} and \cref{exa:BDLat-injective}. We will return to these spaces in \cref{sec:the-spectrum}.
\end{Rem}

\section{Pseudo-coherent complexes}\label{sec:pseudo}

We now turn to commutative algebra and introduce the derived category of pseudo-coherent complexes.

\begin{Hyp}
	Throughout this section, $R$ denotes a commutative noetherian ring.
\end{Hyp}

\begin{Conv}
	We will use homological indexing for complexes. Thus $\Der_+(R) = \SET{X \in \Der(R)}{H_i(X) = 0 \text{ for } i \ll 0}$ is the full replete subcategory of complexes which are homologically bounded on the right. The suspension $\Sigma X = X[1]$ shifts complexes to the left: $(X[1])_n = X_{n-1}$.
\end{Conv}

\begin{Rem}
	A complex of $R$-modules is said to be a \emph{perfect complex} if it is quasi-isomorphic to a bounded complex of finitely generated projective modules. The unbounded derived category $\Der(R)$ is a rigidly-compactly generated tensor-triangulated (``tt'') category whose compact (=dualizable) objects are precisely the perfect complexes. In particular, the perfect complexes $\Dperf(R) = \Der(R)^c=\Der(R)^d$ form an essentially small tensor-triangulated subcategory of~$\Der(R)$.
\end{Rem}

\begin{Rem}
	The notion of a \emph{pseudo-coherent complex} was introduced by Illusie in \cite{SGA6}. The simplest characterization is that a complex $X \in \Der(R)$ is pseudo-coherent if it is quasi-isomorphic to a bounded below complex of finitely generated projective modules. We will write $\Dps(R) \subset \Der(R)$ for the full replete subcategory of pseudo-coherent complexes. It is an essentially small tensor-triangulated subcategory; see \cite[Section~064N]{stacks-project}. Since our ring $R$ is noetherian, a complex~$X$ is pseudo-coherent if and only if it is quasi-isomorphic to a bounded below complex of finitely generated modules and this is the case, moreover, if and only if it is homologically bounded on the right and each homology module is finitely generated:
	\[
		\Dps(R) = \Dpfg(R).
	\]
	See \cite[\href{https://stacks.math.columbia.edu/tag/0FDB}{Proposition~0FDB}]{stacks-project}, for instance. Note that a module is a pseudo-coherent complex if and only if it is finitely generated.
\end{Rem}

\begin{Rem}
	We cannot consider the Balmer spectrum of the bounded derived category $\Der_b(R\text{-mod})\subset \Dps(R)$ because it is not always a tensor category:
\end{Rem}

\begin{Prop}
	The following are equivalent:
	\begin{enumerate}
		\item $\Der_b(R\text{-}\mathrm{mod})$ is a tensor subcategory of $\Der(R)$;
		\item The inclusion $\Dperf(R) = \Der_b(R\text{-}\mathrm{proj}) \hookrightarrow \Der_b(R\text{-}\mathrm{mod})$ is an equivalence;
		\item $R$ is regular.
	\end{enumerate}
\end{Prop}

\begin{proof}
	The equivalence $(b) \Leftrightarrow (c)$ is well-known. Moreover, $(b) \Rightarrow (a)$ is immediate since~$\Dperf(R)$ is a tensor subcategory. We prove $(a) \Rightarrow (c)$. Observe that if~$\Der_b(R\text{-mod})$ is closed under the derived tensor product then for any maximal ideal~$\frakm \subset R$ we would have $H_i(R/\frakm \otimes_R^{\mathbb L} R/\frakm) = 0$ for $i \gg 0$. Passing to the local ring at $\frakm$, this implies that $\Tor_i^{R_{\frakm}}(\kappa(\frakm),\kappa(\frakm)) =0$ for $i\gg 0$ which implies that~$R_{\frakm}$ is regular. See \cite[\S\S 5F--5G]{Lam99}, for example.
\end{proof}

\begin{Rem}
	The tt-category $\Dps(R)$ is never rigid: it always contains non-dualizable objects. Indeed, $\Dperf(R) \subset \Dps(R)$ is precisely the subcategory of dualizable objects. Thus, all objects of $\Dps(R)$ are dualizable if and only if $\Dperf(R)=\Dps(R)$. This is the case if and only if $R=0$ is the zero ring.
\end{Rem}

\begin{Not}
	Let $\cat E \subseteq \Dps(R)$ be a collection of objects. Since $\Dps(R)$ is not rigid, we must distinguish between the thick tensor ideal $\ideal{\cat E}$ generated by $\cat E$ and the radical thick tensor ideal $\sqrt{\langle \cat E\rangle} = \SET{X \in \Dps(R)}{X^{\otimes n} \in \langle \cat E \rangle \text{ for some } n \ge 1}$ generated by $\cat E$. On the other hand, the thick subcategory generated by $\cat E$ will be denoted $\thicksub{\cat E}$.
\end{Not}

\begin{Prop}
	The inclusion $\Dperf(R) \hookrightarrow \Dps(R)$ does not have a left adjoint nor a right adjoint, except when $R=0$.
\end{Prop}

\begin{proof}
	If $P \in \Dperf(R)$ is a perfect complex then $\Hom_{\Der(R)}(R[n],P) = 0$ for ${|n| \gg 0}$. Applying this to the dual of $P$, we also see that $\Hom_{\Der(R)}(P,R[n]) = 0$ for ${|n| \gg 0}$. Note that the complex with zero differentials $\bigoplus_{i \ge 0} R[i] = \prod_{i \ge 0} R[i] \in \Der(R)$ is pseudo-coherent. Hence, if $G:\Dps(R) \to \Dperf(R)$ is right adjoint to the inclusion, it would follow that $\Hom_{\Der(R)}(R[n],\bigoplus_{i \ge 0} R[i]) = 0$ for $n \gg 0$, which is a contradiction. Similarly, if $G$ is left adjoint to the inclusion, it would follow that $\Hom_{\Der(R)}(\prod_{i \ge 0} R[i],R[n])=0$ for $n \gg 0$, which is a contradiction.
\end{proof}

\begin{Rem}\label{rem:koszul-trick}
	It is well-known that if $X$ and $Y$ belong to $\Dps(R)$ and $X \otimes_R^{\bbL} Y \neq 0$ then $\inf(X \otimes_R^{\bbL} Y) \ge \inf(X) + \inf(Y)$ and 
	\[
		H_{\inf(X)+\inf(Y)}(X \otimes_R^{\bbL} Y) = H_{\inf(X)}(X) \otimes_R H_{\inf(Y)}(Y).
	\]
	See \cite[Lemma~2.1]{Foxby77} or \cite[Remark 4.1]{BalmerSanders24pp}, for example. Recall that $\inf(E) \coloneqq \inf\SET{n\in \bbZ}{H_n(E)\neq 0}$ for any nonzero complex $E \in \Der(R)$.
\end{Rem}

\section{Tame primes}\label{sec:tame}

We now make some general observations about the Balmer spectrum of the derived category of pseudo-coherent complexes $\Dps(R)$. Several of these results were proved by Matsui--Takahashi \cite{MatsuiTakahashi17} using a slightly different perspective. We will take for granted some familiarity with tensor triangular geometry; \mbox{see, e.g.,~\cite[\S\S 1--2]{BarthelHeardSanders23a}.} Throughout, $R$ denotes a commutative noetherian ring.

\begin{Rem}\label{rem:comparison-map}
	The fully faithful inclusion $\Dperf(R) \subseteq \Dps(R)$ induces a spectral map 
	\begin{equation}\label{eq:induced-map}
		\Spc(\Dps(R)) \twoheadrightarrow \Spc(\Dperf(R))
	\end{equation}
	which is surjective by \cite[Theorem~1.3]{Balmer18}. Moreover, under the identification $\Spc(\Dperf(R))\cong\Spec(R)$, the map \eqref{eq:induced-map} identifies with the comparison map
	\begin{equation}\label{eq:comparison-map}
		\rho:\Spc(\Dps(R))\to \Spec(R)
	\end{equation}
	defined by $\rho(\cat P) \coloneqq \SET{r \in R}{\cone(r) \not\in \cat P}$. This follows immediately from the naturality of the comparison map \cite[Corollary~5.6]{Balmer10b}. Throughout, we will make the above identification and speak of the comparison map \eqref{eq:comparison-map} rather than \eqref{eq:induced-map}.
\end{Rem}

\begin{Not}
	We will write $\supp(X) \subseteq \Spc(\Dps(R))$ for the universal support of a pseudo-coherent complex $X \in \Dps(R)$ and we will write $\Supp(Y) \subseteq\Spec(R)$ for the Balmer--Favi support of an arbitrary complex $Y \in \Der(R)$ using the identification $\Spc(\Dperf(R))\cong\Spec(R)$.
\end{Not}

\begin{Rem}
	Since $R$ is noetherian, the derived category $\Der(R)$ is stratified in the sense of \cite{BarthelHeardSanders23a}. It follows that the Balmer--Favi support of an arbitrary complex~$X \in \Der(R)$ can be expressed in several different ways. In particular, we have
	\begin{align*}
		\Supp(X) &= \SET{\frakp \in \Spec(R)}{\gp \otimes X \neq 0}\\
				 &= \SET{\frakp \in \Spec(R)}{\kappa(\frakp) \otimes X \neq 0}
	\end{align*}
	where $\gp$ denotes the associated Balmer--Favi idempotent and $\kappa(\frakp)$ denotes the residue field; see \cite[Remark~5.1 and Theorem~4.7]{BarthelHeardSanders23b}. Moreover, it follows that this support theory satisfies the detection property
	\begin{equation}\label{eq:detection}
		\Supp(X) = \emptyset \text{ implies } X=0
	\end{equation}
	and the tensor product property
	\begin{equation}\label{eq:tensor-product}
		\Supp(X \otimes Y) = \Supp(X) \cap \Supp(Y)
	\end{equation}
	for any two $X,Y\in\Der(R)$.
\end{Rem}

\begin{Rem}\label{rem:no-nilpotents}
	It follows from \eqref{eq:detection} and \eqref{eq:tensor-product} that $\Der(R)$ has no nonzero tensor-nilpotent objects. In particular, the same is true of $\Derps(R)$.
\end{Rem}

\begin{Lem}
	Let $R$ be a local noetherian ring with maximal ideal $\frakm$. If $X \in \Der(R)$ is pseudo-coherent then $X \otimes_R^{\mathbb L}R/\frakm = 0$ implies $X=0$.
\end{Lem}

\begin{proof}
	Suppose $X \otimes_R^{\mathbb L}R/\frakm=0$. If $X$ is nonzero then, since it is homologically bounded on the right, we can consider the smallest $i$ such that $H_i(X) \neq 0$. By \cref{rem:koszul-trick}, $0=H_i(X\otimes_R^{\bbL} R/\frakm) = H_i(X) \otimes_R R/\frakm$. Nakayama's Lemma then implies that $H_i(X) = 0$ which is a contradiction.
\end{proof}

\begin{Rem}\label{rem:specialization-closed}
	It follows that for any pseudo-coherent complex $X$, we have
	\begin{align*}
		\Supp(X) &= \SET{\frakp \in \Spec(R)}{X_{\frakp} \neq 0 \text{ in } \Dps(R_{\frakp})} \\
		&= \bigcup_{i \in \mathbb{Z}} \Supp_R(H_i(X))
	\end{align*}
	where $\Supp_R$ denotes the classical support of an $R$-module. Since each homology module $H_i(X)$ is finitely generated, this is a union of closed sets. In other words, the Balmer--Favi support $\Supp(X)$ of a pseudo-coherent complex $X$ is specialization closed.
\end{Rem}

\begin{Rem}\label{rem:thomason}
	Since $R$ is noetherian, the space $\Spec(R)$ is topologically noetherian. Hence every closed subset of $\Spec(R)$ is Thomason closed and every specialization closed subset is Thomason (\cref{rem:hochster-duality}).
\end{Rem}

\begin{Cons}\label{cons:tame}
	Restricting the support function $\Supp(-)$ to $\Dps(R)$, we have a notion of support for pseudo-coherent complexes which satisfies all the axioms for a support datum in the sense of \cite{Balmer05a} except that it need not give closed subsets; recall~\eqref{eq:detection} and \eqref{eq:tensor-product}. The proof of \cite[Theorem~3.2]{Balmer05a} shows that there exists a unique function
	\begin{equation}
		\tame:\Spec(R) \to \Spc(\Dps(R))
	\end{equation}
	such that
	\begin{equation}\label{eq:tame-preimage}
		\tame^{-1}(\supp(X)) = \Supp(X)
	\end{equation}
	for every $X \in \Dps(R)$. Explicitly, 
	\begin{equation}
		\tame(\frakp) \coloneqq \SET{ X \in \Dps(R)}{\frakp \not\in \Supp(X)} \in \Spc(\Dps(R)).
	\end{equation}
	This function need not be continuous, but it has the property that the preimage of any Thomason closed subset is Thomason (\cref{rem:specialization-closed} and \cref{rem:thomason}). One readily checks that this function satisfies:
	\begin{equation}\label{eq:tame-specialization}
		\frakp \subseteq \frakq  \Longleftrightarrow \tame(\frakp) \supseteq \tame(\frakq).
	\end{equation}
	It also follows from the definitions that $\rho(\tame(\frakp)) = \frakp$. Thus the map $\tame$ is injective.
\end{Cons}

\begin{Def}\label{def:tame}
	Following \cite{MatsuiTakahashi17}, we call the primes of $\Spc(\Dps(R))$ lying in the image of $\tame:\Spec(R)\hookrightarrow \Spc(\Dps(R))$ the \emph{tame primes}.
\end{Def}

\begin{Rem}
	Our functions $\tame$ and $\rho$ coincide with the functions denoted $\mathcal S$ and~$\mathfrak s$ in~\cite{MatsuiTakahashi17}. Moreover, the Balmer--Favi support $\Supp(X)$ of a pseudo-coherent complex $X$ coincides with the support denoted $\Supp_R(X)$ in \cite{MatsuiTakahashi17}.
\end{Rem}

\begin{Not}
	For an ideal $I \subseteq R$, we will write $\kosz(I) \in \Dperf(R)$ for the Koszul complex on a set of generators of $I$. Although the isomorphism type of this complex depends on the choice of generators, any two such Koszul complexes generate the same thick subcategory and this is sufficient uniqueness for our purposes.
\end{Not}

\begin{Prop}[Matsui--Takahashi] \label{prop:MT}
	Let $X \in \Dps(R)$ and $\frakp \in \Spec(R)$. If $\frakp \in \Supp(X)$ then $\kosz(\frakp) \in \langle X \rangle$.
\end{Prop}

\begin{proof}
	This is a special case of \cite[Proposition~2.9]{MatsuiTakahashi17}.
\end{proof}

\begin{Cor}\label{cor:key-lemma}
	For any $\cat P \in \Spc(\Dps(R))$, we have $\cat P \subseteq \tame(\rho(\cat P))$.
\end{Cor}

\begin{proof}
	Consider any $X \in \cat P$. We have $\Kosz(\rho(\cat P)) \not\in \cat P$ by the definition of $\rho$. Hence $\Kosz(\rho(\cat P)) \not\in \langle X \rangle$. \Cref{prop:MT} then implies $\rho(\cat P) \not\in \Supp(X) = \tame^{-1}(\supp(X))$. It follows that $X \in \tame(\rho(\cat P))$.
\end{proof}

\begin{Rem}
	The above is the key result about tame primes. It implies that each fiber $\rho^{-1}(\{ \frakp\})$ is irreducible with the tame prime $\tame(\frakp)$ serving as the generic point.
\end{Rem}

\begin{Prop}\label{prop:splitting}
	Consider the splitting
	\[\begin{tikzcd}
		\Spec(R) \ar[rr,bend left=20,"\id"] \ar[r,hook,"\tame"] & \Spc(\Dps(R)) \ar[r,two heads,"\rho"] & \Spec(R).
	\end{tikzcd}\]
	\begin{enumerate}
		\item We have $\tame^{-1}(U) = \rho(U)$ for any generalization closed set $U \subseteq \Spc(\Dps(R))$.
		\item The map $\tame:\Spec(R)\hookrightarrow \Spc(\Dps(R))$ is a topological embedding with respect to the Hochster dual topologies.
		\item The map $\rho:\Spc(\Dps(R)) \twoheadrightarrow \Spec(R)$ satisfies the going-down property.%
	\end{enumerate}
\end{Prop}

\begin{proof}
	$(a)$: The inclusion $\tame^{-1}(U) \subseteq \rho(U)$ holds for any subset $U\subseteq \Spc(\Dps(R))$ simply because $\rho\circ\tame=\id$. The reverse inclusion holds when $U$ is generalization closed by \cref{cor:key-lemma}.

	$(b)$: From \eqref{eq:tame-preimage} and \cref{rem:specialization-closed}, the preimage $\tame^{-1}(V)$ of a Thomason closed subset is specialization closed and hence Thomason by \cref{rem:thomason}. The Thomason closed sets form a basis of \emph{open} sets for the Hochster dual topologies (\cref{rem:hochster-duality}). Thus, $\tame$ is continuous with respect to the Hochster dual topologies. On the other hand, $\rho$ is a spectral map, hence is continuous with respect to the Hochster dual topologies. Thus, with respect to the Hochster dual topologies, $\tame$ is a split monomorphism and hence an embedding.

	$(c)$: For a surjective spectral map, the going-down property is equivalent to being a closed quotient map with respect to the Hochster dual topologies and this is also equivalent to merely being a closed map with respect to the Hochster dual topologies; see \cite[Corollary~5.3.4 and Corollary~6.4.14]{DickmannSchwartzTressl19}. Moreover, since the image of a proconstructible set under a spectral map is proconstructible, it suffices to show that for every dual-closed set $V$, $\rho(V)$ is closed under specialization in the Hochster dual topology. Thus, it suffices to show that $\rho(V)$ is generalization closed if $V$ is generalization closed. In other words, by part $(a)$, the claim is that $\tame^{-1}(V)$ is generalization closed if $V$ is generalization closed, which is true since $\tame$ preserves specializations; see \eqref{eq:tame-specialization}.
\end{proof}

\begin{Prop}[Matsui--Takahashi]
	The map $\tame:\Spec(R)\to \Spc(\Dps(R))$ is continuous if and only if $\Spec(R)$ is finite.
\end{Prop}

\begin{proof}
	This is established in \cite[Theorem~4.7]{MatsuiTakahashi17}. The ``if'' direction is immediate from \cref{rem:specialization-closed} since if $\Spec(R)$ is finite then every specialization closed subset is closed.
\end{proof}

\begin{Prop}
	If $\Spec(R)$ is finite then the following hold:
	\begin{enumerate}
		\item The map $\tame:\Spec(R)\hookrightarrow \Spc(\Dps(R))$ is a topological embedding.
		\item The map $\rho:\Spc(\Dps(R))\twoheadrightarrow \Spec(R)$ is an open quotient map.
	\end{enumerate}
\end{Prop}

\begin{proof}
	Part $(a)$ is immediate since $\tame$ is a split monomorphism in the category of spectral spaces, hence in the category of topological spaces. For part $(b)$, suppose $U\subseteq \Spec(R)$ is a subset such that $\rho^{-1}(U)$ is open. Then $U=\rho(\rho^{-1}(U))=\tame^{-1}(\rho^{-1}(U))$ is open by part $(a)$ of \cref{prop:splitting}.
\end{proof}

\begin{Thm}[Matsui--Takahashi]\label{thm:MT-artinian}
	The following are equivalent:
	\begin{enumerate}
		\item $\rho:\Spc(\Dps(R))\to\Spec(R)$ is a bijection;
		\item $\rho:\Spc(\Dps(R))\to\Spec(R)$ is a homeomorphism;
		\item $\Spec(R)$ is (finite and) discrete.
		\item $R$ is artinian.
	\end{enumerate}
\end{Thm}

\begin{proof}
	It is well-known that $\Spec(R)$ is discrete if and only if $R$ is artinian. The rest is established in \cite[Theorem~6.5]{MatsuiTakahashi17}.
\end{proof}

\begin{Rem}
	One of the goals of this paper is to show that $\Spc(\Dps(R))$ can be considerably more complicated than $\Spec(R)$ in the non-artinian case.
\end{Rem}

\begin{Prop} \label{prop:local-or-domain}
	The following statements hold:
	\begin{enumerate}
		\item If $R$ is local with unique closed point $\frakm$ then $\Spc(\Dps(R))$ is local with unique closed point $\tame(\frakm)$.
		\item If $R$ is a domain with generic point $\eta$ then $\Spc(\Dps(R))$ is irreducible with generic point $\tau(\eta)$.
	\end{enumerate}
\end{Prop}

\begin{proof}
	$(a)$: Recall from \cref{rem:no-nilpotents} that there are no nonzero tensor-nilpotent objects in $\Dps(R)$. Thus, the claim is that $(0)$ is a prime thick ideal in $\Dps(R)$. Well, by construction $\tau(\frakm) = \SET{X \in \Dps(R)}{\frakm \not\in \Supp(X)}$ which is the zero ideal by \cref{rem:specialization-closed} and \eqref{eq:detection}.

	$(b)$: This is proved in \cite[Theorem~3.6]{Matsui19}. In any case, if $\eta$ is a generic point of $\Spec(R)$ then $\tame(\eta)$ is a generic point of $\Spc(\Dps(R))$ by \eqref{eq:tame-specialization} and \cref{cor:key-lemma}.
\end{proof}

\begin{Prop}\label{prop:contains-tame}
	Every nonempty Thomason subset of $\Spc(\Dps(R))$ contains a tame prime.
\end{Prop}

\begin{proof}
	Since every Thomason subset is a union of Thomason closed subsets, it suffices to prove that every nonempty Thomason closed subset contains a tame prime. Each Thomason closed subset is of the form $\supp(X)$ for some $X \in \Dps(R)$. If $\supp(X)$ contains no tame primes then $\Supp(X)=\tame^{-1}(\supp(X))=\emptyset$. But the Balmer--Favi support has the detection property \eqref{eq:detection}, so this implies $X=0$ and hence $\supp(X)=\emptyset$.
\end{proof}

\begin{Prop}\label{prop:tame-closed}
	For every closed point $\frakm \in \Spec(R)$, the tame prime $\tame(\frakm)$ is a closed point of $\Spc(\Dps(R))$ and the fiber $\rho^{-1}(\{\frakm\}) = \{\tame(\frakm)\}$ is a singleton.
\end{Prop}

\begin{proof}
	Let $\cat P \in \Spc(\Dps(R))$ be a prime ideal such that $\rho(\cat P)=\frakm = (f_1,\ldots,f_r)$. Then $\cone(f_i) \not\in \cat P$ for each $i$ by the definition of $\rho$. Moreover, 
	\[
		\Supp(\cone(f_1) \otimes \cdots \otimes \cone(f_r)) = \{\frakm\}.
	\]
	We claim that $\tame(\frakm) \subseteq \cat P$. Otherwise, there exists an $X \in \tame(\frakm)$ with $X \not\in \cat P$. By definition, $X \in \tame(\frakm)$ means that $\frakm \not\in\Supp(X)$. Thus, by \eqref{eq:detection} and \eqref{eq:tensor-product}, we have

	\newsavebox{\tempdiagramm}
	\begin{lrbox}{\tempdiagramm}
		\(
		\begin{aligned}
		\Supp(X \otimes \cone(f_1) \otimes \cdots \otimes \cone(f_r)) &= \Supp(X) \cap \Supp(\cone(f_1) \otimes\cdots\otimes \cone(f_r)) \\
																	  &= \Supp(X) \cap \{\frakm\}\\
																	  &= \emptyset
\end{aligned}
\)
	\end{lrbox}

	\bigskip
	\noindent\resizebox{\linewidth}{!}{\usebox{\tempdiagramm}}
	\smallskip

\noindent so that $X \otimes \cone(f_1) \otimes \cdots \otimes \cone(f_r) = 0 \in \cat P$. Thus $X \in \cat P$ or $\cone(f_i) \in \cat P$ for some~$i$, which is a contradiction. This establishes that $\tame(\frakm) \subseteq\cat P$ and hence $\cat P = \tame(\frakm)$ by \cref{cor:key-lemma}. Also, $\tame(\frakm)$ is closed since if $\tame(\frakm) \rightsquigarrow \cat Q$ then $\frakm = \rho(\tame(\frakm)) \rightsquigarrow \rho(\cat Q)$ so that $\rho(\cat Q)=\frakm$. Hence $\cat Q=\tame(\frakm)$ by what we just proved.
\end{proof}

\begin{Rem}\label{rem:supp-of-perfect}
	For a perfect complex $P$, we have
	\[
		\supp(P) = \rho^{-1}(\Supp(P))
	\]
	since $\Supp(P)$ identifies with the universal support of $P$ in $\Spec(R) \cong \Spc(\Dperf(R))$. On the other hand, \cref{cor:key-lemma} implies that we have an inclusion
	\[
		\supp(X) \supseteq \rho^{-1}(\Supp(X))
	\]
	for any pseudo-coherent complex $X$. This is not usually an equality. If it were an equality for all $X \in \Dps(R)$, then every prime $\cat P \in \Spc(\Dps(R))$ would be tame (and hence $R$ would be artinian by \cref{thm:MT-artinian}). Indeed, we would have
	\[
		\overbar{\{\cat P\}} = \bigcap_{X \not\in \cat P} \supp(X) = \rho^{-1}(\tame^{-1}(\bigcap_{X\not\in\cat P} \supp(X))) = \rho^{-1}(\tame^{-1}(\overbar{\{\cat P\}}))
	\]
	so that $\tame(\rho(\cat P)) \in \overbar{\{\cat P\}}$. Hence $\cat P=\tame(\rho(\cat P))$ by \cref{cor:key-lemma}.
\end{Rem}

\section{Base-change}\label{sec:base-change}

We now consider how $\Dps(-)$ behaves as we vary the ring.

\begin{Rem}\label{rem:base-change}
	Any ring homomorphism $A\to B$ induces a geometric functor $f^*:\Der(A) \to \Der(B)$ which restricts to a functor
	\[
		\Dps(A) \to \Dps(B).
	\]
	In general, the right adjoint $f_*:\Der(B)\to\Der(A)$ does not preserves pseudo-coherent complexes, but it does when $B$ is finitely generated as an $A$-module; that is, when $B=f_*(\unit)$ is itself pseudo-coherent as an $A$-module; see  \cite[\href{https://stacks.math.columbia.edu/tag/064Z}{Lemma~064Z}]{stacks-project} for example. In this case, we have an adjunction
	\[
		\Dps(A) \adjto \Dps(B)
	\]
	where the left adjoint is a tt-functor.
\end{Rem}

\begin{Rem}
	In \cite{Balmer18}, Balmer provides a useful criterion for the surjectivity of the map $\varphi:\Spc(\cat L)\to\Spc(\cat K)$ induced by a tt-functor $\cat K \to \cat L$. Unfortunately, this theorem requires that $\cat K$ is rigid. Nevertheless, there is something that can be said without this assumption:
\end{Rem}

\begin{Prop}\label{prop:img-of-Spc}
	Let $F:\cat K \to \cat L$ be a tt-functor between essentially small \mbox{tt-categories} and let $\varphi:\Spc(\cat L)\to \Spc(\cat K)$ be the induced map. Suppose that~$F$ has a right adjoint $U$ and that the projection formula
	\[
		U(F(a)\otimes b) \simeq a \otimes U(b)
	\]
	holds for $a \in \cat K$ and $b \in \cat L$. Then $\im(\varphi) = \supp(U(\unit))$.
\end{Prop}

\begin{proof}
	The $\subseteq$ inclusion is given at the start of the proof of \cite[Theorem~1.7]{Balmer18} and is a consequence of the projection formula. We prove the $\supseteq$ direction. Let $\cat P \in \supp(U(\unit))$. Consider the thick ideal $\ideal{F(\cat P)}$ and the multiplicative collection of objects $F(\cat K\setminus \cat P)$. We claim that 
	\begin{equation}\label{eq:int}
		\ideal{F(\cat P)} \cap F(\cat K\setminus \cat P) = \emptyset.
	\end{equation}
	It will then follow from \cite[Lemma~2.2]{Balmer05a} that there exists a prime ideal $\cat Q \in \Spc(\cat L)$ such that $\cat Q \supseteq \ideal{F(\cat P)}$ and $\cat Q\cap F(\cat K \setminus \cat P)=\emptyset$. These two conditions imply that $F^{-1}(\cat Q)=\cat P$, that is, $\varphi(\cat Q)=\cat P$. Thus, it remains to prove \eqref{eq:int}. If it did not hold then we would have 
	\[
		U(\ideal{F(\cat P)}) \cap U(F(\cat K \setminus \cat P)) \neq \emptyset.
	\]
	Note that since $U$ is an exact functor $U^{-1}(\cat P)$ is a thick subcategory. Also note that $\ideal{F(\cat P)} = \thicksub{F(\cat P)\otimes \cat L}$. The projection formula implies that $F(\cat P)\otimes \cat L \subseteq U^{-1}(\cat P)$ since $\cat P$ is an ideal. It follows that $U(\ideal{F(\cat P)}) \subseteq \cat P$. Thus it would follow that $\cat P \cap U(F(\cat K \setminus \cat P)) \neq \emptyset$. This means there exists $x \in \cat K \setminus \cat P$ such that $U(\unit)\otimes x \simeq UF(x) \in \cat P$. Since $\cat P$ is prime this would imply $U(\unit) \in \cat P$ which contradicts the hypothesis that $\cat P \in \supp(U(\unit))$.
\end{proof}

\begin{Cor}\label{cor:img-base-change}
	Let $A \to B$ be finite morphism of commutative rings. Let 
	\[
		\varphi:\Spc(\Dps(B))\to\Spc(\Dps(A))
	\]
	be the map induced by base change. Then $\im(\varphi) = \supp(B)$.
\end{Cor}

\begin{proof}
	The hypothesis is that $B$ is finitely generated as an $A$-module. The adjunction $f^*:\Der(A)\adjto \Der(B):f_*$ thus restricts to an adjunction $\Dps(A) \adjto \Dps(B)$ as explained in \cref{rem:base-change}. Moreover, since the projection formula holds for the big categories \cite[Proposition~2.15]{BalmerDellAmbrogioSanders16}, it also holds for the induced adjunction. Hence we can apply \cref{prop:img-of-Spc}.
\end{proof}

\begin{Exa}
	For any ideal $I \subseteq R$, the image of $\Spc(\Dps(R/I))\to\Spc(\Dps(R))$ is the Thomason closed set $\supp(R/I)$. For a prime ideal $I=\frakp$, this is:
\end{Exa}

\begin{Prop}\label{prop:supp-Rp}
	For any $\frakp \in \Spec(R)$, we have
	\begin{equation}\label{eq:supp-R/p}
		\overbar{\{\tame(\frakp)\}} = \supp(R/\frakp) = \supp(\kosz(\frakp)) = \rho^{-1}(\overbar{\{\frakp\}}).
	\end{equation}
\end{Prop}

\begin{proof}
	It follows from \eqref{eq:tame-specialization} and \cref{cor:key-lemma} that $\rho^{-1}(\overbar{\{\frakp\}}) \subseteq \overbar{\{\tame(\frakp)\}}$. Hence it suffices to prove that the equalities in \eqref{eq:supp-R/p} are inclusions $\subseteq$. First note that $\overbar{\{\frakp\}} = \Supp(R/\frakp)$. Hence $R/\frakp \in \langle \Kosz(\frakp)\rangle$ by \cref{prop:MT}, so that $\supp(R/\frakp) \subseteq \supp(\Kosz(\frakp))$. Also, since $\Supp(R/\frakp) = \tame^{-1}(\supp(R/\frakp))$ we have $\tame(\frakp) \in \supp(R/\frakp)$ and hence $\overbar{\{\tame(\frakp)\}} \subseteq \supp(R/\frakp)$. Finally, $\supp(\Kosz(\frakp)) = \rho^{-1}(\overbar{\{\frakp\}})$ by \cref{rem:supp-of-perfect} and the definitions.
\end{proof}

\begin{Exa}
	Let $\frakp \in \Spec(R)$ and consider the commutative diagram
	\[\begin{tikzcd}
		\Dps(R) \ar[r] \ar[d] & \Dps(R/\frakp) \ar[d] \\
		\Dps(R_{\frakp}) \ar[r] & \Dps(\kappa(\frakp))
	\end{tikzcd}\]
	The unique point in the bottom-right maps to the unique closed point $\frakm$ in the bottom-left and the unique generic point $\eta$ in the top-right, which are both in turn mapped to the tame prime $\tame(\frakp)$ in the top-left; recall \cref{prop:local-or-domain}. In particular,
	\begin{equation}\label{eq:tame-is-kernel}
		\tame(\frakp) = \Ker(\Dps(R) \to \Dps(R_\frakp)).
	\end{equation}
	This can also be seen directly from the definition of $\tame(\frakp)$; see \cref{cons:tame}.
\end{Exa}

\begin{Lem}\label{lem:tame-compat}
	Let $A\to B$ be a morphism of noetherian commutative rings. Then the following diagram commutes
	\[\begin{tikzcd}
		\Spec(B) \ar[r,"f"]\ar[d,hook,"\tame"] & \Spec(A) \ar[d,hook,"\tame"] \\
		\Spc(\Dps(B)) \ar[r] & \Spc(\Dps(A))
	\end{tikzcd}\]
	where the vertical maps are the inclusions of the tame primes.
\end{Lem}

\begin{proof}
	Unravelling the definitions, this amounts to the claim that for any $a \in \Dps(A)$ and $\frakp \in \Spec(B)$, we have $\frakp \in \Supp(f^*(a))$ if and only if $f(\frakp) \in \Supp(a)$. Since the rings are noetherian, their derived categories are stratified in the sense of \cite{BarthelHeardSanders23a}, hence $f^{-1}(\Supp(a))=\Supp(f^*(a))$ by \cite[Corollary~14.19]{BarthelCastellanaHeardSanders23app}.
\end{proof}

\begin{Cor}\label{cor:finite-surjective}
	Let $A\to B$ be a finite morphism of commutative noetherian rings such that $\Spec(B)\to \Spec(A)$ is surjective. Then the map
	\[
		\Spc(\Dps(B))\to\Spc(\Dps(A))
	\]
	is surjective.
\end{Cor}

\begin{proof}
	By \cref{cor:img-base-change}, the image of the map is $\supp(B)$. On the other hand, the surjectivity of $\Spec(B)\to\Spec(A)$ and the commutative diagram of \cref{lem:tame-compat} implies that the image contains every tame prime. Since $\supp(B)$ is closed, it must contain every specialization of every tame prime; hence it must contain everything by \cref{cor:key-lemma}.
\end{proof}

\begin{Rem}
	Every tame prime is visible\footnote{A point $x$ in a spectral space is said to be \emph{visible} if its closure $\overbar{\{x\}}$ is a Thomason closed set.} by \cref{prop:supp-Rp}. In fact:
\end{Rem}

\begin{Prop}
	For a closed point $\cat P \in \Spc(\Dps(R))$, the following are equivalent:
	\begin{enumerate}
		\item $\cat P$ is visible.
		\item $\cat P$ is tame.
		\item $\rho(\cat P)$ is a closed point of $\Spec(R)$.
	\end{enumerate}
\end{Prop}

\begin{proof}
	$(a) \Rightarrow (b)$: If $\cat P$ is a visible closed point then $\{\cat P\}$ is Thomason closed, hence it must contain a tame prime by \cref{prop:contains-tame} and thus we must have $\cat P = \tame(\rho(\cat P))$.

	$(b) \Rightarrow (c)$: Suppose $\cat P = \tame(\frakp)$ is tame. If $\rho(\cat P)\rightsquigarrow \frakp$ then $\cat P = \tame(\rho(\cat P)) \rightsquigarrow \tame(\frakp)$ so that~$\cat P=\tame(\frakp)$. Therefore $\rho(\cat P)= \frakp$. This establishes that $\rho(\cat P)$ is closed.

	$(c)\Rightarrow (a)$: If $\rho(\cat P) \eqqcolon \frakm$ is a closed point then $\cat P \in \rho^{-1}(\{\frakm\}) = \{\tame(\frakm)\}$ invoking \cref{prop:tame-closed}. Hence $\cat P$ is visible, since $\{\cat P\} = \rho^{-1}(\{\frakm\})$ is Thomason, being the preimage of a Thomason set under a spectral map.
\end{proof}

\begin{Rem}
	The question of whether all the closed points of $\Spc(\Dps(R))$ are tame was raised in \cite[Question 4.15]{MatsuiTakahashi17} and they provide an affirmative answer if $\Spec(R)$ has only finitely many closed points (i.e.~$R$ is semi-local). We can provide an alternative proof of this fact:
\end{Rem}

\begin{Prop}
	If $R$ is semi-local then all the closed points of $\Spc(\Dps(R))$ are tame. Consequently, the closed points of $\Spc(\Dps(R))$ correspond bijectively with the closed points of $\Spec(R)$.
\end{Prop}

\begin{proof}
	Let $\frakm_1, \frakm_2, \ldots, \frakm_n$ be the collection of maximal ideals of $R$. The tt-functor $\Dps(R) \to \prod_{i=1}^n \Dps(R_{\frakm_i})$ is conservative since $\Supp(X)$ is empty for any object $X$ in the kernel, since it is specialization closed and yet contains no closed points.\footnote{In fact, for any commutative noetherian ring, the family of functors $\Der(R) \to \Der(R_{\frakm})$, ranging over the maximal ideals $\frakm$, is jointly conservative since they are jointly surjective on the spectra of compact objects; see \cite[Corollary~14.24]{BarthelCastellanaHeardSanders23app} or \cite[Theorem~1.9]{BarthelCastellanaHeardSanders24} bearing in mind \cite[Theorem~A]{BarthelHeardSanders23b}.} Thus, \cite[Theorem~1.2]{Balmer18} establishes that the image of the induced map
	\[ 
		\coprod_{i=1}^n \Spc(\Dps(R_{\frakm_i})) \cong \Spc(\prod_{i=1}^n \Dps(R_{\frakm_i})) \to \Spc(\Dps(R))
	\]
	contains all the closed points of $\Spc(\Dps(R))$. Thus, given any closed point $\cat P$ in~$\Spc(\Dps(R))$, there exists a maximal ideal $\frakm$ of $\Spec(R)$ and a prime ideal $\cat Q$ in $\Spc(\Dps(R_\frakm))$ mapping to $\cat P$ under the canonical map $\Spc(\Dps(R_\frakm))\to\Spc(\Dps(R))$. It follows that the unique closed point $\tame(\frakm) \in \Spc(\Dps(R_\frakm))$ also maps to $\cat P$. \Cref{lem:tame-compat} then implies that $\cat P = \tame(\frakm)$ is itself tame.
\end{proof}

We now turn to the process of localization.

\begin{Lem}\label{lem:localization-surjective}
	For any multiplicative subset $S \subset R$, the functor 
	\[
		\Dps(R) \to \Dps(S^{-1}R)
	\]
	is essentially surjective.
\end{Lem}

\begin{proof}
	The argument is similar to the one given in \cite[Lemma~3.9]{Letz21}. Let $A$ be a commutative ring and let $X_\bullet$ be a bounded below complex of $A$-modules, say with $X_i = 0$ for $i<0$ for concreteness. For any sequence of units $s_1,s_2,\ldots$ in~$A$, we can construct a new complex $Y$ by setting $Y_i \coloneqq X_i$ and $d_i^Y \coloneqq s_i.d_i^X$ for each~$i$. The morphism $Y_\bullet \to X_\bullet$ which in degree $i$ is multiplication by $s_1s_2\cdots s_i$ is then a quasi-isomorphism. With this observation in hand, consider the localization $R \to S^{-1} R$. It is standard that every finitely generated $S^{-1} R$-module is (up to isomorphism) of the form $S^{-1} M = M \otimes_R S^{-1} R$ for some finitely generated $R$-module $M$; cf.~\cite[Exercise~2.10]{Eisenbud95}. Moreover, since finitely generated modules over noetherian rings are finitely presented, we have $\Hom_{S^{-1} R}(S^{-1} M,S^{-1} N) \simeq S^{-1} \Hom_R(M,N)$ for any finitely generated $R$-modules $M$ and $N$. It follows that for any $f \in \Hom_{S^{-1} R}(S^{-1} M,S^{-1} N)$ there is $s \in S$ such that $s.f \in \Hom_R(M,N)$. Thus, any bounded below complex of $S^{-1} R$-modules is quasi-isomorphic to a bounded below complex of the form $(S^{-1} M_i,S^{-1} d_i)_i$ for some finitely generated $R$-modules $M_i$ and $R$-linear maps $d_i:M_i \to M_{i-1}$. Although $(S^{-1} d_{i}) \circ (S^{-1} d_{i+1}) = S^{-1}(d_{i} \circ d_{i+1}) = 0$ we do not necessarily have $d_{i}\circ d_{i+1}=0$. Nevertheless, for each $i$, there exists an $s_i \in S$ such that $0=s_i.(d_{i}\circ d_{i+1}) = (s_i.d_i)\circ d_{i+1}$. Thus, replacing each $d_i$ by $s_i.d_i$, we obtain a bounded below complex of $R$-modules $(M_i,s_i.d_i)_i$ which localizes to a complex quasi-isomorphic to our original complex.
\end{proof}

\begin{Prop}\label{prop:localization-verdier}
	For any multiplicative subset $S \subset R$, the functor 
	\[
		\Dps(R)\to\Dps(S^{-1} R)
	\]
	is a Verdier localization.
\end{Prop}

\begin{proof}
	The Bousfield localization functor $\Der(R)\to\Der(S^{-1} R)$ restricts to a tt-functor $F:\Dps(R)\to \Dps(S^{-1} R)$ which is essentially surjective by \cref{lem:localization-surjective}. It induces a conservative essentially surjective tt-functor $\overbar{F}:\Dps(R)/\Ker F \to \Dps(S^{-1} R)$. We will prove that this functor $\overbar{F}$ is full. It will then follow formally that it is fully faithful and hence an equivalence; see, e.g., \cite[Lemma~2.1]{CanonacoOrlovStellari13} or \cite[\S 7]{Sanders25bpp}. We will use the notation $S^{-1}(-)=S^{-1}R \otimes_R -$ for the functor from (complexes of) $R$-modules to (complexes of) $S^{-1}R$-modules. Let $X$ and $Y$ be two bounded below complexes of finitely generated $R$-modules and consider a morphism of complexes $f:S^{-1} X \to S^{-1} Y$. Without loss of generality, we may assume that $Y_n = 0$ for $n<0$. For each index $n$, we have the morphism of $S^{-1}R$-modules $f_n:S^{-1} X_n \to S^{-1} Y_n$. As recalled in the proof of \cref{lem:localization-surjective}, there exists an $s_n \in S$ such that $s_n.f_n = S^{-1}(g_n)$ for some morphism $g_n : X_n \to Y_n$ of $R$-modules. From the equality
	\[
		S^{-1}(d_n^Y)\circ f_{n} = d_n^{S^{-1} Y} \circ f_{n} = f_{n-1} \circ d_n^{S^{-1} X} = f_{n-1} \circ S^{-1}(d_n^X)
	\]
	it follows that 
	\begin{align*}
		0 &= s_n s_{n-1}.(S^{-1}(d_n^Y)\circ f_{n} - f_{n-1} \circ S^{-1}(d_n^X))\\
		  &= S^{-1}(s_{n-1}.d_n^Y)\circ s_n.f_n - s_{n-1}.f_{n-1} \circ S^{-1}(s_n.d_n^X)\\
		  &= S^{-1}( s_{n-1}.d_n^Y \circ g_n - g_{n-1} \circ s_n.d_n^X)
	\end{align*}
	in $\Hom_{S^{-1} R}(S^{-1} X_n, S^{-1} Y_{n-1}) \simeq S^{-1} \Hom_R(X_n,Y_{n-1})$.	Hence there exists $t_n \in S$ such that 
	\[
		0 = t_ns_{n-1}.d_n^Y\circ g_n - g_{n-1} \circ t_ns_n.d_n^X
	\]
	in $\Hom_R(X_n,Y_{n-1})$. For $n<0$, we may take $g_n=0$, $s_n=1$ and $t_n=1$. With these observations in hand, let $\widetilde{X}$ be the complex with $\widetilde{X}_n = X_n$ and $d_n^{\widetilde{X}} = t_ns_n.d_n^X$. Also, let $\widetilde{Y}$ be the complex with $\widetilde{Y}_n = Y_n$ and $d_n^Y = t_ns_{n-1}.d_n^Y$. By construction, the $g_n$ define a morphism of complexes $g:\widetilde{X}\to\widetilde{Y}$. Next consider the morphism of complexes $u:\widetilde{X} \to X$ defined by $u_n = t_0t_1\cdots t_n s_0 s_1 \cdots s_n$ for $n \ge 0$ and $u_n = 1$ for $n<0$. Finally, let $v:\widetilde{Y} \to Y$ be the morphism of complexes given by $v_n = t_0 t_1 \cdots t_n s_0 s_1 \cdots s_{n-1}$ for $n \ge 0$. One may then verify from the definitions that the morphisms $u:\widetilde{X} \to X$, $g:\widetilde{X}\to \widetilde{Y}$, $v:\widetilde{Y}\to Y$, $f:S^{-1} X \to S^{-1} Y$ satisfy
	\[
		f\circ S^{-1}(u) = S^{-1}(v) \circ S^{-1}(g)
	\]
	of morphisms $S^{-1}(\widetilde{X}) \to S^{-1}(Y)$. Moreover, note that $S^{-1}(u)$ and $S^{-1}(v)$ are isomorphisms. Since $u$ and $v$ are maps of pseudo-coherent complexes, they are also inverted by $q:\Dps(R)\to \Dps(R)/\Ker F$. Hence, $f=\overbar{F}(q(v)\circ q(g)\circ q(u)^{-1})$. It follows that $\Dps(R)/\Ker F \to \Dps(S^{-1} R)$ is full and hence is an equivalence.
\end{proof}

\begin{Cor}
	For any $\frakp \in \Spec(R)$, the localization $R\to R_{\frakp}$ induces an equivalence of tt-categories $\Dps(R)/\tame(\frakp) \cong \Dps(R_{\frakp})$.
\end{Cor}

\begin{proof}
	This follows from \cref{prop:localization-verdier} and \eqref{eq:tame-is-kernel}.
\end{proof}

\begin{Rem}
	This establishes that $\Dps(R_{\frakp})$ is the local category of $\Dps(R)$ at the tame prime $\tame(\frakp)$. In particular, the localization $R\to R_{\frakp}$ induces a topological embedding $\Spc(\Dps(R_{\frakp})) \xrightarrow{\sim} \gen(\tame(\frakp)) \subseteq \Spc(\Dps(R))$.
\end{Rem}

\begin{Rem}\label{rem:algebraic-localization}
	Recall that for each $\frakp \in \Spec(R)$, we have the algebraic localization
	\[
		\Dps(R)_{\frakp} \coloneqq \Dps(R)/\ideal{\cone(s) \mid s \not\in \frakp}.
	\]
	One of the difficulties in understanding $\Spc(\Dps(R))$ is that it does not behave well with respect to algebraic localizations in the sense that usually $\Dps(R)_{\frakp} \not\cong \Dps(R_{\frakp})$. It follows from the definitions that for any $\cat P \in \rho^{-1}(\{\frakp\})$ the functor $\Der(R) \to \Der(R_{\frakp})$ factors as
	\begin{equation}\label{eq:factorization}
		\Dps(R)\twoheadrightarrow\Dps(R)_{\frakp} \twoheadrightarrow \Dps(R)/{\cat P} \twoheadrightarrow \Dps(R)/\tame(\frakp) \xrightarrow{\cong} \Dps(R_{\frakp}).
	\end{equation}
	In particular, if the fiber $\rho^{-1}(\{\frakp\})$ is not a singleton then there is no hope for the functor $\Dps(R)_{\frakp}\to\Dps(R_{\frakp})$ to be conservative let alone an equivalence.
\end{Rem}

\begin{Exa}
	Let $R=\bbZ$ and let $p_1,p_2, p_3,\ldots$ be the sequence of prime numbers. Consider $X\coloneqq \coprod_{i \ge 1} \bbZ/p_i [i] \in \Dps(\bbZ)$. Observe that
	\[
		X \not\in \ideal{\bbZ/n \mid n\not\in\frakp} = \Ker(\Dps(\bbZ)\to \Dps(\bbZ)_\frakp)
	\]
	for any $\frakp \in \Spec(R)$ since $X$ cannot be built in finitely many steps from this collection of cyclic modules. Hence $X_{\frakp} \neq 0$ in $\Dps(\bbZ)_{\frakp}$ for all $\frakp \in \Spec(\bbZ)$. On the other hand, $\Supp(X)$ consists precisely of all the closed points. In particular, $X$ vanishes in $\Dps(\bbQ)$ but does not vanish in $\Dps(\bbZ)_{(0)}$. Thus, $\Dps(\bbZ)_{(0)} \to \Dps(\bbQ)$ is not conservative.
\end{Exa}

\begin{Rem}
	In summary, the complexity of $\Dps(R)$ over $\Dperf(R)$ is reflected by the lack of conservativity of the functor $\Dps(R)_{\frakp} \to \Dps(R_{\frakp})$. In contrast,~$\Dperf(R_{\frakp})$ is the idempotent-completion of $\Dperf(R)_{\frakp}$ by the Neeman--Thomason theorem~\cite{Neeman92b}.
\end{Rem}

\begin{Rem}
	The majority of this paper is devoted to understanding $\Spc(\Dps(R))$ in the case when $R$ is a discrete valuation ring. We thus end this section with a few remarks about how one can relate arbitrary noetherian rings to this case. 
\end{Rem}

\begin{Exa}\label{exa:excellent-dimension-1}
	Let $R$ be an excellent local domain of dimension 1 and let $A$ be the integral closure of $R$ in its field of fractions. The ring $A$ is still noetherian by the Krull--Akizuki theorem and it is module-finite over $R$ by the excellence hypothesis. By the lying-over theorem, we know that $\Spec(A) \to \Spec(R)$ is surjective. Thus there exists a prime $\frakp \in \Spec(A)$ which maps to the maximal ideal $\frakm$ of $R$. Then~$A_{\frakp}$ is a discrete valuation ring and the composite
	\[
		\Spec(A_{\frakp}) \hookrightarrow \Spec(A) \twoheadrightarrow \Spec(R)
	\]
	is a homeomorphism. We have an analogous situation with the pair of maps
	\[
		\Spc(\Dps(A_{\frakp})) \hookrightarrow \Spc(\Dps(A)) \twoheadrightarrow \Spc(\Dps(R))
	\]
	by \cref{prop:localization-verdier} and \cref{cor:finite-surjective}.
\end{Exa}

\begin{Rem}
	More generally, given any specialization $\frakp \in \overbar{\{\frakq\}}$ in a noetherian ring $R$, there exists a map $R \to A$ to a discrete valuation ring which realizes this specialization. This is explained in detail in the proof of \cite[\href{https://stacks.math.columbia.edu/tag/054F}{Lemma~054F}]{stacks-project} and the references therein. Briefly, one first passes to a local domain (via localization and taking a quotient) and then to a local domain of dimension 1 by passing to a maximal dominating valuation ring in its field of fractions (see \cite[\href{https://stacks.math.columbia.edu/tag/00P8}{Lemma~00P8}]{stacks-project} and \cite[\href{https://stacks.math.columbia.edu/tag/00IA}{Lemma~00IA}]{stacks-project}) and then one takes a localization of its integral closure as in \cref{exa:excellent-dimension-1}. However, we do not have a complete understanding of how~$\Dps(-)$ behaves under each of these steps. For an excellent ring, we do have some understanding of the last step as indicated in \cref{exa:excellent-dimension-1}.
\end{Rem}

\section{Discrete valuation rings}\label{sec:DVR}

We turn our attention to pseudo-coherent complexes over discrete valuation rings.

\begin{Hyp}
	For the remainder of the paper, $R$ denotes a discrete valuation ring with unique maximal ideal $\frakm = (x)$ and residue field $k=R/\frakm$.
\end{Hyp}

\begin{Rem}
	The goal is to compute $\Spc(\Dps(R))$. As explained in \cref{rem:comparison-map} we have a surjective spectral map
	\[
		\rho:\Spc(\Dps(R)) \twoheadrightarrow \Spec(R)
	\]
	and the latter space is very simple. It consists of two points, a closed point $\frakm$ and a generic point $\eta$:
	\[
		\Spec(R) =\quad\begin{tikzcd}[column sep=tiny,row sep=small]
			\bullet\ar[d,dash,thick] &\mathfrak{m}=(x)  \\
			\bullet  &\eta=(0)&
		\end{tikzcd}
	\]
	By the results of \cref{sec:tame,sec:base-change}, we know that $\Spc(\Dps(R))$ is a local space with unique closed point given by the zero ideal which coincides with the tame prime $\tame(\frakm)$ associated with the unique closed point $\frakm \in \Spec(R)$. Moreover, the fiber $\rho^{-1}(\{\frakm\})=\{\tame(\frakm)\}$ is a singleton. It also follows from general principles that $\Spc(\Dps(R))$ is irreducible with generic point given by the tame prime $\tame(\eta)$ associated to the generic point $\eta \in \Spec(R)$. We will describe this tame prime in \cref{rem:dpfl} below.
\end{Rem}

\begin{Rem}\label{rem:ideals}
	The ideals of $R$ are precisely the zero ideal $(0)$ together with the powers of the maximal ideal $\frakm^i=(x^i)$, $i \ge 0$. This is a straightforward consequence of the Krull intersection theorem which asserts that $\bigcap_{i \ge 0} \mathfrak m^i = (0)$ in any noetherian local ring.
\end{Rem}

\begin{Rem}
	Since $R$ is a principal ideal domain, every finitely generated $R$-module is isomorphic to a direct sum of cyclic modules. Moreover, by \cref{rem:ideals}, every nonzero cyclic module is either $R$ itself or of the form $R/x^i$ for some $i \ge 1$.
\end{Rem}

\begin{Lem}\label{lem:tensor-cyclic}
	For any integers $i,j\ge 0$, the following statements hold:
	\begin{enumerate}
		\item $R/x^i \otimes_R R/x^j \simeq R/x^{\min(i,j)}$.
		\item $\Tor_1^R(R/x^i,R/x^j)\simeq R/x^{\min(i,j)}$.
		\item There is a short exact sequence
			\[ 0 \to R/x^i \xrightarrow{x^j} R/x^{i+j} \to R/x^j \to 0.\]
	\end{enumerate}
\end{Lem}

\begin{proof}
	These are straightforward exercises.
\end{proof}

\begin{Lem}\label{lem:finite-length}
	For a finitely generated $R$-module $M$, the following are equivalent:
	\begin{enumerate}
		\item $M$ has finite length;
		\item $M$ has no free summand;
		\item $M$ is torsion;
		\item $\eta \not\in\Supp_R(M)$.
	\end{enumerate}
\end{Lem}

\begin{proof}
	It follows from \cref{rem:ideals} that the proper cyclic modules $R/\frakm^i$ are artinian while the free module $R$ is of course not artinian. Hence the $R$-modules of finite length are precisely the finitely generated $R$-modules which have no free summand (i.e.~which are torsion). On the other hand, recall that $\Supp_R(M)=V(\Ann(M))$ for any finitely generated module, so $\eta \not\in \Supp_R(M)$ if and only if $\Ann(M)\neq (0)$ which is the case precisely when $M$ has no free summand.
\end{proof}

\begin{Def}
	For a finite length $R$-module $M$, the \emph{Loewy length} of $M$ is
	\[
		\loew(M) \coloneqq \min\SET{i \ge 0}{\frakm^i M = 0}.
	\]
	For example, $\loew(R/x^i)=i$ for each $i \ge 0$. In particular, the Loewy length of the zero module is zero. In general, $\loew(M)$ is the largest $i$ such that $R/x^i$ is a direct summand of $M$.
\end{Def}

\begin{Lem}\label{lem:loew}
	Let $M$ and $N$ be finite length $R$-modules. Then
	\begin{enumerate}
		\item $\loew(M\oplus N) = \max(\loew(M),\loew(N))$.
		\item $\loew(M\otimes_R N) = \min(\loew(M),\loew(N))$.
		\item $\loew(\Tor_1^R(M,N))=\min(\loew(M),\loew(N))$.
	\end{enumerate}
\end{Lem}

\begin{proof}
	Part $(a)$ is immediate from the definition, while parts $(b)$ and $(c)$ follow from \cref{lem:tensor-cyclic}.
\end{proof}

\begin{Lem}\label{lem:loewSES}
	Let $0 \to N \to M \to L \to 0$ be a short exact sequence of finite length $R$-modules. Then $\loew(M) \le \loew(N)+\loew(L)$.
\end{Lem}

\begin{proof}
	We have $\frakm^{\loew(L)}(M/N)=0$ so that $\frakm^{\loew(L)}(M) \subseteq N$. Then 
	\[
		\frakm^{\loew(N)+\loew(L)}M = \frakm^{\loew(N)}(\frakm^{\loew(L)}M) \subseteq \frakm^{\loew(N)} N = 0.
	\]
	Hence $\loew(M) \le \loew(N)+\loew(L)$ as desired.
\end{proof}

We now turn to complexes of $R$-modules.

\begin{Rem}\label{rem:hereditary-splitting}
	Since $R$ is hereditary, every complex $E\in\Der(R)$ is isomorphic to the complex 
	\[
		\bigoplus_{i \in \bbZ} H_i(E)[i]=(\cdots \to H_2(E) \xrightarrow{0} H_1(E) \xrightarrow{0} H_0(E) \xrightarrow{0} H_{-1}(E) \to \cdots)
	\] 
	consisting of the homology of $E$ with zero differentials; see \cite[Section~1.6]{Krause07}, for example. In particular, any pseudo-coherent complex $E \in \Dps(R)$ is isomorphic to a bounded below complex of finitely generated modules with zero differentials.
\end{Rem}

\begin{Rem}\label{rem:dpfl}
	We write $\Dpfl(R) \subset \Dps(R)$ for the full subcategory consisting of complexes homologically bounded on the right with $H_i(E)$ of finite length for each $i\in \bbZ$. It is a prime ideal of $\Dps(R)$. This can be checked directly, or we may observe that it is the tame prime $\tame(\eta)$ associated to the generic point $\eta \in \Spec(R)$. Indeed:
	\begin{align*}
		\tame(\eta)  \coloneqq& \SET{X \in \Dps(R)}{\eta \not\in\Supp(X)} \\
		=&\SET{X \in \Dps(R)}{\eta \not\in \bigcup\nolimits_{i \in \bbZ}\Supp_R(H_i(X))}\\
		=&\Dpfl(R)
	\end{align*}
	by \cref{rem:specialization-closed}, \cref{def:tame} and \cref{lem:finite-length}.
\end{Rem}

\begin{Rem}\label{rem:Dpfl-largest}
	Note that if $X \not\in \Dpfl(R)$ then $X$ must contain a shift of $R$ as a direct summand and hence $\ideal{X} = \Dps(R)$. Thus, the prime ideal $\Dpfl(R)$ is in fact the unique largest proper thick ideal of $\Dps(R)$. Hence much of our discussion will be concerned with understanding the complexes in $\Dpfl(R)$.
\end{Rem}

\begin{Rem}
	We may also consider the thick subcategory $\Dbfl(R)$ consisting of all homologically bounded complexes with finite length homology modules. It is not an ideal in $\Dps(R)$. For example: 
	\[
		R/\frakm \otimes (\cdots \to R \xrightarrow{0} R \xrightarrow{0} R \to 0 \to 0 \to \cdots) \simeq \bigoplus_{i \in \bbN} R/\frakm[i] \not\in \Dbfl(R).
	\]
\end{Rem}

\begin{Lem}\label{lem:cone}
	If $i \le j$ then $\cone(R/x^j \xrightarrow{x^i} R/x^j)\in \Der(R)$ contains $R/x^i$ as a direct summand.
\end{Lem}

\begin{proof}
	Let $M\coloneqq R/x^j$ and consider the morphism of $R$-modules $x^i:M\to M$ regarded as a morphism in the derived category. Its cone is the Koszul complex $\kosz(x;M)=(\cdots\to 0 \to M \xrightarrow{x} M \to 0 \to \cdots)$ concentrated in degrees 1 and 0. Since $R$ is hereditary, it is isomorphic to a complex which has $H_0(\kosz(x;M))\simeq M/x^iM$ as a direct summand. Moreover, $M/x^iM\simeq M\otimes_R R/x^i \simeq R/x^i$ by \cref{lem:tensor-cyclic}.
\end{proof}

\begin{Prop}\label{prop:Dbfl-minimal}
	We have $\Dbfl(R) = \thicksub{R/x} = \thicksub{R/x^i}$ for all $i \ge 1$. This is the smallest nonzero thick subcategory of $\Dps(R)$.
\end{Prop}

\begin{proof}
	The short exact sequence in \cref{lem:tensor-cyclic}(c) implies that $\thicksub{R/x}$ contains~$R/x^i$ for all $i \ge 1$ and this implies that $\Dbfl(R)=\thicksub{R/x}$. On the other hand $R/x \in \thicksub{R/x^i}$ by \cref{lem:cone}. The second statement says that $\Dbfl(R)$ is contained in every nonzero thick subcategory. If $0 \neq X \in \Dps(R)$ then by \cref{rem:hereditary-splitting} the object $X$ contains a shift of $R/x^i$ for some $i \ge 1$ or a shift of $R$ as a direct summand. In the former case, we have $\Dbfl(R) \subseteq \thicksub{X}$ by what we have just proved. In the latter case, the short exact sequence $0\to R \xrightarrow{x} R \to R/x \to 0$ shows that $R/x \in \thicksub{R}$ and hence $\Dbfl(R) \subseteq \thicksub{X}$ as well.
\end{proof}

\begin{Rem}\label{rem:truncation}
	Given $X \in \Dpfl(R)$ and any $n \in \bbZ$ we can form the good truncation $X_{\ge n} \in \Dpfl(R)$ which satisfies
	\[
		H_i(X_{\ge n}) = \begin{cases}
			H_i(X) & \text{ for $i\ge n$}\\
			0 & \text{ for $i<n$}.
		\end{cases}
	\]
	Since every complex splits into the direct sum of its homology, we have $X \simeq X_{\ge n} \oplus X'$ with $X' \in \Dbfl(R)$. As noted above, all nonzero bounded complexes generate the same thick subcategory:
	\[
		0 \neq X \in \Dbfl(R) \Longrightarrow \thicksub{X} = \Dbfl(R).
	\]
	On the other hand, if $X$ is not bounded then $\thicksub{X} = \thicksub{X_{\ge n}}$ for any $n \in \bbZ$. The bottom-line is that only the behaviour in high degree is relevant.
\end{Rem}

\begin{Def}[Loewy sequences]\label{def:loewy-seq}
	Let $E \in \Dpfl(R)$ be a complex with $H_0(E) \neq 0$ and $H_n(E) = 0$ for $n<0$. We define a function $\loew_E:\bbN\to\bbN$ by
	\[
		\loew_E(n)\coloneqq \loew(H_n(E))
	\]
	for each $n \in \bbN$. More generally, if $E \in \Dpfl(R)$ is nonzero then $E[-\inf(E)]$ is concentrated in nonnegative degrees as above and we set $\loew_E \coloneqq \loew_{E[-\inf(E)]}$. Finally, we define $\loew_E$ for the zero complex $E$ to be the constant zero function.
\end{Def}

\begin{Cons}\label{cons:Rxf}
	Let $f:\bbN \to \bbN$ be a function. We write $R/x^f$ for the complex
	\[
		\cdots \to R/x^{f(2)} \to R/x^{f(1)} \to R/x^{f(0)} \to 0 \to 0 \to \cdots
	\]
	which has zero differentials and where $R/x^{f(0)}$ sits in degree $0$. Note that $R/x^f \in \Dpfl(R)$ and we have $\loew_{R/x^f} = f$ provided that $f(0)\neq 0$.
\end{Cons}

\begin{Rem}
	For any $k \in \bbN$, we can define $\sigmakf:\bbN \to \bbN$ by $\sigmakf(n) = f(n+k)$ in conformance with \cref{not:sigma-mu}.
	In particular, if $f \neq \underline{0}$ then we can consider 
	\[
		\inf(f) \coloneqq \inf(R/x^f) = \min\{n\in \bbN \mid f(n) \neq 0\}.
	\]
	According to \cref{def:loewy-seq}, $\loew_{R/x^f} = \sigmainff$. Moreover,
	\[
		R/x^{\sigmainff} = R/x^f[-\inf(f)].
	\]
	Thus, the function $f' \coloneqq \sigmainff$ satisfies $\inf(f') = 0$ and we have $\thicksub{R/x^f} = \thicksub{R/x^{f'}}$ since these complexes are just shifts of each other. It follows that in our study of nonzero complexes $R/x^f$, there will often be no loss of generality in assuming that $f(0) \neq 0$.
\end{Rem}

\begin{Lem}\label{lem:pointwise}
	Let $f,g:\bbN\to\bbN$ be functions. If $f(n) \le g(n)$ for all $n \in \bbN$ then $R/x^f \in \thicksub{R/x^g}$.
\end{Lem}

\begin{proof}
	Consider the morphism of complexes $R/x^g \to R/x^g$ which in degree $n$ is $x^{f(n)}:R/x^{g(n)} \to R/x^{g(n)}$. Its cone has $R/x^f = \bigoplus_{n \in \bbN} R/x^{f(n)}[n]$ as a direct summand by \cref{lem:cone}.
\end{proof}

\begin{Lem}\label{lem:grow}
	Let $f,g:\bbN\to \bbN$ be functions. There is a short exact sequence of complexes
	\[
		0 \to R/x^f \to R/x^{f+g} \to R/x^g \to 0.
	\]
\end{Lem}

\begin{proof}
	This follows from \cref{lem:tensor-cyclic}(c) and the definitions.
\end{proof}

\begin{Exa}\label{exa:join}
	It follows from \cref{lem:pointwise} and \cref{lem:grow} that
	\[
		\thicksub{R/x^f \oplus R/x^g} = \thicksub{R/x^{f \vee g}} = \thicksub{R/x^{f+g}}
	\]
	for any two funtions $f,g:\bbN\to \bbN$.
\end{Exa}

\begin{Prop}\label{prop:asym-direct}
	If $f,g:\bbN\to \bbN$ are two monotonic sequences such that $f\le g$, that is, such that $f$ is asymptotically bounded by $g$, then $R/x^f \in \thicksub{R/x^g}$.
\end{Prop}

\begin{proof}
	The claim is trivial if $f=\underline{0}$. Otherwise, $f$ and hence $g$ is nonzero. Hence $\thicksub{R/x^g} \supseteq \Dbfl(R)$ by \cref{prop:Dbfl-minimal}. By definition, $f\le g$ means that there exist positive integers $A$ and $n_0$ such that $f(n) \le Ag(n)$ for all $n \ge n_0$. Thus the truncation $(R/x^f)_{\ge n_0} \in \thicksub{R/x^{Ag}}$ by \cref{lem:pointwise}. On the other hand, it follows from \cref{lem:grow} that $R/x^{Ag} \in \thicksub{R/x^g}$. Hence $R/x^f \in \thicksub{R/x^g}$ bearing in mind \cref{rem:truncation}.
\end{proof}

We next take the tensor into account.

\begin{Rem}\label{rem:Htensor}
	Let $X,Y\in \Dps(R)$. It follows from the decompositions of \cref{rem:hereditary-splitting} that
	\[
		H_n(X \otimes_R^{\bbL} Y) = \left( \bigoplus_{i+j=n} H_i(X) \otimes_R H_j(Y)\right)\oplus \left(\bigoplus_{i+j=n-1}\Tor_1^R\big(H_i(X),H_j(Y)\big)\right)
	\]
	for each $n \in \bbZ$. We will use this decomposition repeatedly in what follows. Here we have used the notation $\otimes_R^{\mathbb{L}}$ for emphasis, but throughout the paper we are simply writing $\otimes$ for the derived tensor product of complexes.
\end{Rem}

\begin{Rem}
	Recall from \cref{exa:to-mono} that an arbitrary function $f:\bbN\to\bbN$ gives rise to a monotonic function $\mono{f}$.
\end{Rem}

\begin{Prop}\label{prop:ideal-force-mono}
	For any function $f:\bbN\to\bbN$, we have $\ideal{R/x^f} = \ideal{R/x^{\mono{f}}}$.
\end{Prop}

\begin{proof}
	We have the $\subseteq$ inclusion by the above results (e.g., pointwise boundedness). We need to prove that $R/x^{\mono{f}} \in \ideal{R/x^f}$. Consider the complex 
	\begin{equation}\label{eq:Ruparrow}
		R^\uparrow \coloneqq \bigoplus_{n \in \bbN} R[n]= (\cdots\to R\to R\to R \to 0\to \cdots)
	\end{equation}
	concentrated in non-negative degrees with zero differentials. From \cref{rem:Htensor}, we see that 
	\[
		H_n(R/x^f\otimes R^{\uparrow}) = \bigoplus_{i=0}^n R/x^{f(i)}.
	\]
	In particular, this contains $R/x^{\mono{f}(n)}$ as a direct summand. Thus $R/x^f \otimes R^{\uparrow}$ contains $R/x^{\mono{f}}$ as a direct summand.
\end{proof}

\section{Constructing thick ideals}\label{sec:thick}

We would like to have some control over the thick ideal $\ideal{X}$ generated by a pseudo-coherent complex $X \in \Dps(R)$. Clarity will be gained by abstracting our method as follows:

\begin{Hyp}\label{hyp:Kn}
	Let $\cat K$ be a tensor-triangulated category which is equipped with subcategories $\{\cat K_{\ge n}\}_{n \in \bbZ}$ satisfying
	\[
		\Sigma \cat K_{\ge n} \subseteq \cat K_{\ge n+1} \subseteq \cat K_{\ge n} \quad\text{ and }\quad\Sigma^{-1}\cat K_{\ge n} \subseteq \cat K_{\ge n-1}
	\]
	for each $n \in \bbZ$. We make the significant additional hypothesis that for every object $x \in \cat K$, there is some $n \in \bbZ$ such that $x \in \cat K_{\ge n}$. Finally, we assume $\unit \in \cat K_{\ge 0}$ although this is just for convenience.
\end{Hyp}

\begin{Exa}\label{exa:Kexa}
	We could take $\cat K=\Dps(R)$ with $\cat K_{\ge n}$ consisting of those $X \in \Dps(R)$ with $H_i(X) = 0$ for $i<n$.
\end{Exa}

\begin{Cons}\label{cons:Cs}
	Let $\cat C$ be a full subcategory of $\cat K$ and $s \in \bbZ$. We write
	\begin{align*}
		\cat C(s) \coloneqq \SET{z}{ z \text{ is a } &\text{$\oplus$-summand} \text{ of $z'$ and there exists an exact triangle} \\ 
		 & a \to z' \to b\otimes w \to \Sigma a 
	   \text{ with } a,b\in \cat C \text{ and } w \in \cat K_{\ge s}}.
	\end{align*}
\end{Cons}

\begin{Rem}\label{rem:Ctensor}
	If $0 \in \cat C$ then $\cat C(s)$ contains all $\oplus$-summands of objects in $\cat C$ and we also have $\cat C \otimes \cat K_{\ge s} \subseteq \cat C(s)$.
\end{Rem}

\begin{Rem}
	Note that $\Sigma^k \unit \in \cat K_{\ge k}$ for any $k \in \bbZ$. So if $0 \in \cat C$ then $\Sigma^k \cat C \subseteq \cat C(k)$.
\end{Rem}

\begin{Rem}
	If $s \le t$ then $\cat C(t) \subseteq \cat C(s)$ simply because $\cat K_{\ge t} \subseteq \cat K_{\ge s}$.
\end{Rem}

\begin{Rem}
	It follows from $\unit \in \cat K_{\ge 0}$ that $\cat C \oplus \cat C \subseteq \cat C(0)$.
\end{Rem}

\begin{Rem}\label{rem:Ccone}
	If $a \to b \to c \to \Sigma a$ is an exact triangle in $\cat K$ with $a,b \in \cat C$ then~$c \in \cat C(1)$. Indeed, by rotation we have an exact triangle
	\[
		b \to c \to \Sigma a \to \Sigma b
	\]
	and $\Sigma a \simeq a \otimes \Sigma \unit$ with $\Sigma \unit \in \cat K_{\ge 1}$.
\end{Rem}

\begin{Not}
	For a sequence $s_1,s_2,\ldots,s_N$ of integers, write 
	\[
		\cat C(s_1,s_2,\ldots,s_N) \coloneqq (\cat C(s_1,s_2,\ldots,s_{N-1}))(s_N).
	\]
\end{Not}

\begin{Rem}
	If $\cat C \subseteq \cat D$ are full subcategories of $\cat K$ then certainly \[\cat C(s_1,\ldots,s_N) \subseteq \cat D(s_1,\ldots,s_N).\]
\end{Rem}

With the above observations in hand, we can describe the thick ideal generated by an object as follows:

\begin{Prop}\label{prop:Cs}
	Let $x \in \cat K$ and let $\cat C\coloneqq \{ \text{all $\oplus$-summands of $x$}\}$. Then
	\begin{equation}\label{eq:generated}
		\ideal{x} = \bigcup_{\stackrel{N \ge 1,}{(s_1,\ldots,s_N)\in\bbZ^N}} \cat C(s_1,\ldots,s_N).
	\end{equation}
\end{Prop}

\begin{proof}
	We claim that the right-hand side of \eqref{eq:generated} is a thick ideal of $\cat K$. It is a full replete subcategory which contains zero. It follows from our observations (and our hypothesis that every object is ``bounded below'') that it is closed under~$-\otimes \cat K$ and, in particular, closed under the suspension and desuspension --- just observe that if $a \in\cat C(s_1,\ldots,s_N)$ and $b \in \cat K_{\ge n}$ then $a \otimes b \in \cat C(s_1,\ldots,s_N,n)$ by \cref{rem:Ctensor}. It also follows from our observations that it is closed under direct summands. Moreover, it is closed under cones. For this last claim, suppose $a \to b \to c \to \Sigma a$ is an exact triangle in $\cat K$ with $a$ and $b$ contained in the right-hand side of \eqref{eq:generated}. So $a \in \cat C(s_1,\ldots,s_N)$ for some $s_1,\ldots,s_N$ and $b \in \cat C(t_1,\ldots,t_M)$ for some $t_1,\ldots,t_M$. Then certainly $a \in \cat C(s_1,\ldots,s_N,t_1,\ldots,t_M)$ but also $b \in \cat C(s_1,\ldots,s_N,t_1,\ldots,t_M)$ since $\cat C\subseteq \cat C(s_1,\ldots, s_N)$ implies $\cat C(t_1,\ldots,t_M) \subseteq \cat C(s_1,\ldots,s_N,t_1,\ldots,t_M)$. Hence $c \in \cat C(s_1,\ldots,s_N,t_1,\ldots,t_M,1)$ by \cref{rem:Ccone}.

	Thus, the right-hand side of \eqref{eq:generated} is a thick ideal of $\cat K$. It contains $x$ hence contains $\ideal{x}$. On the other hand, if $\cat I$ is any thick ideal then for any full subcategory~$\cat D$, if $\cat D \subseteq \cat I$ then $\cat D(s) \subseteq \cat I$ for any $s \in \bbZ$. Hence, if $\cat I$ is a thick ideal which contains~$x$ then $\cat C \subseteq\cat I$ and hence $\cat C(s_1,\ldots,s_N) \subseteq \cat I$ for any $s_1,\ldots,s_N$. We conclude that the right-hand side of \eqref{eq:generated} is contained in all such $\cat I$ and hence is contained in $\ideal{x}$.
\end{proof}

\begin{Cor}\label{cor:Cs}
	Let $x \in \cat K$ and let $\cat C\coloneqq \{ \text{all $\oplus$-summands of $x$}\}$. Then $y \in \ideal{x}$ if and only if there exist integers $s$ and $N$ such that $y \in \cat C(\underbrace{s,s,\ldots,s}_{N\text{ times}})$.
\end{Cor}

\begin{proof}
	The $(\Leftarrow)$ direction is clear from \cref{prop:Cs}. For the $(\Rightarrow)$ direction recall that if $s \le t$ then $\cat C(t) \subseteq \cat C(s)$. By \cref{prop:Cs}, $y \in \cat C(s_1,\ldots,s_N)$ for some $s_1,\ldots,s_N$ and we can just take $s\coloneqq \min(s_1,\ldots,s_N)$. 
\end{proof}

\begin{Hyp}
	We now take $\cat K\coloneqq \Dps(R)$ and let $\cat K_{\ge n}$ be the subcategory of complexes with homology concentrated in degrees $\ge n$ as in \cref{exa:Kexa}. Let~${x \in \Dpfl(R)}$ and let $\cat C\coloneqq \{ \text{all $\oplus$-summands of $x$}\}$.
\end{Hyp}

\begin{Lem}\label{lem:Csloew}
	In the notation of \cref{cons:Cs}, we have
	\[
		\loew(H_n(z)) \le \loew(H_n(z')) \le \loew(H_n(a)) + \loew(H_n(b\otimes w))
	\]
	and
	\[
		\loew(H_n(b\otimes w)) \le \max_{-\infty< i \le n-s} \loew(H_i(b))
	\]
	for all $n \in \bbZ$.
\end{Lem}

\begin{proof}
	The first statement follows from \cref{lem:loew} and \cref{lem:loewSES} together with \cref{rem:hereditary-splitting}. For the second statement, \cref{rem:Htensor} and \cref{lem:loew} imply that
	\[
		\loew(H_n(b\otimes w)) =\max_{n-1 \le i+j \le n} \min(\loew(H_i(b)),\loew(H_j(w))).
	\]
	Note that $w \in \cat K_{\ge s}$ so $H_j(w)=0$ for $j<s$. It then follows that
	\[
		\loew(H_n(b\otimes w)) \le \max_{-\infty < i \le n-s}\loew(H_i(b))
	\]
	as desired.
\end{proof}

\begin{Rem}\label{rem:dominates}
	Suppose $f:\bbZ \to \bbN$ is a monotonic function which ``dominates'' $\cat C$ in the sense that for each $y \in \cat C$, we have $\loew(H_n(y))\leq f(n)$ for all $n \in \bbZ$. It then follows from \cref{lem:Csloew} that, in the notation of \cref{cons:Cs}, we have
	\begin{align*}
		\loew(H_n(z)) &\leq \loew(H_n(a)) + \max_{-\infty<i \le n-s} \loew(H_i(b))\\
					  &\le f(n) + \max_{-\infty < i \le n-s} f(i)\\
					  &\le f(n)+f(n-s)
	\end{align*}
	for each $n \in \bbZ$. If we then define $f'$ by $f'(n) \coloneqq f(n)+f(n-s)$ then $f'$ is a monotonic function which dominates $\cat C(s)$.
\end{Rem}

\begin{Exa}\label{exa:recurse-dominate}
	Let $s_1,s_2,\ldots$ be a sequence of integers and suppose $f_0:\bbZ \to \bbN$ is a monotonic function which dominates $\cat C_0\coloneqq \cat C$. Then for each $k \ge 1$, the recursively defined function $f_k(n) \coloneqq f_{k-1}(n) + f_{k-1}(n-s_k)$ is a monotonic function which dominates $\cat C_k\coloneqq \cat C(s_1,\ldots,s_k)$. We will utilize this idea in the proof of \cref{thm:loew-thick} below.
\end{Exa}

\begin{Rem}
	Recall from \cref{def:sigma-mu-preorders} that for monotonic sequences $f,g:\bbN \to \bbN$, we defined $f \le_\sigma g$ to mean that $f \le \sigmakg$ for some positive integer $k$. This in turn means that there is a constant $A$ such that $f(n) \le A g(n+k)$ for $n \gg 0$. Although we only defined it for monotonic sequences, the definition of $f \le_\sigma g$ evidently makes sense for arbitrary sequences $f,g:\bbN\to\bbN$. Similarly, the definition of $f \le_\mu g$ also makes sense for arbitrary sequences. The following will be helpful to bear in mind:
\end{Rem}

\begin{Lem}\label{lem:force-mono}
	Suppose $g$ is a monotonic sequence with $g \neq \underline{0}$. Then 
	\begin{enumerate}
		\item $f \le g$ if and only if $\mono{f} \le g$.
		\item $f \le_\sigma g$ if and only if $\mono{f} \le_\sigma g$.
		\item $f \le_\mu g$ if and only if $\mono{f} \le_\mu g$.
	\end{enumerate}
\end{Lem}

\begin{proof}
	The $(\Leftarrow)$ directions are immediate since $f \le \mono{f}$ pointwise. Moreover, parts~$(b)$ and $(c)$ follow from part $(a)$ since if $g \neq \underline{0}$ then $\sigma^k g$ and $\mu^k g$ are also nonzero for all $k\ge 1$. It thus suffices to prove that if $f \le g$ then $\mono{f} \le g$. By definition, $f \le g$ means that there exist positive integers~$A$ and $n_0$ such that $f(n) \le Ag(n)$ for all $n \ge n_0$. Without loss of generality, we can choose $n_0$ large enough so that $g(n_0)\neq 0$. Let $B\coloneqq \max\{\mono{f}(n_0),A\}$. Then for any $n \ge n_0$ we have
	\begin{align*}
		\mono{f}(n) &= \max\{\mono{f}(n_0),\max_{n_0 \le i \le n} f(i)\} \\
		&\le \max\{\mono{f}(n_0),\max_{n_0 \le i \le n} Ag(i)\}\\
		&= \max\{\mono{f}(n_0),Ag(n)\}\\
		&\le \max\{ B, Bg(n)\}\\
		&\le Bg(n)
	\end{align*}
	which establishes that $\mono{f} \le g$.
\end{proof}

\begin{Rem}
	The statement of \cref{lem:force-mono} is not true for $g=\underline{0}$ since $f \le \underline{0}$ is equivalent to $f$ converging to zero whereas $\mono{f} \le \underline{0}$ is equivalent to $f$ being identically zero.
\end{Rem}

\begin{Thm}\label{thm:loew-thick}
	Let $X \in \Dpfl(R)$ and $Y \in \Dps(R)$. If $Y \in \ideal{X}$ then $\mono{\loew_Y} \le_{\sigma} \mono{\loew_X}$.
\end{Thm}

\begin{proof}
	Note that if $Y \in \ideal{X}$ then $Y \in \Dpfl(R)$ by \cref{rem:Dpfl-largest} since $X \in \Dpfl(R)$. Thus $\loew_Y$ is well-defined. Define $f_0 \in \bbZ \to \bbN$ by 
	\[
		f_0(n) \coloneqq \begin{cases}
		\max\limits_{\inf(X) \le i \le n} \loew(H_i(X)) & \text{if } n \ge \inf(X) \\
		0 & \text{if } n < \inf(X).
		\end{cases}
	\]
	Note that $f_0(n) = \mono{\loew_X}(n-\inf(X))$ for $n \ge \inf(X)$ according to \cref{exa:to-mono} and the conventions of \cref{def:loewy-seq}. In any case, $f_0$ is a monotonic function which dominates $\cat C\coloneqq \{ \text{all $\oplus$-summands of $X$}\}$ in the sense of \cref{rem:dominates}. By \cref{cor:Cs}, the hypothesis $Y \in \ideal{X}$ implies that there are integers $N$ and $s$ such that $Y \in \cat C(\underbrace{s,s,\ldots,s}_{N\text{ times}})$. Inductively define $f_k$ for $k\ge 1$ by
	\[
		f_k(n) \coloneqq f_{k-1}(n) + f_{k-1}(n-s).
	\]
	By \cref{rem:dominates} and \cref{exa:recurse-dominate}, for any $E \in \cat C(\underbrace{s,s,\ldots,s}_{k\text{ times}})$ we have
	\[
		\loew(H_n(E)) \le f_k(n)
	\]
	for all $n \in \bbZ$. Moreover, we can prove inductively that
	\[
		f_k(n) \le 2^k f_0(n+k|s|)
	\]
	for all $n \in \bbZ$ and $k \ge 0$. This establishes that 
	\[
		\loew(H_n(Y)) \le 2^N f_0(n+N|s|)
	\]
	for all $n \in \bbZ$. In particular, recalling \cref{def:loewy-seq}, we have
	\begin{align*}
		\loew_Y(n) &= \loew(H_{n+\inf(Y)}(Y)) \\ 
				   &\le 2^N f_0(n+\inf(Y) + N|s|) \\
				   &= 2^N \mono{\loew_X}(n-\inf(X)+\inf(Y)+N|s|)
	\end{align*}
	for all $n$ large enough so that $n+\inf(Y)+N|s| \ge \inf(X)$. Setting 
	\[
		t \coloneqq \max\{0,-\inf(X)+\inf(Y)+N|s|\}
	\]
	we thus have
	\[
		\loew_Y(n) \le 2^N \sigma^t \mono{\loew_X}(n)
	\]
	for all sufficiently large $n$. That is, we have the asymptotic bound $\loew_Y \le \sigma^t \mono{\loew_X}$. This proves that $\loew_Y \le_\sigma \mono{\loew_X}$. If $\mono{\loew_X} \neq \underline{0}$ then \cref{lem:force-mono} implies $\mono{\loew_Y} \le_{\sigma} \mono{\loew_X}$. On the other hand, if $\mono{\loew_X}=\underline{0}$ then $X=0$ so that $Y=0$ and hence $\mono{\loew_Y}=\underline{0}$ too.
\end{proof}

\begin{Rem}\label{rem:loew-pref}
	According to \cref{def:loewy-seq}, we have $\loew_X(n) = \loew(H_{n+\inf(X)})$ for all~$n \in \bbN$. An alternative approach would be to consider the function $\loew_X':\bbN\to\bbN$ defined by 
	\begin{equation}
		\loew_X'(n) \coloneqq \loew(H_n(X))
	\end{equation}
	for each $n\in \bbN$. Note that this function only captures the Loewy lengths of the homology groups in non-negative degrees. It is a matter of preference which function is deemed more natural (bearing in mind \cref{rem:truncation}). For the purposes of exposition, either choice has pros and cons. In any case:
\end{Rem}

\begin{Lem}\label{lem:two-defs-equiv}
	For any $X \in \Dpfl(R)$, the monotonic functions $\mono{\loew_X}$ and $\mono{\loew_X'}$ are $\sigma$-equivalent.
\end{Lem}

\begin{proof}
	This is a routine verification from the definitions and \cref{lem:force-mono}(b).
\end{proof}

\begin{Exa}\label{exa:loew_Rxf}
	If $f:\bbN\to\bbN$ is a monotonic sequence then $f \sim_\sigma \loew_{R/x^f}$.
\end{Exa}

\begin{Rem}\label{rem:useful-bound-relations}
	It will be helpful to bear in mind that if $f \sim_\sigma f'$ and $g \sim_\sigma g'$ then $f\le_\sigma g$ if and only if $f' \le_\sigma g'$. Moreover, note that if $g$ is $\sigma$-stable (\cref{def:sigma-mu-stable}) then $f \le_\sigma g$ if and only if $f \le g$.
\end{Rem}

\begin{samepage}
\begin{Cor}\label{cor:thickRxf}
	Let $f,g : \bbN \to \bbN$ be monotonic sequences. The following are equivalent:
	\begin{enumerate}
		\item $R/x^f$ is contained in the thick subcategory generated by $R/x^g$.
		\item $R/x^f$ is contained in the thick ideal generated by $R/x^g$.
		\item $f \le_\sigma g$.
	\end{enumerate}
\end{Cor}
\end{samepage}

\begin{proof}
	The implication $(a)\Rightarrow (b)$ is immediate.

	$(b) \Rightarrow (c)$:
	Since $f$ and $g$ are monotonic, we have $\loew_{R/x^f} \le_\sigma \loew_{R/x^g}$ by \cref{thm:loew-thick}. Hence $f \le_\sigma g$ by \cref{exa:loew_Rxf}.

	$(c)\Rightarrow (a)$:
	By definition, $(c)$ states that $f \le \sigmakg$ for some $k \ge 1$. \Cref{prop:asym-direct} then implies that $R/x^f \in \thicksub{R/x^{\sigmakg}}$. On the other hand, $R/x^{\sigmakg}$ is a direct summand of $R/x^g[-k]$, namely it is the truncation $(R/x^g[-k])_{\ge 0}$ of \cref{rem:truncation}. Thus $R/x^f \in \thicksub{R/x^g}$.
\end{proof}

\begin{Rem}
	It is natural to ask whether the converse of \cref{thm:loew-thick} holds. That is, to ask whether $\ideal{X} = \SET{Y\in \Dpfl(R)}{\mono{\loew_Y} \le_{\sigma} \mono{\loew_X}}$ for any $X \in \Dpfl(R)$. In particular, do we have $\ideal{R/x^f} = \SET{Y \in \Dpfl(R)}{\mono{\loew_Y} \le_{\sigma} f}$ for any monotonic sequence $f$? To explore these questions we will need a new idea explored in the next section; see \cref{thm:f-explodable} and \cref{cor:summary-ideal-Rxf}.
\end{Rem}

\section{Explodable complexes}\label{sec:explodable}

We are still considering a discrete valuation ring $R$.

\begin{Cons}
	Let $a: \bbN\to \bbN$ be a function. For any complex $X \in \Der(R)$, define the complex
	\[
		X^{\oplus a} \coloneqq \bigoplus_{n \in \bbN} H_n(X)^{\oplus a(n)}[n]
	\]
	which has zero differentials. For example, if $X \in \Der_{\ge 0}(R)$ is concentrated in non-negative degrees then $X^{\oplus \underline{1}} \cong X$.
\end{Cons}

\begin{Def}\label{def:explodable}
	A pseudo-coherent complex $X$ is said to be \emph{explodable} if $X^{\oplus a} \in \ideal{X}$ for every function $a:\bbN \to \bbN$.
\end{Def}

\begin{Exa}
	Every bounded pseudo-coherent complex is explodable. Indeed, if~$X\neq0$ is bounded then $X^{\oplus a}$ is also bounded for any $a$ and hence is contained in the thick subcategory generated by $X$ by \cref{rem:truncation} and \cref{rem:Dpfl-largest}. Moreover, the statement is trivially true for $X=0$.
\end{Exa}

\begin{Rem}
	The above construction is most naturally applied to complexes which are concentrated in non-negative degrees. Indeed, $X^{\oplus a} = (X_{\ge 0})^{\oplus a}$ where $X_{\ge 0}$ is the truncation of $X$. It follows from this observation and \cref{rem:truncation} that $X$ is explodable if and only if its truncation $X_{\ge 0}$ is explodable.
\end{Rem}

\begin{Rem}\label{rem:a-force-mono}
	If $\mono{a}$ is the monotonic sequence associated to $a$ (as in \cref{exa:to-mono}) then~$X^{\oplus a}$ is a direct summand of $X^{\oplus \mono{a}}$. It follows that $X$ is explodable if and only if $X^{\oplus a} \in \ideal{X}$ for every \emph{monotonic} sequence $a:\bbN\to\bbN$.
\end{Rem}

\begin{Lem}\label{lem:sigma-explodable}
	Let $f:\bbN\to\bbN$ be monotonic. If the complex $R/x^f$ is explodable then the complex $R/x^{\sigmakf}$ is explodable for all $k \ge 1$.
\end{Lem}

\begin{proof}
	The statement is trivial if $f=\underline{0}$ so suppose $f \neq \underline{0}$. This implies $\sigmakf(n) \neq 0$ for $n \gg 0$ since $f$ is monotonic. Consider any monotonic function $a:\bbN\to\bbN$. By \cref{rem:a-force-mono}, it suffices to prove
	\[
		(R/x^{\sigmakf})^{\oplus a} \in \ideal{R/x^{\sigmakf}}.
	\]
	This is trivial if $a = \underline{0}$ so assume $a \neq \underline{0}$. Hence $a(n) \neq 0$ for $n \gg 0$ since $a$ is monotonic. Next note that $a=\sigma^k b$ for the function $b:\bbN\to\bbN$ defined by $b(n)=1$ for $n<k$ and $b(n)=a(n-k)$ for $n \ge k$. Observe that 
	\begin{equation}\label{eq:sigma-explodable-eq}
		(R/x^f)^{\oplus b}[-k] \simeq (R/x^{\sigmakf})^{\oplus a} \oplus E
	\end{equation}
	where $E \in \Dbfl(R)$. Since $(R/x^{\sigmakf})^{\oplus a} \neq 0$, we have $E \in \thicksub{(R/x^{\sigmakf})^{\oplus a}}$ by \cref{rem:truncation}. Thus \eqref{eq:sigma-explodable-eq} implies that
	\[
		\thicksub{(R/x^f)^{\oplus b}}
		=\thicksub{(R/x^{\sigmakf})^{\oplus a}}.
	\]
	In particular, applied to $a=b=\underline{1}$, we have
	\[
		\thicksub{R/x^f} = \thicksub{R/x^{\sigmakf}}.
	\]
	Thus, if $R/x^f$ is explodable then
	\[
		(R/x^{\sigmakf})^{\oplus a} \in \ideal{(R/x^f)^{\oplus b}} \subseteq \ideal{R/x^f} = \ideal{R/x^{\sigmakf}}
	\]
	as desired.
\end{proof}

\begin{Lem}\label{lem:Af-explodable}
	Let $f:\bbN \to \bbN$ be a function. For any constant $A$ and function $a:\bbN\to\bbN$, the complex $(R/x^{Af})^{\oplus a}$ is contained in the thick subcategory generated by $(R/x^f)^{\oplus a}$.
\end{Lem}

\begin{proof}
	For each $n \in \bbN$, we have a short exact sequence
	\[
		0 \to (R/x^{f(n)})^{\oplus a(n)} \to (R/x^{2f(n)})^{\oplus a(n)} \to (R/x^{f(n)})^{\oplus a(n)} \to 0
	\]
	by \cref{lem:tensor-cyclic}(c), which assemble into a short exact sequence of complexes
	\[
		0 \to (R/x^f)^{\oplus a} \to (R/x^{2f})^{\oplus a} \to (R/x^f)^{\oplus a} \to 0.
	\]
	Choose an $N$ such that $2^N \ge A$. Thus, with $N$ extensions we can build $(R/x^{2^N\hspace{-0.3ex}f})^{\oplus a}$. This has an endomorphism whose cone contains $(R/x^{Af})^{\oplus a}$ as a direct summand; see \cref{lem:cone}.
\end{proof}

\begin{Lem}\label{lem:Rxfexplodable}
	Let $X \in \Dzfl(R)$ and let $f:\bbN\to \bbN$ be a monotonic function such that $\loew(H_n(X)) \le f(n)$ for all $n \in \bbN$. Then $X \in \thicksub{(R/x^f)^{\oplus a}}$ where $a(n)$ denotes the number of indecomposable summands in $H_n(X)$.
\end{Lem}

\begin{proof}
	Suppose $M$ is a finite length $R$-module with decomposition
	\[
		M\simeq R/x^{k_1} \oplus R/x^{k_2} \oplus \cdots \oplus R/x^{k_m}.
	\]
	It follows from \cref{lem:cone} that if $\ell \ge \loew(M) = \max\{k_i \mid 1\le i \le m\}$ then $M$ is a direct summand of the cone of the diagonal map
	\[
		(R/x^\ell)^{\oplus m} \to (R/x^{\ell})^{\oplus m}
	\]
	which in degree $i$ is multiplication by $x^{k_i}$. Since $f(n) \ge \loew(H_n(X))$, we have for each $n \in \bbN$ a map
	\[
		(R/x^{f(n)})^{\oplus a(n)} \xrightarrow{g_n} (R/x^{f(n)})^{\oplus a(n)}
	\]
	whose cone contains $H_n(X)$ as a direct summand. The cone of the morphism
	\[
		(R/x^{f})^{\oplus a} = \coprod_{n \ge 0} (R/x^{f(n)})^{\oplus a(n)}[n] \xrightarrow{\coprod_{n\ge 0}g_n[n]} \coprod_{n \ge 0} (R/x^{f(n)})^{\oplus a(n)}[n] = (R/x^{f})^{\oplus a}
	\]
	is $\coprod_{n \ge 0} \cone(g_n)[n]$ and hence contains $\coprod_{n \ge 0}H_n(X)[n]$ as a direct summand. Thus,~$X$ is contained in the thick subcategory generated by $(R/x^{f})^{\oplus a}$.
\end{proof}

\begin{Prop}\label{prop:RxfexplodableXX}
	Let $X \in \Dpfl(R)$ and let $f:\bbN\to\bbN$ be a monotonic function satisfying $\mono{\loew_X} \le_\sigma f$. If the complex $R/x^f$ is explodable then $X \in \ideal{R/x^f}$.
\end{Prop}

\begin{proof}
	The statement is trivially true if $X=0$ so suppose $X \neq 0$. Observe that $\loew_{X'} = \loew_X$ for $X' \coloneqq X[-\inf(X)] \in \Dzfl(R)$. Moreover, $X \in \thicksub{(R/x^f)^{\oplus a}}$ if and only if $X' \in \thicksub{(R/x^f)^{\oplus a}}$. Thus, replacing $X$ by $X'$ we may assume $X \in \Dzfl(R)$ with $\loew_X(n) = \loew(H_n(X))$ for all $n \ge 0$. By hypothesis, $\mono{\loew_X} \le \sigmakf$ for some $k \ge 1$. If $\underline{f} =\underline{0}$ then $\loew_X = \underline{0}$ so that $X = 0$. Thus, we may assume $f \neq \underline{0}$. By \cref{lem:sigma-explodable} and \cref{cor:thickRxf}, replacing $f$ by $\sigmakf$, we may assume without loss of generality that $\mono{\loew_X} \le f$. Thus there are constants $A$ and $n_0$ so that $\mono{\loew_X}(n) \le Af(n)$ for all $n \ge n_0$. Since $R/x^f \neq 0$, \cref{rem:truncation} implies that $X \in \ideal{R/x^f}$ if and only if $X_{\ge n_0} \in \ideal{R/x^f}$. Replacing $X$ by $X_{\ge n_0}$ we may assume, without loss of generality, that $\loew(H_n(X)) \le Af(n)$ for all $n \in\bbN$. Let $a:\bbN\to\bbN$ be the function with $a(n)$ equal to the number of indecomposable summands in $H_n(X)$. \Cref{lem:Rxfexplodable} then implies that $X$ is contained in the thick subcategory generated by $(R/x^{Af})^{\oplus a}$. \Cref{lem:Af-explodable} then implies that $X$ is contained in the thick subcategory generated by $(R/x^{f})^{\oplus a}$. If the complex $R/x^f$ is explodable then $(R/x^f)^{\oplus a} \in \ideal{R/x^f}$ so that~$X \in
\ideal{R/x^f}$. \end{proof}

\begin{Thm}\label{thm:f-explodable}
	Let $f$ be a monotonic sequence such that $R/x^f$ is explodable. The following hold:
	\begin{enumerate}
		\item $\ideal{R/x^f} = \SET{E \in \Dpfl(R)}{\mono{\loew_E}\le_\sigma f}$.
		\item $\ideal{R/x^f} = \ideal{E}$ for all $E \in \Dpfl(R)$ with $\mono{\loew_E} \sim_\sigma f$.
		\item Any complex $E$ satisfying $\mono{\loew_E} \sim_\sigma f$ is explodable.
	\end{enumerate}
\end{Thm}

\begin{proof}
	$(a)$: The $\subseteq$ inclusion follows from \cref{thm:loew-thick} bearing mind \cref{exa:loew_Rxf} and the $\supseteq$ inclusion follows from \cref{prop:RxfexplodableXX}.

	$(b)$: Part $(a)$ provides the $\supseteq$ inclusion. On the other hand, $E[-\inf(E)]$ contains~$R/x^{\loew_E}$ as a direct summand. It then follows from \cref{prop:ideal-force-mono} that the thick ideal $\ideal{E}$ contains $R/x^{\mono{\loew_E}}$.

	$(c)$: For any $a:\bbN\to\bbN$, $\mono{\loew_{E^{\oplus a}}} = \mono{\loew_E}$ so $(a)$ and $(b)$ imply that $E^{\oplus a} \in \ideal{E}$.
\end{proof}

\begin{Rem}
	We now want to understand for which monotonic sequences $f$ is the complex $R/x^f$ explodable. The following constructions are taken from \cite{MatsuiTakahashi17}:
\end{Rem}

\begin{Def}
	Let $E$ be a complex with zero differentials. Let $\Esplit$ denote the complex
	\[
		\cdots \to 0 \to E_2 \to 0 \to E_1\to 0 \to E_0\to 0\to \cdots
	\]
	with $E_i$ in degree $2i$, let $\Eeven$ denote the complex
	\[
		\cdots \to E_6 \to E_4 \to E_2 \to E_0 \to E_{-2} \to \cdots
	\]
	with $E_{2i}$ in degree $i$, and let $\Eodd$ denote the complex
	\[
		\cdots \to E_7 \to E_5 \to E_3 \to E_1 \to E_{-1} \to \cdots
	\]
	with $E_{2i+1}$ in degree $i$.
\end{Def}

Matsui--Takahashi asked whether a complex $E$ satisfies the following properties:

\begin{Def}
	We say that a pseudo-coherent complex $E$ satisfies property

	\smallskip
	\begin{enumerate}
		\itemindent=1em
		\setlength\itemsep{1ex}
		 \item[(MT1)] if $\Eeven,\Eodd \in \ideal{E}$;
		 \item[(MT2)] if $E \in \ideal{\Esplit}$.
	\end{enumerate}
\end{Def}

\begin{Prop}[Matsui--Takahashi]\label{prop:MT1-explodable}
	If (MT1) holds for $E$ then $E$ is explodable.
\end{Prop}

\begin{proof}
	This follows from \cite[Corollary~7.3]{MatsuiTakahashi17}.
\end{proof}

Our next task is to gain an understanding of when condition (MT1) holds.

\begin{Lem}\label{lem:MT1a}
	Let $f$ be a monotonic sequence. Condition (MT1) holds for the complex $R/x^f$ if and only if the following two conditions hold:
	\begin{enumerate}
		\item $\mu f \le_\sigma f$, and
		\item $\mu \sigma f \le_\sigma f$.
	\end{enumerate}
\end{Lem}

\begin{proof}
	For the complex $E=R/x^f$, we have $\Eeven = R/x^{\mu f}$ and $\Eodd = R/x^{\mu \sigma f}$. \Cref{cor:thickRxf} then implies that $\Eeven \in \ideal{E}$ if and only if $\mu f \le_\sigma f$ and $\Eodd \in \ideal{E}$ if and only if $\mu \sigma f \le_\sigma f$.
\end{proof}

\begin{Lem}\label{lem:MT1b}
	Let $f$ be a monotonic sequence. The following conditions are equivalent:
	\begin{enumerate}
		\item $\mu f \le_{\sigma} f$
		\item $\mu \sigma f \le_{\sigma} f$
		\item $f$ is $\mu$-stable.
	\end{enumerate}
\end{Lem}

\begin{proof}
	$(a)\Rightarrow (b)$: Observe that
	\[
		(\mu\sigma f)(n) = f(2n+1) \le f(2n+2) = \mu f(n+1) = (\sigma \mu f)(n)
	\]
	so $\mu\sigma f \le_\sigma \sigma \mu f \le_\sigma \sigma f \le_\sigma f$.

	$(b)\Rightarrow (c)$: We first establish that $f$ is $\sigma$-stable. By assumption, there exist positive integers $A$ and $k$ such that $f(2n+1) \le Af(n+k)$ for $n \gg 0$. This implies that for $n \gg 0$ even, $f(n+1) \le A f(n/2 +k) \le Af(n)$ while for $n \gg 0$ odd, $f(n+1)\le f(n+2) \le A f((n+1)/2 +k) \le Af(n)$. Thus, $\sigma f \le f$. It then follows using $(b)$ that
	\[
		\mu f \le \mu \sigma f \le \sigmakf \le f
	\]
	which demonstrates that $f$ is $\mu$-stable.

	$(c) \Rightarrow (a)$: This follows immediately from $\mu f \le f$.
\end{proof}

\begin{Prop}\label{prop:MT1-mu-stable}
	Let $f$ be a monotonic sequence. Condition (MT1) holds for the complex $E=R/x^f$ if and only if $f$ is $\mu$-stable.
\end{Prop}

\begin{proof}
	This immediately follows from \cref{lem:MT1a} and \cref{lem:MT1b}.
\end{proof}

\begin{Exa}
	The complex $R/x^f$ satisfies (MT1) for the polynomial~$f(n)=n^d$. Matsui--Takahashi \cite[Example~7.5]{MatsuiTakahashi17} established the $d=1$ case.
\end{Exa}

\begin{Exa}\label{exa:not-MT1}
	The proposition establishes that (MT1) does not hold in general. For example, it doesn't hold for the complex $R/x^f$ with $f(n)=2^n$, since this $f$ is not $\mu$-stable. This provides a negative answer to \cite[Question~7.4]{MatsuiTakahashi17}.
\end{Exa}

\begin{Cor}\label{cor:mu-stable-explodable}
	If $f$ is $\mu$-stable then the complex $R/x^f$ is explodable.
\end{Cor}

\begin{proof}
	This follows from \cref{prop:MT1-explodable} and \cref{prop:MT1-mu-stable}.
\end{proof}

Next we consider condition (MT2).

\begin{Rem}
	Note that $(R/x^f)_{\mathrm{split}} = R/x^{\fsplit}$ where $\fsplit:\bbN \to \bbN$ is defined by
	\[
		\fsplit(n) = \begin{cases}
		f(n/2) & \text{if $n$ is even}\\
		0 & \text{if $n$ is odd}.
		\end{cases}
	\]
	Observe that $\mono{\fsplit}(n) = f(\floor{n/2})$ for each $n \in \bbN$.
\end{Rem}

\begin{Lem}\label{lem:auto-sigma-stable}
	Let $f$ be a monotonic sequence. If $f \le_\sigma \mono{\fsplit}$ then $f$ is $\sigma$-stable.
\end{Lem}

\begin{proof}
	By definition, $f \le_\sigma \mono{\fsplit}$ means that there is an $A$, $k$ and $n_0$ such that $f(n) \le A \mono{\fsplit}(n+k)$ for $n \ge n_0$. Hence, for $n \ge \max(n_0,k+1)$ we have $f(n+1) \le A\mono{\fsplit}(n+k+1) \le A\mono{\fsplit}(2n) = A f(n)$.
\end{proof}

\begin{Prop}\label{prop:MT2-mu-stable}
	Let $f$ be a monotonic sequence. Condition (MT2) holds for the complex $E=R/x^f$ if and only if $f$ is $\mu$-stable.
\end{Prop}

\begin{proof}
	It follows from \cref{prop:ideal-force-mono} and \cref{cor:thickRxf} that $E \in \ideal{\Esplit}$ if and only if $f \le_\sigma \mono{\fsplit}$. In other words, (MT2) holds if and only if $f \le_\sigma \mono{\fsplit}$. One readily checks that $f \le_\sigma \mono{\fsplit}$ implies $\mu f \le_\sigma f$. Since $f$ is also $\sigma$-stable (by \cref{lem:auto-sigma-stable}), this implies $\mu f \le f$, so that $f$ is $\mu$-stable. Conversely, suppose $f$ is $\mu$-stable. Thus there exist constants $A$ and $n_0$ such that $f(2n) \le A f(n)$ for $n \ge n_0$. Consider an $n > 2n_0 +1$. If $n$ is even then $f(n) \le Af(n/2)$ while if $n$ is odd then $f(n) \le f(n+1) \le Af(\floor{n/2}+1) \le A f(2\floor{n/2}) \le A^2f(\floor{n/2})$. Thus we have $f(n) \le A^2f(\floor{n/2}) = A^2 \mono{\fsplit}(n)$ for all $n > 2n_0 +1$. So that $f \le \mono{\fsplit}$ which of course implies $f \le_\sigma \mono{\fsplit}$.
\end{proof}

\begin{Thm} \label{thm:MT1-MT2-mustable}
	If $f$ is a monotonic sequence, then the following are equivalent:
	\begin{enumerate}
		\item (MT1) holds for $R/x^f$;
		\item (MT2) holds for $R/x^f$;
		\item $f$ is $\mu$-stable.
	\end{enumerate}
\end{Thm}

\begin{proof}
	This follows from \cref{prop:MT1-mu-stable} and \cref{prop:MT2-mu-stable}.
\end{proof}

\begin{Cor}\label{cor:summary-ideal-Rxf}
	If $f$ is a $\mu$-stable monotonic sequence then
	\[
		\ideal{R/x^f} 
		= \SET{E \in \Dpfl(R)}{\mono{\loew_E}\le_\sigma f}
		= \SET{E \in \Dpfl(R)}{\mono{\loew_E}\le f}.
	\]
\end{Cor}

\begin{proof}
	This follows from \cref{thm:MT1-MT2-mustable}, \cref{prop:MT1-explodable} and \cref{thm:f-explodable} bearing in mind that, since $f$ is $\mu$-stable, we have $g \le f \Leftrightarrow g \le_\sigma f \Leftrightarrow g \le_\mu f$ for any monotonic sequence $g$.
\end{proof}

\section{Convolutions and radical ideals}\label{sec:convolutions}

In the previous sections, we have established several results concerning the thick ideal $\ideal{E}$ generated by a pseudo-coherent complex $E$ with a particular emphasis on complexes of the form $E=R/x^f$. The strongest results were obtained for complexes which are explodable (\cref{def:explodable}). In this section, we will provide more satisfactory results concerning the \emph{radical} thick ideal $\radideal{E}$ generated by a pseudo-coherent complex $E$. The results we obtain for $\radideal{E}$ are more complete than those we have established for $\ideal{E}$. One reason for this is that while we do not know if the complex $R/x^f$ is always explodable, we can prove that it is always ``explodable up to tensor powers'' meaning that $(R/x^f)^{\oplus a} \in \radideal{R/x^f}$ for any sequence $a$; see \cref{prop:almost-explodable}. This phenomenon leads to \cref{thm:rid-full} which explicitly describes the principal radical ideal $\radideal{R/x^f}$ for any monotonic sequence $f$.

\begin{Def}
	We define the \emph{convolution} $f\ast g$ of two sequences $f,g:\bbN\to\bbN$ by
	\[
		(f\ast g)(n) \coloneqq \max_{0\le i\le n}\min(f(i),g(n-i)).
	\]
\end{Def}

\begin{Rem}\label{rem:conv-properties}
	The following facts are readily verified:
	\begin{enumerate}
		\item Convolution is commutative: $f\ast g = g \ast f$ pointwise.
		\item Convolution is associative: $f\ast (g\ast h) = (f\ast g)\ast h$ pointwise.\footnote{Both $(f\ast (g\ast h))(n)$ and $( (f\ast g)\ast h)(n)$ are the maximum value the function~$\min(f(x),g(y),h(z))$ takes over the simplex $x+y+z=n$, $x,y,z\ge 0$.}
		\item $f\ast g$ is monotonic if either $f$ or $g$ is monotonic.
		\item $f\ast g \le f$ if $f$ is monotonic.
	\end{enumerate}
\end{Rem}

\begin{Lem}\label{lem:loew-conv}
	Let $f$ be a monotonic sequence with $f(0)\neq 0$ and let $E \in \Dpfl(R)$ be a nonzero complex with $\inf(E)=0$. Then
	\[
		\loew_{R/x^f \otimes E} = f \ast \loew_{E}.
	\]
\end{Lem}

\begin{proof}
	For notational simplicity, let $g(n) \coloneqq \loew_E(n) = \loew(H_n(E))$. It follows from \cref{rem:Htensor}, \cref{lem:loew} and \cref{rem:koszul-trick} that for any $n \in \bbN$ we have
	\begin{align*}
		(\loew_{R/x^f\otimes E})(n) &= \loew(H_n(R/x^f \otimes E)) \\
								  &= \max\big(\max_{i+j=n} \min(f(i),g(j)),\max_{i+j=n-1} \min(f(i),g(j))\big) \\
										&= \max\big( (f\ast g)(n), (f\ast g)(n-1)\big) \\
										&= (f\ast g)(n)
	\end{align*}
	since $f\ast g$ is monotonic (since $f$ is monotonic).
\end{proof}

\begin{Lem}\label{lem:fast}
	For any monotonic sequence $f$, we have 
	\[
		(f\ast f)(n) = f(\floor{n/2})
	\]
	for all $n \in \bbN$.
\end{Lem}

\begin{proof}
	For $0 \le i \le n$, define $g(i) \coloneqq f(n-i)$. Thus, 
	\[
		(f \ast f)(n) = \max_{0\le i \le n} \min(f(i),g(i)).
	\]
	Observe that while $f$ is monotonically increasing, the function $g$ is monotonically decreasing on the interval $0,\ldots,n$. Moreover, note that the ``graphs'' of $f$ and $g$ are symmetric about $x=n/2$. More precisely, for any $0 \le k \le \floor{n/2}$, we have
	\[
		g(\floor{n/2}-k) = f(\ceil{n/2}+k).
	\]
	Note that for such $k$ we have $\floor{n/2}-k \le \ceil{n/2}+k$. Hence $g(\floor{n/2}-k) = f(\ceil{n/2}+k) \ge f(\floor{n/2}-k)$. It follows that 
	\[
		\max_{0 \le i \le \floor{n/2}} \min(f(i),g(i)) = \max_{0 \le i \le \floor{n/2}} f(i) = f(\floor{n/2}).
	\]
	On the other hand, for such $k$ we also have $f(\ceil{n/2}+k) = g(\ceil{n/2}-i) \ge g(\ceil{n/2}+k)$. It follows that 
	\[
		\max_{\ceil{n/2} \le i \le n} \min(f(i),g(i)) = \max_{\ceil{n/2}\le i \le n} g(i) = g(\ceil{n/2}).
	\]
	Finally, just observe that $g(\ceil{n/2}) = f(\floor{n/2})$.
\end{proof}

\begin{Lem}\label{lem:convmu}
	Let $f$ and $g$ be monotonic sequences. Then $f \ast g \le f\wedge g \le \mu(f\ast g)$ pointwise. Hence $f\ast g$ and $f \wedge g$ are $\mu$-equivalent.
\end{Lem}

\begin{proof}
	The first inequality $f \ast g \le f \wedge g$ holds by \cref{rem:conv-properties}(d). On the other hand, we have
	\[
		(f \wedge g)(n) = \min(f(n),g(n)) \le \max_{i+j=2n} \min(f(i),g(j)) = \mu(f\ast g)(n)
	\]
	for all $n\in \bbN$. It follows that $f\ast g$ and $f \wedge g$ are $\mu$-equivalent (\cref{def:sigma-mu-preorders}).
\end{proof}

\begin{Lem}\label{lem:mufast}
	For any monotonic sequence $f$, we have $\mu(f\ast f)=f$ pointwise.
\end{Lem}

\begin{proof}
	This immediately follows from \cref{lem:fast}
\end{proof}

\begin{Lem}\label{lem:fastcons}
	Let $f$ and $g$ be monotonic sequences. If $f \le \mu g$ then $f \ast f \le g$.
\end{Lem}

\begin{proof}
	This is a routine consequence of \cref{lem:fast}.
\end{proof}

\begin{Lem}\label{lem:almost-explodable3}
	For any monotonic sequences $f$ and $a$, the complex $((R/x^f)^{\oplus a})^{\otimes 2}$ is contained in the thick subcategory generated by the pseudo-coherent complex with zero differentials $E \in \Dzfl(R)$ defined by
	\begin{equation}\label{eq:Edef}
		E_n \coloneqq \bigoplus_{0 \le i \le \floor{n/2}} (R/x^{f(i)})^{\oplus {a(i)a(n-i)}}
	\end{equation}
	for all $n \in \bbN$.
\end{Lem}

\begin{proof}
	By \cref{rem:Htensor} and \cref{lem:tensor-cyclic} we have
	\begin{align}\label{eq:Rxfta}
		H_n((R/x^f)^{\oplus a} \otimes (R/x^f)^{\oplus a}) &= \bigoplus_{0 \le i \le n} (R/x^{\min(f(i),f(n-i))})^{\oplus a(i)a(n-i)} \\
															&\oplus \bigoplus_{0 \le i \le n-1} (R/x^{\min(f(i),f(n-1-i))})^{\oplus a(i)a(n-1-i)}\nonumber
	\end{align}
	for each $n \in \bbN$. As explained in the proof of \cref{lem:fast}, 
	\[
		\min(f(i),f(n-i)) = \begin{cases}
		f(i) & \text{for } 0 \le i \le \floor{n/2}\\
		f(n-i) & \text{for } \ceil{n/2} \le i \le n.
		\end{cases}
	\]
	It follows that the first term of \eqref{eq:Rxfta} coincides with 
	\begin{equation*}
		\bigoplus_{0 \le i \le \floor{n/2}} (R/x^{f(i)})^{\oplus a(i)a(n-i)} \oplus \bigoplus_{\ceil{n/2} \le i \le n} (R/x^{f(n-i)})^{\oplus a(i)a(n-i)}
	\end{equation*}
	which in turn coincides with
	\begin{equation}\label{eq:blah}
		\left(\bigoplus_{0 \le i \le \floor{n/2}} (R/x^{f(i)})^{\oplus a(i)a(n-i)}\right)
		\oplus
		\left(\bigoplus_{0 \le i \le \floor{n/2}} (R/x^{f(i)})^{\oplus a(i)a(n-i)}\right).
	\end{equation}
	One similarly shows that the second term of \eqref{eq:Rxfta} coincides with the analogue of \eqref{eq:blah} in which $n$ is replaced by $n-1$; since $a$ is monotonic, this is a direct summand of \eqref{eq:blah}. In summary, $(R/x^f)^{\oplus a} \otimes (R/x^f)^{\oplus a}$ is a direct summand of~$E^{\oplus 4}$ where $E \in \Dzfl(R)$ is the complex with zero differentials defined by 
	\[
		E_n = \bigoplus_{0 \le i \le \floor{n/2}} (R/x^{f(i)})^{a(i)a(n-i)}.
	\]
	Thus, $(R/x^f)^{\oplus a} \otimes (R/x^f)^{\oplus a}$ is contained in $\thicksub{E^{\oplus 4}} = \thicksub{E}$.
\end{proof}

\begin{Prop}\label{prop:tensast}
	For any monotonic sequence $f$, we have
	\[
		(R/x^f)^{\otimes 2} \in \ideal{R/x^{f \ast f}}.
	\]
\end{Prop}

\begin{proof}
	By \cref{lem:almost-explodable3}, it suffices to prove that $E \in \ideal{R/x^{f \ast f}}$ where $E\in \Dzfl(R)$ is the complex with zero differentials defined by
	\[
		E_n = \bigoplus_{0 \le i \le \floor{n/2}} R/x^{f(i)}
	\]
	for each $n \in \bbN$. Now recall the complex $R^{\uparrow}$ from \eqref{eq:Ruparrow}. For each $n \in \bbN$, we have
	\[
		H_n(R/x^{f \ast f} \otimes R^{\uparrow}) = \bigoplus_{0 \le i \le n} R/x^{(f \ast f)(i)}
		= \bigoplus_{0 \le i \le n} R/x^{f(\floor{i/2})}
	\]
	where the last equality is from \cref{lem:fast}. Focusing on the even $i$, we see that $E$ is a direct summand of $R/x^{f \ast f} \otimes R^{\uparrow}$. Hence $E \in \ideal{R/x^{f\ast f}}$.
\end{proof}

\begin{Cor}\label{cor:tensastiterated}
	For any monotonic sequence $f$, we have
	\[
		(R/x^f)^{\otimes 2^k} \in \ideal{R/x^{f^{\ast 2^k}}}
	\]
	for each $k \ge 1$.
\end{Cor}

\begin{proof}
	First note that if $a \in \ideal{b}$ in a tensor-triangulated category, then a standard thick subcategory argument implies that $a \otimes a \in \ideal{b \otimes b}$. We then prove the claim by induction. The base $k=1$ case is \cref{prop:tensast}. For the inductive step, observe that
	\[
		(R/x^f)^{\otimes 2^{k+1}} = 
		(R/x^f)^{\otimes 2^k}
		\otimes
		(R/x^f)^{\otimes 2^k}
		\in \ideal{R/x^{f^{\ast 2^k}}\otimes R/x^{f^{\ast 2^k}}} \subseteq \ideal{R/x^{f^{\ast 2^{k+1}}}}
	\]
	by the inductive hypothesis and another application of \cref{prop:tensast}.
\end{proof}

\begin{Thm}\label{thm:rid}
	Let $f$ and $g$ be monotonic sequences. We have $R/x^f \in \radideal{R/x^g}$ if and only if $f \le_\mu g$.
\end{Thm}

\begin{proof}
	The claim is trivial if $f=\underline{0}$ so we may assume $f \neq \underline{0}$. If we then define~$f'$ by $f'(n) \coloneqq f(n+\inf(f))$ then $R/x^{f'} = R/x^f[-\inf(f)]$ and $f' \sim_\sigma f$. In particular, $f' \sim_\mu f$. Thus, replacing $f$ by $f'$ we may assume without loss of generality that $f(0) \neq 0$. Hence $\loew(R/x^f)=f$.

	$(\Rightarrow)$: Suppose $R/x^f \in \radideal{R/x^g}$. Thus $(R/x^f)^{\otimes m} \in \ideal{R/x^g}$ for some $m \ge 1$. It follows that $(R/x^f)^{\otimes 2^k} \in \ideal{R/x^g}$ for some $k \ge 1$. By an inductive application of \Cref{lem:loew-conv}, we have $\loew_{(R/x^f)^{\otimes 2^k}} = (\loew_{R/x^f})^{\ast 2^k} = f^{\ast 2^k}$. In particular, this Loewy sequence is monotonic. \Cref{thm:loew-thick} then implies that $\loew_{(R/x^f)^{\otimes 2^k}} \le_{\sigma} g$, that is, $\smash{f^{\ast 2^k}} \le_\sigma g$. Hence $\smash{f^{\ast 2^k}} \le_\mu g$. On the other hand, \cref{lem:mufast} implies that $h \le_\mu h \ast h$ for any monotonic sequence $h$. Thus, $f \le_\mu f^{\ast 2^k}$ so that $f \le_\mu g$.

	$(\Leftarrow)$: If $f \le_\mu g$ then $f \le \mu^k g$ for some $k \ge 1$. \Cref{lem:fastcons} then implies $f^{\ast 2^k} \le g$. Hence $R/x^{f^{\ast 2^k}} \in \ideal{R/x^g}$ by \cref{cor:thickRxf}. It then follows from \cref{cor:tensastiterated} that $(R/x^f)^{\otimes 2^k} \in \ideal{R/x^g}$.
\end{proof}

\begin{Prop}\label{prop:almost-explodable}
	For any monotonic sequence $f$, we have
	\[
		(R/x^f)^{\oplus a} \otimes (R/x^f)^{\oplus a} \in \ideal{R/x^{f}}
	\]
	for any function $a:\bbN\to\bbN$.
\end{Prop}

\begin{proof}
	Replacing $a$ by $\mono{a}$, we may assume without loss of generality that $a$ is monotonic. By \cref{lem:almost-explodable3}, it suffices to prove that $E \in \ideal{R/x^f}$ where $E\in \Dzfl(R)$ is the complex with zero differentials defined by
	\[
		E_n = \bigoplus_{0 \le i \le \floor{n/2}} (R/x^{f(i)})^{a(i)a(n-i)}
	\]
	for $n \in \bbN$. To this end, define $b :\bbN \to \bbN$ by
	\[
		b(n) \coloneqq \sum_{j=0}^{2n} a(j)a(2n-j)
	\]
	and observe that
	\begin{equation}\label{eq:blah2}
		H_n( R/x^f \otimes (R^{\uparrow})^{\oplus b}) = \bigoplus_{0 \le i\le n}(R/x^{f(i)})^{\oplus b(n-i)}
	\end{equation}
	where $R^{\uparrow}$ is the complex from \eqref{eq:Ruparrow}. For each  $0 \le i \le \floor{n/2}$, we have
	\[
		b(n-i) \ge \sum_{j=0}^{n} a(j)a(2n-2i-j)\ge \sum_{j=0}^n a(j)a(n-j) \ge a(i)a(n-i).
	\]
	Hence \eqref{eq:blah2} has $E_n$ as a direct summand. It follows that 
	\[
		E \in \thicksub{R/x^f \otimes (R^{\uparrow})^{\oplus b}} \subseteq \ideal{R/x^f}
	\]
	as desired.
\end{proof}

\begin{Cor}\label{cor:EE}
	For any $E \in \Dpfl(R)$, we have $E\otimes E \in \ideal{R/x^{\loew_E}}$.
\end{Cor}

\begin{proof}
	Note that $E[-\inf(E)]\otimes E[-\inf(E)] \cong (E\otimes E)[-2\inf(E)]$ and $\loew_E = \loew_{E[-\inf(E)]}$ by definition. Thus, replacing $E$ by $E[-\inf(E)]$ we may assume that $E \in \Dzfl(R)$ with $H_0(E) \neq 0$. Then $\loew_E(n) = \loew(H_n(E))$ for all $n \in \bbN$.

	Let $f \coloneqq \mono{\loew_E}$. By \cref{prop:ideal-force-mono}, it suffices to prove that $E\otimes E \in \ideal{R/x^f}$. Moreover, by \cref{lem:Rxfexplodable}, we have $E \in \thicksub{(R/x^f)^{\oplus a}}$ where $a(n)$ is the number of indecomposable summands in $H_n(E)$. It follows that 
	\[
		E \otimes E \in \thicksub{(R/x^f)^{\oplus a} \otimes (R/x^f)^{\oplus a}}
	\]
	by a standard thick subcategory argument. Hence $E \otimes E \in \ideal{R/x^f}$ by \cref{prop:almost-explodable}.
\end{proof}

\begin{Thm}\label{thm:rid-full}
	For any monotonic sequence $f$, we have
	\[
		\radideal{R/x^f} = \SET{E \in \Dpfl(R)}{\mono{\loew_E}\le_\mu f}.
	\]
\end{Thm}

\begin{proof}
	If $0\neq E \in \radideal{R/x^f}$ then $E'\coloneqq E[-\inf(E)] \in \radideal{R/x^f}$ and $R/x^{\loew_{E'}}$ is a direct summand of $E'$. Hence $R/x^{\mono{\loew_{E'}}} \in \ideal{E'} \subseteq \radideal{R/x^f}$ by \cref{prop:ideal-force-mono}. Hence $\mono{\loew_E} \coloneqq\mono{\loew_{E'}} \le_\mu f$ by \cref{thm:rid}. Conversely, if $E \in \Dpfl(R)$ satisfies $\mono{\loew_E} \le_\mu f$ then $R/x^{\mono{\loew_E}} \in \radideal{R/x^f}$ by \cref{thm:rid} and $E \otimes E \in \ideal{R/x^{\loew_E}} = \ideal{R/x^{\mono{\loew_E}}}$ by \cref{cor:EE} and \cref{prop:ideal-force-mono}.  Hence $E\otimes E \in \radideal{R/x^f}$
	and therefore $E \in \radideal{R/x^f}$ as well.
\end{proof}

\begin{Cor}\label{cor:rid-fullfull}
	For any $X \in \Dpfl(R)$, we have 
	\[
		\radideal{X} = {\textstyle \radideal{R/x^{\loew_X}}} = \SET{E \in \Dpfl(R)}{\mono{\loew_E}\le_\mu \mono{\loew_X}}.
	\]
\end{Cor}

\begin{proof}
	The statement is true if $X=0$. Otherwise, replacing $X$ by $X[-\inf(X)]$, we may assume without loss of generality that $\loew_X(n) = \loew(H_n(X))$ for all $n \in \bbZ$. Hence, $R/x^{\loew_X}$ is a direct summand of $X$ so that $R/x^{\loew_X} \in \radideal{X}$. On the other hand,
	\[
		\radideal{R/x^{\loew_X}} = \radideal{R/x^{\smash{\mono{\loew_X}}\vphantom{\loew_X}}} = \SET{E \in \Dpfl(R)}{\mono{\loew_E}\le_\mu \mono{\loew_X}}
	\]
	by \cref{prop:ideal-force-mono} and \cref{thm:rid-full}. In particular, $X \in \radideal{R/x^{\loew_X}}$.
\end{proof}

\begin{Cor}\label{cor:rid-meet}
	For any two monotonic sequences $f$ and $g$, we have
	\begin{equation}\label{eq:rid-meet}
		\radideal{R/x^f \otimes R/x^g} = \radideal{R/x^{f\ast g}} =  \radideal{R/x^{f \wedge g}} = \radideal{R/x^f} \cap \radideal{R/x^g\vphantom{R/x^f}}.
	\end{equation}
\end{Cor}

\begin{proof}
	The statement is true if either $f=\underline{0}$ or $g=\underline{0}$. Otherwise, replacing $f$ by $\sigma^{\inf(f)} f$ and $g$ by $\sigma^{\inf(g)} g$ we may assume that $f(0)\neq 0$ and $g(0)\neq 0$. Then $\loew_{R/x^f} = f$ and $\loew_{R/x^g}=g$. \Cref{lem:loew-conv} then asserts that $\loew_{R/x^f \otimes R/x^g} = f \ast g$. Hence the first equality in \eqref{eq:rid-meet} follows from \cref{cor:rid-fullfull}. On the other hand,	$f\ast g$ and $f \wedge g$ are $\mu$-equivalent by \cref{lem:convmu}. Hence the second equality in~\eqref{eq:rid-meet} follows from \cref{thm:rid}. Finally, the last equality also follows from \cref{cor:rid-fullfull} since if $E \in \radideal{R/x^f} \cap \radideal{R/x^g}$ then $\mono{\loew_E} \le_\mu f$ and $\mono{\loew_E} \le_\mu g$ which implies that $\mono{\loew_E} \le_\mu f \wedge g$.
\end{proof}

\begin{Cor}\label{cor:Rxf-radical}
	Let $f$ be a monotonic sequence. The ideal $\ideal{R/x^f}$ is radical if and only if $f$ is $\mu$-stable.
\end{Cor}

\begin{proof}
	$(\Rightarrow)$: Observe that $R/x^{\mu f} \in \radideal{R/x^f}$ by \cref{thm:rid} since ${\mu f \le_{\mu} f}$. Hence $R/x^{\mu f} \in \ideal{R/x^f}$ if the ideal $\ideal{R/x^f}$ is radical. This implies that $\mu f \le_{\sigma} f$ by \cref{cor:thickRxf} which is equivalent to $f$ being $\mu$-stable by \cref{lem:MT1b}.

	$(\Leftarrow)$: If $E \in \radideal{R/x^f}$ then $\mono{\loew_E} \le_{\mu} f$ by \cref{thm:rid-full}. This implies $\mono{\loew_E} \le f$ if $f$ is $\mu$-stable. Hence $E \in \ideal{R/x^f}$ by \cref{cor:summary-ideal-Rxf}.
\end{proof}

\section{The spectrum for a discrete valuation ring}\label{sec:the-spectrum}

We have developed the tools needed to compute the Balmer spectrum of $\Dps(R)$. Recall from \cref{exa:balmer-spectrum} that $\PRad(\Dps(R))$ denotes the bounded distributive lattice of principal radical ideals of $\Dps(R)$.

\begin{Thm}\label{thm:main}
	Let $R$ be a discrete valuation ring. We have an isomorphism of bounded distributive lattices
	\[
		(\AsymSeq/\mu)_+ \xrightarrow{\sim} \PRad(\Dps(R))
	\]
	given by $[f]_{\mu} \mapsto \radideal{R/x^f}$ and $\infty \mapsto \Dps(R)=\radideal{R}$.
\end{Thm}

\begin{proof}
	It follows from \cref{thm:rid} that the map is well-defined, injective and order-preserving. Moreover, recall from \cref{rem:Dpfl-largest} that if $E \in \Dps(R)\setminus \Dpfl(R)$ then $\radideal{E} = \radideal{R}=\Dps(R)$. Thus \cref{cor:rid-fullfull} implies that the map is surjective. It evidently preserves the bottom and top elements. Moreover, it preserves binary joins by \cref{exa:join} and it preserves binary meets by \cref{cor:rid-meet}. Hence, it is an isomorphism of bounded distributive lattices.
\end{proof}

\begin{Cor}\label{cor:main}
	Let $R$ be a discrete valuation ring. We have an isomorphism
	\[
		\Spc(\Dps(R)) \cong \Specplus(\AsymSeq/\mu)^{\vee}.
	\]
\end{Cor}

\begin{proof}
	This follows from \cref{thm:main} by Stone duality (\cref{rem:stone-duality}) bearing in mind \cref{exa:balmer-spectrum}.
\end{proof}

\begin{Exa}
	Recall that the points of $\Specplus(\AsymSeq/\mu)^{\vee}$ are the prime ideals of the lattice $(\AsymSeq/\mu)_+$. There is a unique largest prime ideal, namely the collection of all asymptotic sequences $\AsymSeq/\mu \subsetneq (\AsymSeq/\mu)_+$. Bearing in mind that the specialization order is flipped under Hochster duality, this largest prime ideal is the unique generic point of $\Specplus(\AsymSeq/\mu)^{\vee}$.
\end{Exa}

\begin{Rem}
	In general, under the isomorphism of \cref{cor:main}, a prime ideal $\cat P \in \Spc(\Dps(R))$ corresponds to the prime ideal $\SET{[f]_\mu}{R/x^f \in \cat P} \in \Specplus(\AsymSeq/\mu)^{\vee}$. For example, we see that the generic point $\tame(\eta)=\Dpfl(R)$ indeed corresponds to~$\AsymSeq/\mu$.
\end{Rem}

\begin{Prop}\label{prop:prime-is-prime}
	Let $f$ be a monotonic sequence. Then $\radideal{R/x^f}$ is a prime ideal of $\Dps(R)$ if and only if $[f]_\mu$ is a prime element of the lattice $(\AsymSeq/\mu)_+$.
\end{Prop}

\begin{proof}
	In general, a principal radical ideal $\radideal{a} \in \PRad(\cat K)$ is a prime ideal of $\cat K$ if and only if $\radideal{a}$ is a prime element of the lattice $\PRad(\cat K)$. Thus, the claim follows from \cref{thm:main}.
\end{proof}

\begin{Exa}\label{exa:zero-prime}
	The zero sequence $[\underline{0}]_\mu$ is evidently a prime in $(\AsymSeq/\mu)_+$ and hence, since it is the bottom element, the singleton $\{[\underline{0}]_\mu\}$ is a prime ideal. It is the smallest prime ideal and the unique closed point in $\Specplus(\AsymSeq/\mu)^{\vee}$. Of course, we evidently see that it corresponds to the closed point $\tame(\frakm)=(0)$ of $\Spc(\Dps(R))$.
\end{Exa}

\begin{Prop}\label{prop:bounded-is-prime}
	The bounded sequence $[\underline{1}]_\mu$ is a prime element of $(\AsymSeq/\mu)_+$. Hence $\{ [\underline{0}]_\mu, [\underline{1}]_\mu \}$ is the smallest nonzero prime ideal of $(\AsymSeq/\mu)_+$.
\end{Prop}

\begin{proof}
	The first statement amounts to proving that if $f$ and $g$ are monotonic sequences that are not asymptotically bounded, then $g \wedge h$ is not asymptotically bounded. Indeed, for any positive integers $A$ and $n$, there exists an $n_0 \ge n$ such that $f(n_0) > A$ and there exists an $n_1 \ge n$ such that $g(n_1) > A$. Since $f$ and $g$ are monotonic, we have $(f \wedge g)(m) > A$ for $m=\max(n_0,n_1)$. The second statement follows from the fact that $[\underline{1}]$ is the smallest nonzero asymptotic monotonic sequence (\cref{exa:bounded-seq}) so the principal ideal it generates is $\{ [\underline{0}], [\underline{1}] \}$.
\end{proof}

\begin{Cor}
	The thick ideal $\ideal{R/x} = \SET{E \in \Dpfl(R)}{\loew_E \text{ is bounded}}$ is a prime ideal. It is the smallest nonzero prime ideal of $\Dps(R)$.
\end{Cor}

\begin{proof}
	It follows from \cref{prop:bounded-is-prime} and \cref{prop:prime-is-prime} that the radical ideal $\radideal{R/x^{\underline{1}}}$ is a prime ideal of $\Dps(R)$. This coincides with $\ideal{R/x^{\underline{1}}}$ by \cref{cor:Rxf-radical} since $\underline{1}$ is $\mu$-stable. Moreover, $\ideal{R/x^{\underline{1}}} = \ideal{R/x}$ by \cref{prop:ideal-force-mono}. The equality in the statement follows from \cref{thm:rid-full} bearing in mind \cref{lem:two-defs-equiv}, \cref{rem:useful-bound-relations} and \cref{lem:force-mono}. It is the smallest nonzero prime ideal (in fact, the smallest nonzero thick ideal) by \cref{prop:Dbfl-minimal}.
\end{proof}

\begin{Rem}\label{rem:comp-first}
	In summary, we have the following picture:
	
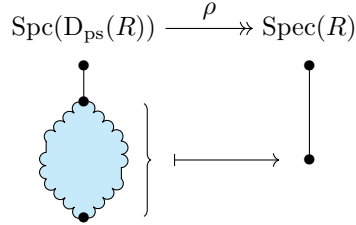
\[
	\begin{tikzpicture}
		\node (B) at (3,2.5) {$\Spc(\Dps(R))$};
		\node (C) at (6,2.5) {$\Spec(R)$};
		\draw[->>] (B) -- (C) node[anchor=south,midway] {$\rho$};
		\filldraw[fill=cyan!20,snake=bumps]
			(3,1.5cm) -- (3.5cm,1.0) -- (3.5cm,0.5cm) -- (3.0,0.0) -- (2.5,0.5) -- (2.5,1.0) -- (3,1.5);
		\node (b) at (3,-.025cm) {$\bullet$};
		\node (b) at (3,1.525cm) {$\bullet$};
		\node (b) at (3,2.0cm) {$\bullet$};
		\draw (3,2.0cm) -- (3,1.5cm);
		\draw[snake=brace] (3.8,1.5)--(3.8,0.0);
		\draw[|->] (4.2,0.75) -- (5.6,0.75);
		\node (b) at (6,0.75cm) {$\bullet$};
		\node (b) at (6,2.0cm) {$\bullet$};
		\draw (6,2.0cm) -- (6,0.75cm);
	\end{tikzpicture}
	\]
	On the left, we have three explicitly specified points, namely (from top to bottom) the zero ideal $(0)=\tame(\frakm)$, the smallest nonzero prime ideal $\ideal{R/x}$ and the largest prime ideal $\Dpfl(R)=\tame(\eta)$.
\end{Rem}

\section{The complexity of asymptotic sequences}\label{sec:complexity}

The lattice of $\mu$-equivalence classes of asymptotic sequences is extremely complicated and correspondingly the cyan region of $\Spec(\Dps(R))$ depicted above is quite complicated. We want to obtain more explicit information about it.

\begin{Rem}
	Recall from \cref{rem:sublattices} that we have an inclusion of distributive lattices $\AsymSeq^\mu \hookrightarrow \AsymSeq/\mu$. Hence we have a surjective spectral map
	\begin{equation}\label{eq:to-mu-stable}
		\Spc(\Dps(R)\cong\Specplus(\AsymSeq/\mu)^\vee \twoheadrightarrow \Specplus(\AsymSeq^\mu)^\vee.
	\end{equation}
	We can gain information about the former space by investigating the latter space. Thus, we want to study the lattice of $\mu$-stable sequences further.
\end{Rem}

\begin{Rem}
	Recall that we have discussed some examples of $\sigma$-stable sequences and $\mu$-stable sequences in \cref{exa:stable}. In particular, polynomials are $\mu$-stable (and hence $\sigma$-stable) while exponential functions are $\sigma$-stable but not $\mu$-stable. Considering these examples, one might wonder whether sequences fail to be $\mu$-stable or $\sigma$-stable simply because they grow too fast. With this in mind we can ask: If $f$ is asymptotically bounded by a $\sigma$-stable (respectively, $\mu$-stable) sequence $g$, does it follow that $f$ is itself $\sigma$-stable (respectively, $\mu$-stable)? Here is a counterexample:
\end{Rem}

\begin{Exa}
	Let $f(n) = \max\{k! \mid k! \le n\}$ and let $g(n) = n$. Then $f \le g$ since $f(n)=g(n)$ if $n=k!$ for some $k$ and $f(n) < g(n)$ if $n \neq k!$ for all $k$. Also, $g$ is $\mu$-stable. However, we claim that $f$ is not $\sigma$-stable (and hence not $\mu$-stable either). In fact, for each positive integer $k$, we have $f(k!-1) = (k-1)!$ and $(\sigma f)(k!-1) = f(k!)=k!$ for all $k$. For any fixed $A$, choosing any $k > A$, we have
	\[
		(\sigma f)(k!-1) = k! > A(k-1)! = Af(k!-1)
	\]
	which shows that $f$ is not $\sigma$-stable. This shows that $\AsymSeq^\sigma$ and $\AsymSeq^\mu$ are not ideals of $\AsymSeq$ since they are not down-sets.
\end{Exa}

\begin{Rem}
	While the lattice $\AsymSeq^\mu$ is complicated, we can generalize the polynomials of \cref{exa:polynomial} to obtain a more manageable sublattice.
\end{Rem}

\begin{Def}[Power sequences]\label{def:power-seq}
	For any non-negative real number $\alpha \in \bbR_{\ge 0}$, define $f_\alpha(n) \coloneqq \floor{n^{\alpha}}$. We call $f_\alpha$ a \emph{power sequence}. It is straightforward to check that power sequences are $\mu$-stable. Also note that $f_0$ is the bounded sequence $\underline{1}$.
\end{Def}

\begin{Rem}
	If we set $g_\alpha(n) \coloneqq \ceil{n^\alpha}$ then clearly $f_\alpha(n) \le g_\alpha(n) \le f_\alpha(n)+1 \le g_\alpha(n)+1$, which implies $f_\alpha \sim g_\alpha$. Thus, choosing the floor (over the ceiling) does not change the nature of the power sequence. 
\end{Rem}

\begin{Rem}
	Let $\alpha,\beta \in \bbR_{\ge 0}$. We have $f_\alpha \le f_\beta$ if and only if $f_\alpha \le_\mu f_\beta$ if and only if $\alpha \le \beta$. In particular, two power sequences $f_\alpha$ and $f_\beta$ are asymptotically equivalent if and only if they are $\mu$-equivalent if and only if $\alpha=\beta$. There is thus a continuum of asymptotic equivalence classes of power sequences, indexed by non-negative real numbers, namely $\SET{[f_\alpha]}{\alpha \in \bbR_{\ge 0}}$. Consequently, $\AsymSeq^\mu$ and hence $\AsymSeq^\sigma$, $\AsymSeq$, $\AsymSeq/\mu$ and $\AsymSeq/\sigma$ all contain uncountably many elements.
\end{Rem}

\begin{Def}
	We write $\PowerSeq$ for the subset of $\AsymSeq^\mu$ consisting of the power sequences together with $[\underline{0}]$. It is a totally ordered sublattice of $\AsymSeq^\mu$. We can thus augment \eqref{eq:to-mu-stable} by considering the surjections
	\[
		\Spc(\Dps(R))\cong\Specplus(\AsymSeq/\mu)^\vee \twoheadrightarrow \Specplus(\AsymSeq^\mu)^\vee \twoheadrightarrow \Specplus(\PowerSeq)^\vee.
	\]
\end{Def}

\begin{Rem}\label{rem:pseq-spec}
	We can describe the spectrum $\Specplus(\PowerSeq)^\vee$ by working through \cref{exa:totally-ordered}. Since the lattice $\PowerSeq_+$ is totally ordered, every element (excluding~$\infty$) is prime. In particular, we have the smallest prime ideal $\mathfrak{o} \coloneqq \langle\underline{0}\rangle =\{\underline{0}\}$ and for each $0 \le \alpha < \infty$ we have the prime ideal $\frakp_\alpha \coloneqq \langle f_\alpha \rangle = \{ \underline{0}\} \cup \SET{f_\beta}{0\le\beta \le \alpha}$, including $\frakp_0 = \langle f_0 \rangle =  \langle \underline{1}\rangle = \{\underline{0}, \underline{1}\}$. In addition, for each $0 < \alpha \le \infty$, we have a prime ideal $\frakq_\alpha \coloneqq \{\underline{0}\} \cup \SET{f_\beta}{\beta < \alpha}$, including the largest prime ideal $\frakq_\infty = \PowerSeq$. These are all the prime ideals, since they are all the nonempty proper down-sets. Since Hochster duality flips the specialization order, we have $\frakp \rightsquigarrow \frakq$ if and only $\frakq \subseteq \frakp$. We may visualize this space as follows:
	\[\begin{tikzpicture}
		\node (a) at (-1.8,0.75) {$\Specplus(\PowerSeq)^\vee =$};
		\draw [pattern color=red,thin,pattern={Lines[angle=170,distance=4pt]}] (0,0.75) ellipse (0.375cm and 0.75cm);
		\draw [semithick] (0,0.75) ellipse (0.375cm and 0.75cm);
		\node (b) at (0,0cm) {$\bullet$};
		\node (b) at (0,1.5cm) {$\bullet$};
		\node (b) at (0,2.0cm) {$\bullet$};
		\draw (0,2.0cm) -- (0,1.5cm);
	\end{tikzpicture}\]
	The top point (the closed point) is $\langle \underline{0}\rangle$, the next point is $\frakp_0 = \langle \underline{1} \rangle$, and the bottom point (the generic point) is $\frakq_\infty=\PowerSeq$. The points on the left-hand branch are the points $\frakp_\alpha$, $0<\alpha<\infty$, going downwards as $\alpha$ increases. Similarly, the points on the right-hand branch are the points $\frakq_\alpha$, $0<\alpha<\infty$, also going downwards as $\alpha$ increases. The red diagonal lines hint at the specialization order. We have 
	\begin{align*}
		\overbar{\{\frakp_\alpha\}} &= \{\mathfrak{o}\}\cup \SET{\frakp_\beta}{0\le \beta \le \alpha}\cup\SET{\frakq_\beta}{0<\beta \le \alpha}, \text{ and}\\
		\overbar{\{\frakq_\alpha\}} &= \{\mathfrak{o}\}\cup \SET{\frakp_\beta}{0\le \beta < \alpha}\cup\SET{\frakq_\beta}{0<\beta \le \alpha}.
	\end{align*}
	In particular, for each $\alpha >0$, we have $\frakp_\alpha \rightsquigarrow \frakq_\alpha$. The diagonal lines in the figure indicate that, in contrast, $\frakq_\alpha \not\rightsquigarrow \frakp_\alpha$ but rather $\frakq_\alpha \rightsquigarrow \frakp_\beta$ for each $\beta < \alpha$. Thus, the depiction does not tell the whole story, since there are also specializations going from the left branch to the right branch. The nonempty Thomason closed subsets are the sets
	\[
		V(\alpha) \coloneqq \{\mathfrak{o}\} \cup \SET{\frakp_{\beta}}{\beta < \alpha} \cup \SET{\frakq_\beta}{\beta \le \alpha}
	\]
	for $0 \le \alpha \le \infty$. We see that the visible points are precisely the $\frakq_{\alpha}$, $0<\alpha\le \infty$. In particular, the point $\langle [\underline{1}]\rangle = \frakp_0$ is not visible. The space is not noetherian.
\end{Rem}

\begin{Rem}
	One can readily check that the spectral map corresponding to the sublattice inclusion $\{\underline{0},\underline{1}\}_+ \subseteq \PowerSeq_+ \subseteq (\AsymSeq/\mu)_+$ can be identified, under the homeomorphism $\Spc(\Dps(R)) \cong \Specplus(\AsymSeq/\mu)^\vee$, with the comparison map $\rho:\Spc(\Dps(R)) \to \Spec(R)$. Recalling \cref{rem:comp-first}, we have
	\[
	\begin{tikzpicture}
		\node (A) at (0,2.5) {$\Spc(\Dps(R))$};
		\node (B) at (3,2.5) {$\Specplus(\PowerSeq)^\vee$};
		\node (C) at (6,2.5) {$\Spec(R)$};
		\draw[->>] (A) -- (B);
		\draw[->>] (B) -- (C);
		\filldraw[fill=cyan!20,snake=bumps]
			(0,1.5cm) -- (0.5cm,1.0) -- (0.5cm,0.5cm) -- (0.0,0.0) -- (-0.5,0.5) -- (-0.5,1.0) -- (0,1.5);
		\node (b) at (0,-.025cm) {$\bullet$};
		\node (b) at (0,1.525cm) {$\bullet$};
		\node (b) at (0,2.0cm) {$\bullet$};
		\draw (0,2.0cm) -- (0,1.5cm);
	%
		\draw [pattern color=red,thin,pattern={Lines[angle=170,distance=4pt]}] (3,0.75) ellipse (0.375cm and 0.75cm);
		\draw [semithick] (3,0.75) ellipse (0.375cm and 0.75cm);
		\node (b) at (3,0cm) {$\bullet$};
		\node (b) at (3,1.5cm) {$\bullet$};
		\node (b) at (3,2.0cm) {$\bullet$};
		\draw (3,2.0cm) -- (3,1.5cm);
		\draw[snake=brace] (3.8,1.5)--(3.8,0.0);
		\draw[|->] (4.2,0.75) -- (5.6,0.75);
		\node (b) at (6,0.75cm) {$\bullet$};
		\node (b) at (6,2.0cm) {$\bullet$};
		\draw (6,2.0cm) -- (6,0.75cm);
	\end{tikzpicture}
	\]
\end{Rem}

\begin{Rem}\label{rem:not-noetherian}
	This discussion proves that $\Spc(\Dps(R))$ for a discrete valuation ring is not noetherian since it surjects onto a non-noetherian space. Moreover, it has at least $2^{\aleph_0}$ many points, while $\Spec(R)$ has only two points. Matsui \cite[Theorem~3.1]{Matsui19} proves that if $\Spc(\Dps(R))$ is noetherian then $\Spec(R)$ is finite. The converse is thus far from true.
\end{Rem}

\begin{Rem}
	Recall from \cref{prop:prime-is-prime} that the principal radical ideal $\radideal{R/x^f}$ is a prime ideal of $\Dps(R)$ if and only if $[f]_\mu$ is a prime element of the lattice $(\AsymSeq/\mu)_+$. We will establish, however, that the lattice of $\mu$-equivalence classes of monotonic sequences has only two prime elements, namely $[\underline{0}]$ and $[\underline{1}]$. It will follow that, although $\Dps(R)$ has at least a continuum many prime ideals, only two of them are finitely generated; see \cref{cor:two-fg} below. The proof of this result is somewhat intricate and will be given in the next section.
\end{Rem}

\section{Technical constructions}\label{sec:technical}

In order to prove that the bounded distributive lattice $(\AsymSeq/\mu)_+$ has few prime elements, we need to construct some bizarre monotonic sequences.

\begin{Cons}\label{cons:general}
	Let $k_0 < k_1 < k_2 < \cdots$ be a strictly increasing sequence of natural numbers. Suppose $f_0$ and $f_1$ are monotonic sequences which satisfy
	\begin{enumerate}
		\item $f_0(k_{i+1}-1) \le f_1(k_{i+1})$ and $f_1(k_{i+1}-1) \le f_0(k_{i+1})$ for each $i \in \bbN$; and
		\item $f_0(k_0-1) \le f_1(k_0)$ if $k_0 > 0$.
	\end{enumerate}
	Note that for each $n \ge k_0$, there is a unique $i \in \bbN$ such that $k_i \le n < k_{i+1}$. For any subset $S\subseteq\bbN$, we can define
	\begin{equation}
		f_S(n) \coloneqq \begin{cases}
			f_1(n) & \text{if } k_i \le n < k_{i+1} \text{ and } i \in S\\
			f_0(n) & \text{if } k_i \le n < k_{i+1} \text{ and } i \not\in S\\
			f_0(n) & \text{if } n < k_0
		\end{cases}
	\end{equation}
	The hypotheses ensure that $f_S$ is monotonic for any $S \subseteq \bbN$. Observe that $f_{\emptyset} = f_0$ while $f_{\bbN}(n) = f_1(n)$ for $n \ge k_0$. The construction is simplest when $k_0=0$ but we will need the extra flexibility which allows the indexing sequence to start at $k_0 > 0$.
\end{Cons}

\begin{Lem}\label{lem:f-intersection}
	In the context of \cref{cons:general}, suppose that $f_0 \le f_1$. Then
	\begin{enumerate}
		\item $f_{S_1} \wedge f_{S_2} \sim f_{S_1 \cap S_2}$ and
		\item $f_{S_1} \vee f_{S_2} \sim f_{S_1 \cup S_2}$
	\end{enumerate}
	for any two subsets $S_1,S_2 \subseteq \bbN$.
\end{Lem}

\begin{proof}
	By hypothesis, there are positive integers $A$ and $N$ such that $f_0(n) \le A f_1(n)$ for all $n \ge N$. Without loss of generality we may take $N \ge k_0$. We will establish that 
	\[
		(f_{S_1} \wedge f_{S_2})(n) \le f_{S_1 \cap S_2}(n) \le A(f_{S_1} \wedge f_{S_2})(n)
	\]
	and 
	\[
		(f_{S_1 \cup S_2})(n) \le (f_{S_1} \vee f_{S_2})(n) \le A(f_{S_1 \cup S_2})(n)
	\]
	for all $n \ge N$. For any such $n$, there is a unique $i$ such that $k_i \le n < k_{i+1}$. If $i \in S_1 \cap S_2$ then $f_{S_1 \cap S_2}(n) = f_{S_1}(n)=f_{S_2}(n) = f_1(n)$ while if $i \not\in S_1 \cup S_2$ then $f_{S_1 \cup S_2}(n) = f_{S_1}(n) = f_{S_2}(n) = f_0(n)$. In both of these cases,
	\[
		f_{S_1 \cap S_2}(n) = (f_{S_1} \wedge f_{S_2})(n) = (f_{S_1}\vee f_{S_2})(n) = f_{S_1 \cup S_2}(n).
	\]
	It thus remains to consider the case where $i$ belongs to precisely one of the subsets. Suppose $i \in S_1 \setminus S_2$ or $i \in S_2 \setminus S_1$. Then $(f_{S_1} \wedge f_{S_2})(n) = \min(f_1(n),f_0(n))$ and $(f_{S_1} \vee f_{S_2})(n) = \max(f_1(n),f_0(n))$. If $f_0(n) \le f_1(n)$ then
	\[
		(f_{S_1} \wedge f_{S_2})(n) = f_0(n) = f_{S_1 \cap S_2}(n)
		\;\text{ and }\;
		f_{S_1 \cup S_2}(n) = f_1(n) = (f_{S_1} \vee f_{S_2})(n).
	\]
	Otherwise, if $f_1(n) < f_0(n)$ then
	\[
		(f_{S_1} \wedge f_{S_2})(n) = f_1(n) < f_0(n) = f_{S_1 \cap S_2}(n)
		 = f_0(n) \le A f_1(n) = A (f_{S_1} \wedge f_{S_2})(n)
	\]
	and
	\[
		f_{S_1 \cup S_2}(n) = f_1(n) < f_0(n) = (f_{S_1} \vee f_{S_2})(n)
		= f_0(n) \le A f_1(n) = A f_{S_1 \cup S_2}(n)
	\]
	which finishes the proof.
\end{proof}

\begin{Cons}\label{cons:new-together}
	Let $f_0$ be an unbounded monotonic sequence. We construct a strictly increasing sequence $k_0 < k_1 < \cdots$ as follows. Let $k_0 \coloneqq \min\{n  \mid f_0(n) \neq 0\}$ and inductively define
	\[
		k_{i+1} \coloneqq \min\{ n \mid f_0(n) > 2^i f_0(k_i)\}.
	\]
	Now define $f_1 :\bbN\to\bbN$ by 
	\[
		f_1(n) \coloneqq \begin{cases}
			2^i f_0(k_i) & \text{if $k_i \le n < k_{i+1}$ for some (unique) $i \ge 1$, and}\\
			0 & \text{if $n < k_0$}.
		\end{cases}
	\]
	This sequence is monotonic and constructed so that $f_0 \le f_1$ and $f_1 \not\le f_0$. Moreover, the sequences $f_0$ and $f_1$ satisfy the hypotheses of \cref{cons:general}. Thus, we have a well-defined monotonic sequence $f_S$ for every $S \subseteq \bbN$.
\end{Cons}

\begin{Lem}\label{lem:new-bdded-finite}
	In the situation of \cref{cons:new-together}, the sequence $f_S$ is asymptotically bounded by $f_0$ if and only if the subset $S \subseteq \bbN$ is finite.
\end{Lem}

\begin{proof}
	If $S$ is finite then $f_S(n) = f_0(n)$ for $n \gg 0$ and hence $f_S \le f_0$. Now suppose that $S$ is infinite. For any positive integers $A$ and $N$, we may choose $i \in S$ large enough so that $k_i > N$ and $2^i > A$. Then $f_S(k_i)=f_1(k_i) = 2^i f_0(k_i) > Af_0(k_i)$. This establishes that $f_S \not\leq f_0$.
\end{proof}

\begin{Thm}\label{thm:not-prime-not-mu}
	Let $f$ be an unbounded monotonic sequence. Its asymptotic equivalence class $[f]$ is not a prime element in the lattice $\AsymSeq_+$.
\end{Thm}

\begin{proof}
	Using \cref{cons:general} and \cref{cons:new-together} above for $f_0 \coloneqq f$, we can consider $f_{\mathrm{even}} \coloneqq f_{S_1}$ where $S_1 = \bbN_{\mathrm{even}}$ is the subset of even natural numbers and $f_{\mathrm{odd}} \coloneqq f_{S_2}$ where $S_2 = \bbN_{\mathrm{odd}}$ is the subset of natural numbers. By \cref{lem:new-bdded-finite}, $f_{\mathrm{even}} \not\le f$ and $f_{\mathrm{odd}} \not\le f$. However, by \cref{lem:f-intersection}, $f_{\mathrm{even}}\wedge f_{\mathrm{odd}} \sim f_\emptyset = f$. In other words, $[f_{\mathrm{even}}] \wedge [f_\mathrm{odd}] \le [f]$ and yet $[f_{\mathrm{even}}] \not\le [f]$ and $[f_{\mathrm{odd}}] \not\le [f]$. That is, $[f]$ is not a prime element.
\end{proof}

To prove the analogous statement for $(\AsymSeq/\mu)_+$ we make a slight modification:

\begin{Cons}\label{cons:new-together-mu}
	Let $f_0$ be an unbounded monotonic sequence. We construct a strictly increasing sequence $k_0 < k_1 < \cdots$ as follows. Let $k_0 \coloneqq \min\{n  \mid f_0(n) \neq 0\}$ and inductively define
	\[
		k_{i+1} \coloneqq \min\{ n \mid f_0(n) > 2^i f_0(2^i k_i)\}.
	\]
	Now define $f_1 :\bbN\to\bbN$ by 
	\[
		f_1(n) \coloneqq \begin{cases}
		2^i f_0(2^i k_i) & \text{if $k_i \le n < k_{i+1}$ for some (unique) $i \ge 1$, and}\\
		0 & \text{if $n < k_0$}.
		\end{cases}
	\]
	This sequence is monotonic and constructed so that $f_0 \le f_1$ and $f_1 \not\le_\mu f_0$. Moreover, the sequences $f_0$ and $f_1$ satisfy the hypotheses of \cref{cons:general}. Thus, we have a well-defined monotonic sequence $f_S$ for every $S \subseteq \bbN$.
\end{Cons}

\begin{Lem}\label{lem:new-bdded-finite-mu}
	In the situation of \cref{cons:new-together-mu}, we have $f_S \le_\mu f_0$ if and only if the subset $S \subseteq \bbN$ is finite.
\end{Lem}

\begin{proof}
	As before, if $S$ is finite then $f_S(n) = f_0(n)$ for $n \gg 0$ and hence $f_S \le f_0$. Now suppose that $S$ is infinite. We claim that $f_S \not\le_{\mu} f_0$. In other words, $f_S \not\le \mu^k\hspace{-0.30ex}f_0$ for every $k \ge 1$. For any positive integers $k$, $A$ and $N$, since $S$ is infinite, we may choose $i \in S$ large enough so that $k_i > N$, $2^i >A$ and $i>k$. Then 
	\[
		f_S(k_i)=f_1(k_i) = 2^i f_0(2^i k_i) > Af_0(2^k k_i) = A\mu^k \hspace{-.30ex}f_0(k_i).
	\]
	This establishes that $f_S \not\leq \mu^k\hspace{-0.30ex}f_0$.
\end{proof}

\begin{Thm} \label{thm:not-prime-mu}
	Let $f$ be an unbounded monotonic sequence. Its $\mu$-equivalence class~$[f]_{\mu}$ is not a prime element in $(\AsymSeq/\mu)_+$.
\end{Thm}

\begin{proof}
	The proof of \cref{thm:not-prime-not-mu} goes through verbatim using \cref{cons:new-together-mu}, \cref{lem:new-bdded-finite-mu} and $\le_\mu$ instead of \cref{cons:new-together}, \cref{lem:new-bdded-finite} and $\le$.
\end{proof}

\begin{Cor} \label{cor:new-not-prime}
	Let $f$ be an unbounded monotonic sequence. The radical ideal
	\[
		\radideal{R/x^f} = \SET{ E \in \Dpfl(R) }{ \mono{\loew_E} \le f }
	\]
	is not a prime ideal of $\Dps(R)$.
\end{Cor}

\begin{proof}
	If this were a prime ideal then \cref{prop:prime-is-prime} implies that $[f]_\mu$ would be a prime element of the bounded distributive lattice $(\AsymSeq/\mu)_+$ which is not the case by \cref{thm:not-prime-mu}.
\end{proof}

\begin{Cor}\label{cor:new-not-prime-mu}
	Let $f$ be an unbounded $\mu$-stable sequence. The radical ideal
	\begin{align*} 
		\ideal{R/x^f} &= \SET{ E \in \Dpfl(R) }{ \loew(H_n(E)) \text{ is asymptotically bounded by $f$}}
	\end{align*}
	is not a prime ideal of $\Dps(R)$.
\end{Cor}

\begin{proof}
	This is a special case of \cref{cor:new-not-prime}. The ideal $\ideal{R/x^f}$ is radical by \cref{cor:Rxf-radical} and the displayed equality follows from \cref{thm:rid-full} bearing in mind that 
	\begin{flalign*}
		\kern10em	\mono{\loew_E} \le_\mu f &\Longleftrightarrow \loew_E \le_\mu f &\text{(\cref{lem:force-mono})}\\
								 &\Longleftrightarrow \loew_E' \le_\mu f & \text{(\cref{lem:two-defs-equiv})}\\
							 &\Longleftrightarrow \loew_E' \le f & (\text{$f$ is $\mu$-stable})
	\end{flalign*}
	where $\loew_E'(n) = \loew(H_n(E))$ was defined in \cref{rem:loew-pref}.
\end{proof}

\begin{Rem}\label{rem:MT-error}
	It is claimed in \cite[Theorem~F and Theorem~7.11]{MatsuiTakahashi17} that for each integer $c \ge 1$, the subcategory
	\[
		\cat L_c \coloneqq \SET{E \in \Dpfl(R)}{ \loew(H_n(E)) \text{ is asymptotically bounded by } n^{c-1}}
	\]
	is a prime ideal of $\Dps(R)$. This is true for $c=1$ but \cref{cor:new-not-prime-mu} establishes that this is false for $c \ge 2$. These $\cat L_c$ are \emph{not} prime for $c \ge 2$. The error lies in the proof of \cite[Theorem~7.11]{MatsuiTakahashi17}. The offending statement is: ``As $a_e > te^{c-1}$, we must have $a_e > b_{n-e}$, and $b_{n-e} \le tn^{c-1}$ for all $n \ge e$.'' This logic is faulty, however. There is no reason \emph{a priori} why we couldn't have $a_e \le tn^{c-1}$ for some $n > e$ and hence could have $b_{n-e} > tn^{c-1}$ for some $n > e$. Translated to our point of view, their argument amounts to an attempt to prove that $h(n) = n^{c-1}$ is prime; more specifically, that $f\ast g \le h$ implies $f \le h$ or $g\le h$. (Since $h$ is $\mu$-stable, this is equivalent to saying that $[h]_\mu$ is a prime element of $(\AsymSeq/\mu)_+$ bearing in mind \cref{lem:convmu}.) Here is the analogue of the incorrect argument using our notation:
	\begin{quotation}
		{\small
		Suppose $f \ast g \le h$. So there exists $A$ and $n_0$ such that $(f\ast g)(n) \le Ah(n)$ for all $n \ge n_0$. Suppose $f\not\le h$. So there exists $e \ge n_0$ such that $f(e) > Ah(e)$. Now, for any $n \ge e$ we have
		\[
			\min(f(e),g(n-e)) \le (f\ast g)(n) \le Ah(n).
		\]
		Since $f(e) > Ah(e)$ we must have $g(n-e) \le Ah(n)$ for all $n \ge e$. (This previous sentence is unjustified.) Hence $g(n) \le Ah(n+e)$ for $n \gg 0$. That is, $g \le \sigma^e h$. Since $h$ is $\mu$-stable, it is $\sigma$-stable, hence $g \le h$.
		}
	\end{quotation}
	In summary, the argument is faulty and we have in fact proved that the statement is false. Nevertheless, we found Matsui and Takahashi's study of this example very intriguing. It served as the inspiration for the current paper.
\end{Rem}

\begin{Cor}\label{cor:two-fg}
	There are at least $2^{\aleph_0}$ prime ideals in $\Dps(R)$ but only two of these prime ideals are finitely generated.
\end{Cor}

\begin{proof}
	Since $\cat P \in \Spc(\Dps(R))$ is radical, it is finitely generated as a thick ideal if and only if it is finitely generated as a radical ideal. Moreover, being finitely generated is equivalent to being generated by a single object. Thus, if $\cat P$ is finitely generated then it is a principal radical ideal. We proved in \cref{cor:rid-fullfull} that every proper principal radical ideal $\radideal{E}$ of $\Dps(R)$ is of the form $\radideal{R/x^f}$ for some monotonic sequence~$f$. This sequence cannot be unbounded by \cref{cor:new-not-prime}. The only remaining cases are the zero sequence $[\underline{0}]$ and the bounded sequence~$[\underline{1}]$, which do indeed provide finitely generated prime ideals (\cref{exa:zero-prime} and \cref{prop:bounded-is-prime}).
\end{proof}

\newcommand{\etalchar}[1]{$^{#1}$}

\end{document}